\documentclass[11pt]{article}
\usepackage{amssymb,amsmath,amsfonts,mathrsfs,amsthm,epsfig,latexsym,color}
\usepackage{enumitem,geometry}
\usepackage{fourier}
\usepackage{CJKutf8}
\usepackage[all]{xy}

\makeatletter
\newcommand{\subsectionruninhead}{\@startsection{subsection}{2}{0mm}
{-\baselineskip}{-0mm}{\bf\large}}
\newcommand{\subsubsectionruninhead}{\@startsection{subsubsection}{3}{0mm}
{-\baselineskip}{-0mm}{\bf\normalsize}}
\makeatother

\newtheorem*{theorem*}{Theorem}
\newtheorem*{proof*}{Proof}
\newtheorem*{proposition*}{Proposition}
\newtheorem*{corollary*}{Corollary}
\newtheorem*{claim*}{Claim}
\newtheorem*{remark*}{Remark}
\newtheorem{problem}{Problem}
\newtheorem{theorem}{Theorem}[section]
\newtheorem{proposition}{Proposition}[section]
\newtheorem{notation}[proposition]{Notation}
\newtheorem{corollary}[proposition]{Corollary}
\newtheorem{lemma}[proposition]{Lemma}
\newtheorem{claim}[proposition]{Claim}

\theoremstyle{definition}
\newtheorem{definition}[proposition]{Definition}
\theoremstyle{remark}
\newtheorem{remark}[proposition]{Remark}

\numberwithin{equation}{section}

\def\DD{\mathbb{D}} 
\def\NN{\mathbb{N}}  
\def\RR{\mathbb{R}} 
\def\SS{\mathbb{S}} 

\def\ZZ{\mathbb{Z}}
\def\d0{\delta^{(0)}}
\def\H{{\rm Hol}}

\def\e{{\varepsilon}}

\begin{document}
\title{ $C^r$-Closing lemma for partially hyperbolic diffeomorphisms with 1D-center bundle}
\author{Shaobo Gan  ~ and ~ Yi Shi}
\date{\today}
\maketitle
\begin{abstract}
For every $r\in\mathbb{N}_{\geq2}\cup\{\infty\}$, we prove the $C^r$-closing lemma for general and conservative partially hyperbolic diffeomorphisms with one-dimensional center bundle. In particular, it implies periodic points are dense for $C^r$-generic conservative partially hyperbolic diffeomorphisms with one-dimensional center bundle.
\end{abstract}

\tableofcontents

\section{Introduction}\label{sec:introduction}

Let $M$ be a $d$-dimensional $C^{\infty}$ closed Riemannian manifold. For every $r\in\mathbb{N}\cup\{\infty\}$, let ${\rm Diff}^r(M)$ denote the set consisting of all $C^r$-diffeomorphisms on $M$, endowed with the $C^r$-topology. Let $m$ be the Lebesgue measure on $M$, and ${\rm Diff}^r_m(M)$ be the set of all $C^r$-diffeomorphisms on $M$ that preserves $m$, endowed with the $C^r$-topology.

For every $f\in{\rm Diff}^r(M)$,
a point $p\in M$ is a \emph{periodic point} of $f$ if there exists an integer $k>0$, such that $f^k(p)=p$. We denote by ${\rm Per}(f)$ the set of all periodic points of $f$. It is obvious that periodic points have the simplest orbits. 
So finding periodic points is the most fundamental problem in the study of dynamical systems. Unfortunately, there are dynamical systems that do not admit any periodic points. For instance, the irrational rotations on the circle.

On the other hand, a non-wandering point of a diffeomorphism on a closed manifold always exists. A point $x\in M$ is called a \emph{non-wandering point} of $f$, if there exist a sequence of points $y_n$ converging to $x$ and $k_n>0$, such that $f^{k_n}(y_n)$ converge to $x$. Denote $\Omega(f)$ the \emph{non-wandering set} of $f$, which consists of all non-wandering points of $f$. Notice that if $f\in{\rm Diff}^r_m(M)$, then every $x\in M$ is a non-wandering point of $f$ by the Poincar\'e Recurrence Theorem.

For finding periodic points, it is natural to ask whether we can perturb the system to make a non-wandering point periodic. 
This idea was originated from H. Poincar\'e's famous paper ``Les m\'ethodes nouvelles de la m\'ecanique c\'eleste'' \cite{P}. In this milestone work, he wrote

\vskip 1mm

{\it `` Voici un fait que je n’ai pu d\'emontrer rigoureusement, mais qui me parait pourtant tr$\grave{e}$s vraisemblable. \'Etant donn\'ees des \'equations . . . et une solution particuli$\grave{e}$re quelconque de ces \'equations, on peut toujours trouver une solution p\'eriodique (dont la p\'eriode peut, il est vrai, $\hat{e}$tre tr$\grave{e}$s longue), telle que la diff\'erence entre les deux solutions soit aussi petite qu’on le veut, pendant un temps aussi long qu’on le veut. ''
}


\vskip 1mm

We can state this idea in the language of modern dynamical systems.
\begin{definition}
	For every $f\in{\rm Diff}^r(M)$ (or ${\rm Diff}^r_m(M)$) and a point $x\in M$, we say $x$ is $C^r$-\emph{closable} for $f$, if there exists $g\in{\rm Diff}^r(M)$ (or ${\rm Diff}^r_m(M)$), which is arbitrarily $C^r$-close to $f$, such that $x$ is a periodic point of $g$. 
\end{definition}

The following open problem was asked by R. Thom, and considered as one of ``The 18 mathematical problems for the next century'' proposed by S. Smale \cite{S}.

\begin{problem}[{\bf $C^r$-Closing Lemma}]
	Let $r\in\mathbb{N}\cup\{\infty\}$ and $f\in{\rm Diff}^r(M)$. Is every non-wandering point $C^r$-closable for $f$?
\end{problem}

For $r=1$, the $C^1$-closing lemma was proved by C. Pugh \cite{Pu1}. This celebrated work inspired a series of perturbation theorems for dynamics in the $C^1$-topology. C. Pugh and C. Robinson \cite{PR} proved the $C^1$-closing lemma for conservative and Hamiltonian diffeomorphisms. R. Ma\~n\'e \cite{M} proved the $C^1$-ergodic closing lemma, which provided more information about the closing periodic orbits. In the 1990s, an improved version of the $C^1$-closing lemma was proved by S. Hayashi \cite{Ha}, which was called the $C^1$-connecting lemma right now. See also the work of L. Wen and Z. Xia \cite{WX1,WX2}. The ergodic closing lemma and connecting lemma played crucial roles for proving the famous $C^1$-stability conjecture. 
In 2004, C. Bonatti and S. Crovisier \cite{BC} proved the $C^1$-chain connecting lemma, which is the strongest perturbation lemma in the $C^1$-topology.  In particular, their works \cite{BC} and \cite{C} show that every chain recurrent point is $C^1$-closable for generic diffeomorphisms. 
See \cite{C} for a series of applications in $C^1$-generic dynamics.

For $r>1$, the problem turns out to be very delicate and difficult.
The $C^r$-closing lemma for interval endomorphisms was proved by L. S. Young \cite{Y}.
However, C. Gutierrez \cite{G} built an example which  showed that Pugh's local perturbation method for proving $C^1$-closing lemma does not work in the $C^2$ case. M. Herman \cite{H2,H3} constructed an example of Hamiltonian vector field for which $C^r$-closing lemma is false for sufficiently large $r$. Both counter-examples imply that we need some global perturbation techniques for proving $C^r$-closing lemma when $r>1$. Recently, M. Asaoka and K. Irie \cite{AI} proved a $C^{\infty}$-closing lemma for Hamiltonian diffeomorphisms of closed surfaces. We refer surveys \cite{AZ,Pu2} for more history and explanations of various closing lemmas. 

If $f\in {\rm Diff}^r(M)$ exhibits hyperbolicity, i.e. there exists a $Df$-invariant splitting $TM=E^s\oplus E^u$, such that $Df$ uniformly contracts vectors in $E^s$ and expands vectors in $E^u$, then Anosov shadowing lemma \cite{An} shows that every non-wandering point is the limit of a sequence of periodic points for $f$.

We call $f\in {\rm Diff}^r(M)$ \emph{partially hyperbolic}, if there exist a continuous $Df$-invariant splitting $TM=E^s\oplus E^c\oplus E^u$, and an integer $k\in\NN$, such that for every $x\in M$,
$$
\|Df^k|_{E^s_x}\|<\min\left\lbrace 1,m(Df^k|_{E^c_x})\right\rbrace \leq
\max\left\lbrace 1,\|Df^k|_{E^c_x}\|\right\rbrace <m(Df^k|_{E^u_x}).
$$
Here $\|\cdot\|$ is the norm of the linear operator, and $m(\cdot)$ is the conorm: $m(A)=\inf\left\lbrace \|A\cdot u\|:\|u\|=1 \right\rbrace $. The bundle $E^c$ is called the center bundle of $f$.
Let ${\rm PH}^r(M)$ and ${\rm PH}^r_m(M)$ be the sets of all $C^r$ general and conservative partially hyperbolic diffeomorphisms on $M$, respectively.
Both ${\rm PH}^r(M)\subset{\rm Diff}^r(M)$ and ${\rm PH}^r_m(M)\subset {\rm Diff}^r_m(M)$ are an open sets in the $C^r$-topology. It is clear that the set of all $f\in{\rm PH}^r(M)$ (or ${\rm PH}^r_m(M)$) with ${\rm dim}E^c=1$ is an open set in ${\rm Diff}^r(M)$ (or ${\rm Diff}^r_m(M)$).

In the 1990's, C. Pugh and M. Shub proposed the famous Stable Ergodicity Conjecture, which states that ergodic diffeomorphisms are $C^r$-open and dense in ${\rm PH}^r_m(M)$. After that,
the partially hyperbolic system has become one of the main topics of research in dynamical systems.
People have achieved a lot of positive answers for this conjecture, see \cite{FRH,HHU2,BW,ACW}. Moreover, a lot of progresses \cite{Po,HP1,HP2,BFFP} have been made on the classification of 3-dimensional partially hyperbolic diffeomorphisms. Meanwhile, more exciting examples \cite{HHU3,GOH,BPP,BGP,BGHP} have appeared, including non-dynamically coherent partially hyperbolic diffeomorphisms on 3-manifolds \cite{HHU3,BGHP,BFFP1}. All these works make the partially hyperbolic systems an active and fruitful object in the study of dynamical systems. 

Let's mention two important works \cite{HHU2,CPS} on partially hyperbolic diffeomorphisms with one-dimensional center bundle. For every $r\in\mathbb{N}_{\geq2}\cup\{\infty\}$, F. Rodriguez Hertz, J. Rodriguez Hertz and R. Ures \cite{HHU2} proved that accessibility is a $C^r$-dense and open property for conservative partially hyperbolic diffeomorphisms with one-dimensional center bundle. This implies 
Stable Ergodicity Conjecture holds if the center bundle is one-dimensional, see \cite{BW,HHU2}. Recently, S. Crovisier, R. Potrie and M. Sambarino \cite{CPS} proved the set of diffeomorphisms having at most finitely many attractors contains a dense and open subset of the space of $C^1$ partially hyperbolic diffeomorphisms with one-dimensional center bundle.

In this paper, for every $r\in\mathbb{N}_{\geq2}\cup\{\infty\}$, we prove that the $C^r$-closing lemma holds for partially hyperbolic diffeomorphisms with one-dimensional center bundle. In particular, all our results hold for partially hyperbolic diffeomorphisms on 3-manifolds. 

\vskip 3mm 
\noindent{\bf Theorem A.\ } {\em Let $r\in\mathbb{N}_{\geq2}\cup\{\infty\}$ and $f\in{\rm PH}^r(M)$ with one-dimensional center bundle. Every non-wandering point is $C^r$-closable for $f$.}
\vskip 3mm

Since the set of all $f\in{\rm PH}^r(M)$ with one-dimensional center bundle is an open set in ${\rm Diff}^r(M)$, from the classical genericity argument, Theorem A implies that for every $r\in\mathbb{N}_{\geq2}\cup\{\infty\}$, $C^r$-generic $f\in{\rm PH}^r(M)$ satisfies $\Omega(f)=\overline{{\rm Per}(f)}$, if the center bundle of $f$ is one-dimensional.

For the conservative partially hyperbolic systems with one-dimensional center bundle, we also proved the $C^r$-closing lemma. Due to the Poincar\'e Recurrence Theorem, every point in the compact manifold is non-wandering for conservative diffeomorphisms.

\vskip 3mm 
\noindent{\bf Theorem B.\ } {\em Let $r\in\mathbb{N}_{\geq2}\cup\{\infty\}$ and $f\in{\rm PH}^r_m(M)$ with one-dimensional center bundle. Every point is $C^r$-closable for $f$.}
\vskip 3mm

In this paper, we prove the following theorem, which is equivalent to Theorem B by applying the classical genericity argument.

\vskip 3mm 
\noindent{\bf Theorem B'.\ } {\em For every $r\in\mathbb{N}_{\geq2}\cup\{\infty\}$, $C^r$-generic $f\in{\rm PH}^r_m(M)$ with one-dimensional center bundle has dense periodic points.}
\vskip 3mm

We prove Theorem A and Theorem B' in Section \ref{subsec:proof}.
A natural idea to get periodic points for partially hyperbolic diffeomorphisms with one dimensional center foliation is to find some periodic center leaf and push the system along this leaf to obtain periodic point. However, as we mention before, there are robustly non-dynamically coherent partially hyperbolic diffeomorphisms on 3-manifolds \cite{HHU3,BGHP,BFFP1}. Moreover, even though the diffeomorphism is dynamically coherent, the center foliation is not absolutely continuous in general, see \cite{SW}. This forbids us to make $C^r$-perturbations ($r\geq 2$) which preserve the center foliation. Furthermore, if the center bundle is not orientable, or it is orientable but $Df$ reverses the orientation, then constructions of perturbing the diffeomorphism are seriously restricted.

Our main techniques for perturbation is the following theorem.

\begin{theorem}\label{Thm:push}
	For every $r\in\mathbb{N}\cup\{\infty\}$, let $f\in{\rm PH}^r(M)$ with one-dimensional center bundle and $x\in\Omega(f)$. Assume one of the following holds:
	\begin{enumerate}
		\item The center bundle $E^c$ is oriented, and $Df$ preserves the orientation of $E^c$.
		\item There exists a neighborhood $U$ of $x$ such that if $y,f^k(y)\in U$, then $Df^k(y)$ preserves the local orientation of $U$.
	\end{enumerate}
    There exist a $C^\infty$ vector field $X$, a sequence of real numbers $\tau_n\rightarrow 0$, and a sequence of points $p_n\rightarrow x$ as $n\rightarrow\infty$, such that every $p_n$ is a periodic point of $f_n=X_{\tau_n}\circ f$.
    
    In particular, we have
    \begin{enumerate}
    	\item If case 1 holds, then $X$ can be any $C^\infty$ vector field transverse to $E^s\oplus E^u$ everywhere, there exist corresponding $\tau_n\rightarrow0$ and $p_n\rightarrow x$ such that $p_n\in{\rm Per}(X_{\tau_n}\circ f)$.
    	\item If case 2 holds, then $X$ is a $C^\infty$ vector field satisfying
    	   \begin{itemize}
    	   	\item The support of $X$ is contained in $U$, and the angle between $X$ and $E^c$ is sufficiently small everywhere;
    	   	\item $X$ is non-negative with respect to the local orientation of $E^c$ and non-vanishing at $x$.
    	   \end{itemize}
    \end{enumerate}
\end{theorem}

\begin{remark} This theorem is a combination of Theorem \ref{Thm:Global-Perturb} and Theorem \ref{Thm:Local-Perturb}.
	\begin{enumerate}
		\item The first case of Theorem \ref{Thm:push} is proved in Theorem \ref{Thm:Global-Perturb}, which is a global perturbation result. Here the vector field $X$ only needs to be transverse to $E^s\oplus E^u$  everywhere, thus non-vanishing everywhere. But we don't need the angle between $X$ and $E^c$ to be small. This allowed us to apply this theorem to conservative case, see Section \ref{subsec:proof}.
		\item The second case of Theorem \ref{Thm:push} is proved in Theorem  \ref{Thm:Local-Perturb}, which is a local perturbation result. We compose a local ``pushing'' in a small neighborhood around $x$. It can be applied to the case $Df$ reverses the global orientation of $E^c$, or $E^c$ is non-orientable. But we need the angle between $X$ and $E^c$ to be sufficiently small.  It clear that if $f$ satisfies case 1, then it also satisfies the condition in case 2, which we can also apply local perturbations.
	\end{enumerate}
\end{remark}

The main idea for proving Theorem \ref{Thm:push} is the following: firstly, we construct a family of periodic center curves approximating the non-wandering point, see Proposition \ref{prop:center-curve}. Secondly, we lift the dynamics of $f$ on $M$ to a tubular neighborhood of these periodic curves, see Proposition \ref{prop:PH-local-diffeo}. Thirdly, we apply perturbations generated by some vector field transverse to $E^s\oplus E^u$ or even very close to $E^c$. We need the Lipschitz shadowing property of center invariant manifold (Theorem \ref{thm:Lipschitz}) to estimate the influence of this perturbation to make sure the point in periodic center leaf is actually moving forward along center direction, see Proposition \ref{Prop:Moving-forward}. Finally, for the local perturbation case, the idea is basic the same. But we need more delicate estimation about how much the perturbation influence the point moving forward, see Section \ref{subsec:domination} and \ref{subsec:local}.

A main step for proving Theorem \ref{Thm:push} is Proposition \ref{Prop:Moving-forward}, which estimates the influence of this perturbation to make sure the point in periodic center leaf is actually moving forward. We believe it is a useful tool for studying partially hyperbolic diffeomorphism with one-dimensional center bundle. 

The following theorem is a simplified version of Proposition \ref{Prop:Moving-forward}. We refer Section \ref{sec:perturb} for detailed definitions.

\begin{theorem}\label{Thm:moving-forward}
	Let $F^s\oplus F^u$ be a fiber bundle over $\sqcup_{i=0}^{k-1}\RR_i$, and the diffeomorphism $f:F^s\oplus F^u\rightarrow F^s\oplus F^u$ is periodic and normally hyperbolic at $\sqcup_{i=0}^{k-1}\RR_i$. Assume the stable and unstable bundles $E^s,E^u$ are close to $F^s,F^u$ respectively. For every vector field $X$ which is uniformly transverse to $E^s\oplus E^u$ and every $\tau>0$ small enough, there exists $\Delta_\tau>0$ satisfying
	\begin{itemize}
		\item The diffeomorphism $F_\tau=X_\tau\circ f:F^s\oplus F^u\rightarrow F^s\oplus F^u$ admits an $F_\tau$-invariant section $\sigma_\tau:\sqcup_{i=0}^{k-1}\RR_i\rightarrow F^s\oplus F^u$, and $F_\tau$ is normally hyperbolic at $\sqcup_{i=0}^{k-1}\sigma_\tau(\RR_i)$.
		\item For every $t\in\RR_i$, we denote $h_\tau(t)$ be the unique intersecting point of $(F^s\oplus F^u)_t(\delta)$ with $\sigma_\tau(\RR_i)$, then $h_{\tau}:\RR_i\rightarrow\sigma_{\tau}(\RR_i)$ is a diffeomorphism, and
		$$
		F_\tau\circ h_{\tau}(t)~>~h_{\tau}\big( f(t)+\Delta_{\tau}\big),
		\qquad \forall x\in\RR_i,~i=0,\cdots,k-1.
		$$
		Here the order $``>''$ on $\sigma_{\tau}(\RR_{i+1})$ inherits from $\RR_{i+1}$ $(\RR_k=\RR_0)$.
	\end{itemize}
\end{theorem}

\begin{remark}
	This theorem means if we compose a pushing along center direction in the normally hyperbolic curves, then the conjugating point is actually moving forward comparing the original point. This estimation also works for the leaf conjugacy of dynamically coherent and plaque expansive partially hyperbolic diffeomorphisms with one-dimensional center bundle. But we don't need this conclusion in this paper.
\end{remark}

The study of $C^r$-perturbation lemmas and $C^r$-generic properties for general dynamical systems is extremely difficult for $r>1$. Very few results have been achieved in this field. However, ${\rm PH}^r(M)$ and ${\rm PH}^r_m(M)$ are $C^r$-open sets in ${\rm Diff}^r(M)$ and  ${\rm Diff}^r_m(M)$ respectively. It is natural to study  $C^r$-perturbation lemmas and $C^r$-generic properties in ${\rm PH}^r(M)$ and ${\rm PH}^r_m(M)$, which may inspire the study of general dynamics in $C^r$-topology.
There are a lot of further questions for partially hyperbolic diffeomorphisms in $C^r$-topology where $r\in\mathbb{N}_{\geq2}\cup\{\infty\}$. We list some of them below. We hope our work opening a gate for further studies. 

\begin{problem}
	Let $r\in\mathbb{N}_{\geq2}\cup\{\infty\}$ and $f\in{\rm PH}^r(M)$ with one-dimensional center bundle. Whether the following properties hold for $f$?
	\begin{itemize}
		\item The $C^r$-ergodic closing lemma, $C^r$-connecting lemma, and $C^r$-chain connecting lemma.
		\item The $C^r$-structural stability conjecture: $f$ is $C^r$-structurally stable if and only if it is Axiom-A.
		\item For conservative case, are non-uniformly hyperbolic ones $C^r$-dense?
	\end{itemize}
\end{problem}

\vskip2mm

\noindent {\bf Acknowledgements:} We are grateful to C. Bonatti, S. Crovisier, A. da Luz, H. Hu, R. Potrie, X. Wang, L. Wen, J. Yang, J. Zhang, and Y. Zhu for their valuable comments. S.Gan is supported by NSFC 11771025 and 11831001. Y. Shi is supported by NSFC 12071007, 11831001 and 12090015.

\section{Preliminaries and proof of main theorems}
\label{sec:prel-proof}

In this section, we first give some preliminaries of the partially hyperbolic diffeomorphisms. Then we prove Theorem A and Theorem B' based on Theorem \ref{Thm:Local-Perturb}, \ref{Thm:Topo-degree}, \ref{Thm:Global-Perturb} and \ref{Thm:Div-free}. 

\subsection{Preliminaries}
\label{subsec:prel}

In the whole paper, we assume $r\in\mathbb{N}_{\geq2}\cup\{\infty\}$.
Let  $f\in{\rm PH}^r(M)$, and $TM=E^s\oplus E^c\oplus E^u$ be the partially hyperbolic splitting of $f$. We also assume ${\rm dim}E^c=1$ through the whole paper. Denote $s={\rm dim}E^s$, $u={\rm dim}E^u$, and $d={\rm dim}M=s+u+1$. For every $x\in M$ and $v\in T_xM$, we denote 
$$
v=v^s+v^c+v^u\in T_xM=E^s_x\oplus E^c_x\oplus E^u_x
$$
the decomposition associated to the splitting. For $i=s,c,u$, we denote $\pi^i_x:T_xM\rightarrow E^i_x$ the projection $\pi^i_x(v)=v^i$. The following lemma defines the well-known adapted metric on $M$ with respect to $f$. We refer \cite{HPS} for its proof.

\begin{lemma}\label{lem:metric}
	There exists a $C^\infty$ adapted Riemannian metric on $M$, such that for every $x\in M$,
	\begin{enumerate}
		\item $\|Df|_{E^s_x}\|<\min \left\lbrace 1,\|Df|_{E^c_x}\| \right\rbrace \leq \max \left\lbrace 1,\|Df|_{E^c_x}\| \right\rbrace <m(Df|_{E^u_x})$;
		\item The angles of the splitting $T_xM=E^s_x\oplus E^c_x\oplus E^u_x$ satisfy
		\begin{align*}
		&\min_{v^s\in E^s_x,\|v^s\|=1}d_{T_xM}(v^s,E^c_x\oplus E^u_x)>1-10^{-3},
		\qquad
		\min_{v^c\in E^c_x,\|v^c\|=1}d_{T_xM}(v^c,E^s_x\oplus E^u_x)>1-10^{-3},
		\\
		&\min_{v^u\in E^u_x,\|v^u\|=1}d_{T_xM}(v^u,E^s_x\oplus E^c_x)>1-10^{-3}.
		\end{align*}
		\item For every $i=s,c,u$, the projection $\pi^i_x:T_xM\rightarrow E^i_x$ satisfies $\|\pi^i_x\|\leq 2$, i.e. $\|\pi^i_x(v)\|\leq2\|v\|$ for every $v\in T_xM$.
	\end{enumerate}
    Here $\|\cdot\|$ is the norm of a vector in $T_xM$ and $d_{T_xM}(\cdot,\cdot)$ is the distance induced by the Riemannian metric on $T_xM$.  The invariant bundles are made almost mutually orthogonal in the adapted metric.
\end{lemma}

We fix an adapted metric satisfying all properties in Lemma \ref{lem:metric} through all this paper.
Denote $d(\cdot,\cdot)$ the distance induced by this metric. For every $\delta>0$, every $x\in M$ and every subset $K\subset M$, we denote 
$$
B(x,\delta)=\left\lbrace y\in M: d(x,y)\leq\delta \right\rbrace
\qquad {\rm and} \qquad
B(K,\delta)=\left\lbrace y\in M: d(K,y)\leq\delta \right\rbrace.
$$ 
For every $i=s,c,u$, we denote
$E^i_x(\delta)=\left\lbrace v\in E^i_x:~\|v\|\leq\delta \right\rbrace$  and
$$
T_xM(\delta)=E^s_x(\delta)\oplus E^c_x(\delta)\oplus E^u_x(\delta) =\left\lbrace v\in T_xM:~\|v^s\|\leq \delta,~\|v^s\|\leq\delta, ~\|v^u\|\leq\delta \right\rbrace.
$$
Then for $\delta$ small enough, the exponential map $\exp_x:T_xM(\delta)\rightarrow M$ is an imbedding for every $x\in M$. Moreover, it satisfies $B(x,\delta/2)\subset \exp_x\left( T_xM(\delta)\right)$.

The classical stable manifold theorem \cite{HPS} shows that there are two families of $f$-invariant foliations ${\cal F}^s$ and ${\cal F}^u$ tangent to $E^s$ and $E^u$ respectively. For every $x\in M$, we denote ${\cal F}^{s/u}_x$ the leaf containing $x$ and $d_{{\cal F}^{s/u}}(\cdot,\cdot)$ the distance induced by the inherited Riemannian metric on the leaf of ${\cal F}^{s/u}$.  Each leaf of ${\cal F}^{s/u}_x$ is contained in the stable/unstable manifold of $x$, which is a $C^r$-immersed submanifold. 

\begin{definition}
	For every $\delta$  small enough, let ${\cal F}^s_x(\delta)$ be the connected component of 
	$\exp_x\left(T_xM(\delta)\right)\cap {\cal F}^s_x$ 
	which contains $x$;
	and ${\cal F}^u_x(\delta)$ be the connected component of 
	$\exp_x\left(T_xM(\delta)\right)\cap {\cal F}^u_x$ which contains $x$. They are the local strong stable and strong unstable manifolds of $x$, which are both $C^r$ imbedded submanifolds. 
\end{definition}

We fix a constant $0<\lambda<1$ satisfying
\begin{align}\label{const-hyperbolic}
    \max_{x\in M}\left\lbrace 
    \|Df|_{E^s_x}\|,m(Df|_{E^u_x})^{-1}\right\rbrace <\lambda <1.
\end{align}
Notice that the contracting rate of $f$ in $E^s$ is strictly smaller than $\lambda$, and the expanding rate of $f$ in $E^u$ is strictly larger than $\lambda^{-1}$.
The following lemma is a corollary of the stable manifold theorem, see \cite{HPS}.

\begin{lemma}\label{lem:local-stable}
	There exists $\d0=\d0(M,f,\lambda)>0$, such that for every $x\in M$, the exponential map $\exp_x:T_xM(\d0)\rightarrow M$ is a diffeomorphism, and there exist two $C^r$-smooth functions 
	$$
	\varphi^s_x: E^s_x(\d0)\rightarrow E^c_x(\d0)\oplus E^u_x(\d0)
	\qquad {\it and} \qquad 
	\varphi^u_x: E^u_x(\d0)\rightarrow E^s_x(\d0)\oplus E^c_x(\d0),
	$$ 
	such that
	\begin{align*}
	    &{\cal F}^s_x(\d0)=
	    \exp_x\left({\rm Graph}(\varphi^s_x)\right)=
	    \exp_x\left(\left\lbrace v^s+\varphi^s_x(v^s):~v^s\in E^s_x(\d0) \right\rbrace \right); \\
	    &{\cal F}^u_x(\d0)=
	    \exp_x\left({\rm Graph}(\varphi^u_x)\right)=
	    \exp_x\left(\left\lbrace v^u+\varphi^u_x(v^u):~v^u\in E^u_x(\d0) \right\rbrace \right).
	\end{align*}
	Moreover, they satisfy the following properties:
	\begin{enumerate}
		\item For every $0<\delta\leq\d0$, we have 
		    \begin{itemize}
		    	\item $E^s_x(\delta)$ is diffeomorphic to ${\cal F}^s_x(\delta)$ by the map 
		    	$\exp_x\circ\left({\rm Id}+\varphi^s_x\right):
		    	E^s_x(\delta)\longrightarrow {\cal F}^s_x(\delta)$;
		    	\item $E^u_x(\delta)$ is diffeomorphic to ${\cal F}^u_x(\delta)$ by the map 
		    	$\exp_x\circ\left({\rm Id}+\varphi^u_x\right):
		    	E^u_x(\delta)\longrightarrow {\cal F}^u_x(\delta)$.
		    \end{itemize}
	    
		\item These two functions satisfy $\varphi^s_x(0^s_x)=0^{cu}_x$, $\varphi^u_x(0^u_x)=0^{cs}_x$,
		$$ 
		\frac{\partial\varphi^s_x}{\partial s}(0^s_x)=0, \quad
		\|\partial\varphi^s_x/\partial s\|<10^{-3};
		\qquad {\it and} \qquad
		\frac{\partial\varphi^u_x}{\partial u}(0^u_x)=0, \quad
		\|\partial\varphi^u_x/\partial u\|<10^{-3}.
		$$
		
		\item For every $\e>0$, there exists $0<\delta_\e<\d0$, such that for every $x\in M$, we have
		$$
		\left\|\frac{\partial\varphi^s_x}{\partial s}(v^s) \right\|<\e,
		\quad \forall v^s\in E^s_x(\delta_\e);
		\qquad {\it and} \qquad
		\left\|\frac{\partial\varphi^u_x}{\partial u}(v^u) \right\|<\e,
		\quad \forall v^u\in E^u_x(\delta_\e).
		$$
		
		\item
		For every $0<\delta\leq\d0$ and $n\geq0$, we have 
		\begin{align*}
		    f^n\left({\cal F}^s_x(\delta)\right)
		    &~\subseteq~{\cal F}^s_{f^n(x)}(\lambda^n\delta)
		    ~\subseteq~\left\lbrace y\in{\cal F}^s_{f^n(x)}:~
		     d_{{\cal F}^s}\left(f^n(x),y\right)\leq 
		     2\lambda^n\delta \right\rbrace; \\
		    f^{-n}\left({\cal F}^u_x(\delta)\right)
		    &~\subseteq~{\cal F}^u_{f^{-n}(x)}(\lambda^n\delta)
		    ~\subseteq~\left\lbrace y\in{\cal F}^u_{f^{-n}(x)}:~
		    d_{{\cal F}^u}\left(f^{-n}(x),y\right)\leq
		      2\lambda^n\delta \right\rbrace.
		\end{align*}
		
		\item For every $0<\eta<1$, there exists $\delta_{\eta}>0$, such that for every $0<\delta\leq\delta_{\eta}$, we have
		\begin{align*}
		{\cal F}^s_x\left((1+\eta)^{-1}\delta\right)
		&~\subseteq~\left\lbrace y\in{\cal F}^s_x:~
		d_{{\cal F}^s}\left(x,y\right)\leq\delta \right\rbrace 
		~\subseteq~ {\cal F}^s_x\left((1+\eta)\delta\right); \\
		{\cal F}^u_x\left((1+\eta)^{-1}\delta\right)
		&~\subseteq~\left\lbrace y\in{\cal F}^u_x:~
		d_{{\cal F}^u}\left(x,y\right)\leq\delta \right\rbrace 
		~\subseteq~ {\cal F}^u_x\left((1+\eta)\delta\right).
		\end{align*}
	\end{enumerate}
\end{lemma}

\begin{definition}
	Let $\gamma:[a,b]\rightarrow M$ be a $C^1$ immersion. We say $\gamma$ is a \emph{center curve} if $0\neq\gamma'(t)\in E^c(\gamma(t))$ for every $t\in[a,b]$, i.e. $\gamma$ is tangent to $E^c$ everywhere. Notice here the image of $\gamma$ may have self-intersections.
	
	We call $\gamma$ is a \emph{local center curve} if it is a center curve and its image $\gamma([a,b])$ has no self-intersection. For notational simplicity, we also denote 
	$$
	\gamma=\gamma([a,b])\subset M
	$$
	the image a local center curve of $\gamma:[a,b]\rightarrow M$ on $M$, and $|\gamma|$ the length of $\gamma$.
\end{definition}

The following lemma is Proposition 3.4 of \cite{BBI} for absolutely partially hyperbolic diffeomorphisms with one-dimensional center bundle. However, it works for general partially hyperbolic diffeomorphisms with dim$E^c=1$.

\begin{lemma}[\cite{BBI}, Proposition 3.4]\label{lem:cs-manifold}
	Let $f\in{\rm PH}^r(M)$ with one-dimensional center bundle. There exists $0<\delta'\leq\d0$, such that for every $0<\delta\leq\delta'$, and any local center curve $\gamma^c:[a,b]\rightarrow M$ with $|\gamma^c|\leq2\delta'$, the set
	$$
	{\cal F}^{ci}_{\gamma^c}(\delta)=\bigcup_{x\in\gamma^c}{\cal F}^i_x(\delta), \qquad i=s,u,
	$$
	is a $C^1$ imbedded submanifold tangent to the bundle $E^{ci}=E^c\oplus E^i$ everywhere for $i=s,u$.
\end{lemma}

\begin{proof}
	 The proof is the same with absolutely partially hyperbolic case, we give a short sketch. For $i=u$, we take the negative iteration $\gamma_n^c=f^{-n}(\gamma^c)$ which is still a center curve. Now we take a $C^{\infty}$ family of smooth disks 
	 $\big\{{\cal D}^u_x(\delta)\big\}_{x\in M}$ with $\delta$-radius, which is $C^r$-close to ${\cal F}^u_x(\delta)$ for every $x$ (thus transverse to $E^s\oplus E^c$ everywhere), and $C^\infty$ smooth with respect to the base point $x$. From the continuity, the submanifold
	 $$
	 {\cal D}^{cu}_{\gamma_n^c}~=~\bigcup_{x\in\gamma_n^c}{\cal D}^u_x(\delta)
	 $$
	 is $C^1$-smooth and uniformly transverse to $E^s$ everywhere. 
	 
	 From the dominated splitting $E^s\oplus_{\prec}E^{cu}$, the iteration 
	 $f^n\big({\cal D}^{cu}_{\gamma_n^c}\big)$ is a $C^1$-submanifold whose tangent space uniformly converges to $E^{cu}$ as $n\rightarrow+\infty$.
	 Since the $f^n$-iterations of ${\cal D}^u_x(\delta)$ converge to some disk in ${\cal F}^u_{f^n(x)}$ and uniformly expanding as $n\rightarrow+\infty$, the submanifold
	 $$
	 \bigcup_{x\in\gamma^c}f^n\big({\cal D}^u_{f^{-n}(x)}\big)(\delta)
	 ~\subseteq~f^n\big({\cal D}^{cu}_{\gamma_n^c}\big)
	 $$
	 uniformly converges to ${\cal F}^{ci}_{\gamma^c}(\delta)$ as $n\rightarrow+\infty$. Thus ${\cal F}^{ci}_{\gamma^c}(\delta)$ is a $C^1$ imbedded submanifold tangent to the bundle $E^{cu}=E^c\oplus E^u$ everywhere.
\end{proof}

From Lemma \ref{lem:cs-manifold} and the local product structure, we have the following lemma.

\begin{lemma}\label{lem:local-product}
	Let $f\in{\rm PH}^r(M)$ with one-dimensional center bundle. There exist two constants $\delta''\in(0,\delta']$ and $C'>1$, such that for every $x,y\in M$ with $d(x,y)=\delta\leq\delta''/C'$, and every $C^1$ local center curve $\gamma^c$ which centered at $x$ with radius $\delta''$, the following properties hold.
	\begin{enumerate}
		\item There exists a unique point
		$$
		z={\cal F}^u_y(\delta'')\pitchfork{\cal F}^{cs}_{\gamma^c}(\delta'')={\cal F}^u_y(\delta'')\pitchfork\left(\bigcup_{w\in\gamma^c}{\cal F}^s_w(\delta'')\right).
		$$
		\item If we denote $w\in\gamma^c$ the point where $z\in{\cal F}^s_w(\delta'')$, then
		$$
		d_{\gamma^c}(x,w)<C'\cdot\delta, \qquad d_{{\cal F}^s}(z,w)<C'\cdot\delta, \qquad  \it{and} \qquad
		d_{{\cal F}^u}(z,y)<C'\cdot\delta.
		$$
		\item In particular, if $z=w\in\gamma^c\cap{\cal F}^u_y(\delta'')$, then
		$$
		d_{\gamma^c}(x,z)<C'\cdot\delta, \qquad   \it{and} \qquad d_{{\cal F}^{s/u}}(z,y)<C'\cdot\delta.
		$$
	\end{enumerate}
    The same conclusion holds if one exchanges $s$ and $u$.
\end{lemma}

\begin{lemma}\label{lem:tubular}
	Let $\lambda\in(0,1)$ be the constant defined in (\ref{const-hyperbolic}).
	There exist two constant $\rho\in(0,\sqrt{\lambda}-\lambda)$ and $\delta'''\in(0,\d0)$, such that for every $\delta\in(0,\delta''')$, let $D^{cu}\subset M$ be a $(u+1)$-dimensional closed imbedded disk tangent to $E^c\oplus E^u$ everywhere, and let $d_{D^{cu}}(\cdot,\cdot)$ be the distance induced by the inherited Riemannian metric on $D^{cu}$,
    if $x\in D^{cu}$ satisfies 
    $$
    B^{cu}(x,\delta)=\left\lbrace y\in D^{cu}:
      ~d_{D^{cu}}(x,y)\leq\delta \right\rbrace\subseteq {\rm Int}(D^{cu}),
    $$  
    then we have
	$$
	B\left( {\cal F}^s_x(\lambda\delta),\rho\delta \right)
	\subset 
	B^s_{\sqrt{\lambda}\delta}\left( B^{cu}(x,\delta) \right)
	=
	\bigcup_{y\in B^{cu}(x,\delta)}{\cal F}^s_y(\sqrt{\lambda}\delta)
	$$
	The same conclusion holds if we consider $D^{cs}$ tangent to $E^s\oplus E^c$ and ${\cal F}^u_x(\lambda\delta)$.
\end{lemma}

\begin{proof}
	Since for every $x\in M$, the exponential map $\exp_x:T_xM(\d0)\rightarrow M$ is a local diffeomorphism,  there exists a constant $C''>1$, such that for every $x\in M$ and every $v_1,v_2\in T_xM(\d0/2)$, they satisfies
	$$
	\frac{1}{C''}\|v_1-v_2\|<
	d\left(\exp_x(v_1),\exp_x(v_2)\right)
	<C''\|v_1-v_2\|.
	$$ 
	
	Let $\eta\in(0,10^{-3})$ be a constant satisfying $(1+\eta)^2(\lambda+2\eta)<\sqrt{\lambda}$ and the constant
	$$
	\rho=\frac{\eta}{2C''}.
	$$
	There exists a constant $0<\delta_{\eta}<\delta'/100$ satisfying the following properties:
	\begin{enumerate}
		\item For every $y\in M$ and $0<\delta\leq 10\delta_{\eta}$, we have
		$$
		{\cal F}^s_x\left((1+\eta)^{-1}\delta\right)
		~\subseteq~\left\lbrace y\in{\cal F}^s_x:~
		d_{{\cal F}^s}\left(x,y\right)\leq\delta \right\rbrace 
		~\subseteq~ {\cal F}^s_x\left((1+\eta)\delta\right).
		$$
		
		\item For every imbedded disk $D^{cu}$ which satisfying $B^{cu}(x,10\delta_{\eta})\subseteq {\rm Int}(D^{cu})$, there exists a $C^1$-map $\psi^{cu}_x:E^c_x(2\delta_\eta)\oplus E^u_x(2\delta_\eta)\rightarrow E^s_x(2\delta_\eta)$ satisfying
		$$
		\exp_x\left({\rm Graph}(\psi^{cu}_x)\right)=
		B^{cu}(x,10\delta_{\eta})\cap T_xM(2\delta_\eta),\quad
		\psi^{cu}_x(0^{cu}_x)=0^s_x, 
		\quad {\rm and } \quad
		\|\partial \psi^{cu}_x/\partial cu\|<\eta.
		$$
		Moreover, for every $v^{cu}_1,v^{cu}_2\in E^c_x(2\delta_\eta)\oplus E^u_x(2\delta_\eta)$, if we denote $y_i=\exp_x\left(v^{cu}_i+\psi^{cu}_x(v^{cu}_i)\right)$ for $i=1,2$, then 
		$$
		(1+\eta)^{-1}\|v^{cu}_1-v^{cu}_2\|
		< d_{D^{cu}}(y_1,y_2) <
		(1+\eta)\|v^{cu}_1-v^{cu}_2\|.
		$$
		
		\item For every $z\in \exp_x\left(T_xM(2\delta_\eta)\right)$, there exists a $C^r$-map $\varphi^s_{x,y}:E^s_x(2\delta_\eta)\rightarrow E^c_x(4\delta_\eta)\oplus E^u_x(4\delta_\eta)$ satisfying
		$$
		\exp_x\left({\rm Graph}(\varphi^s_{x,y})\right)~=~
		{\cal F}^s_y(10\delta_\eta)~\cap ~
		\left(E^s_x(2\delta_\eta)\oplus E^c_x(4\delta_\eta)\oplus E^u_x(4\delta_\eta)\right).
		$$
		Moreover, it satisfies $\|\partial\varphi^s_{x,y}/\partial s\|<\eta$, and for every $u_1,u_2\in{\rm Graph}(\varphi^s_{x,y})$, if we denote $z_i=\exp_x(u_i)$ and $u^s_i=\pi^s_x(u_i)\in E^s_x(2\delta_\eta)$ for $i=1,2$, then
		$$
		(1+\eta)^{-1}\|u^s_1-u^s_2\|
		< d_{{\cal F}^s}(z_1,z_2) <
		(1+\eta)\|u^s_1-u^s_2\|.
		$$
	\end{enumerate}
	All these properties are deduced from the $C^1$-continuity of $D^{cu}$, and the local stable manifolds vary continuously in $C^r$-topology with respect to the base point.
	
	For every $0<\delta\leq\delta_\eta$, the local stable manifold ${\cal F}^s_x(\lambda\delta)$ satisfies
	$$
	\exp_x^{-1}\left( {\cal F}^s_x(\lambda\delta) \right)
	~\subseteq~
	E^s_x(\lambda\delta)\oplus E^c_x(\eta\delta)\oplus E^u_x(\eta\delta) 
	~\subseteq~
	E^s_x((\lambda+\eta)\delta)\oplus E^c_x(2\eta\delta)\oplus E^u_x(2\eta\delta).
	$$
	This implies $E^s_x((\lambda+\eta)\delta)\oplus E^c_x(2\eta\delta)\oplus E^u_x(2\eta\delta)$ contains a $(\eta\delta/2)$-neighborhood of $\exp_x^{-1}\left( {\cal F}^s_x(\lambda\delta) \right)$ in $T_xM$. 
	
	For every $v\in E^s_x((\lambda+\eta)\delta)\oplus E^c_x(2\eta\delta)\oplus E^u_x(2\eta\delta)$, if we denote $y=\exp_x(v)$, since the map $\varphi^s_{x,y}$ satisfies $\|\partial \varphi^s_{x,y}/\partial s \|<\eta$, its image of   $E^s_x((\lambda+\eta)\delta)$ satisfies
	$$
	\varphi^s_{x,y}\left( E^s_x((\lambda+\eta)\delta) \right)
	~\subseteq~
	E^c_x(4\eta\delta)\oplus E^u_x(4\eta\delta).
	$$
	So there exists a unique point
	$u={\rm Graph}(\varphi^s_{x,y})\cap {\rm Graph}(\psi^{cu}_x)$, i.e. 
	$$
	z=\exp_x(u)={\cal F}^s_y(10\delta_\eta)\cap B^{cu}(x,10\delta_\eta). 
	$$ 
	
	Moreover, for $u^s=\pi^s_x(u)$ and $u^{cu}=\pi^c_x(u)+\pi^u_x(u)$, we have
	$$
	\|u^{cu}\|\leq 8\eta\delta, \qquad {\rm and} \qquad 
	\|u^s\|\leq 8\eta^2\delta.
	$$
	From the second property of $\delta_{\eta}$, we have
	$$
	d_{D^{cu}}(x,z)\leq (1+\eta)\cdot8\eta\delta<\delta,
	\qquad i.e. \quad z\in B^{cu}(x,\delta).
	$$
	From the third property of $\delta_{\eta}$, we have
	$$
	d_{{\cal F}^s}(y,z)\leq
	(1+\eta)\|\pi^s_x(v)-u^s\|
	\leq (1+\eta)\left[ (\lambda+\eta)\delta+8\eta^2\delta \right]
	<(1+\eta)(\lambda+2\eta)\delta.
	$$
	Since we have assumed $(1+\eta)^2(\lambda+2\eta)<\sqrt{\lambda}$, the first property of $\delta_{\eta}$ implies $y\in{\cal F}^s_z(\sqrt{\lambda}\delta)$. This implies
	$\exp_x^{-1}\left( B^s_{\sqrt{\lambda}\delta}\left( B^{cu}(x,\delta) \right) \right)$ contains $(\eta\delta)/2$-neighborhood of $\exp_x^{-1}\left( {\cal F}^s_x(\lambda\delta) \right)$ in $T_xM$. So for $\rho=\eta/2C''$, the set $B^s_{\sqrt{\lambda}\delta}\left( B^{cu}(x,\delta) \right)$ contains a $(\rho\delta)$-neighborhood of ${\cal F}^s_x(\lambda\delta)$.
	So we take $\delta'''=\delta_{\eta}$ which proves the lemma.	
\end{proof}

\begin{notation}\label{notation}
	From now on, we denote 
	$$
	\delta_0=\min\left\lbrace \d0, \delta', \delta'',\delta'''\right\rbrace 
	\qquad {\it and} \qquad
	C_0=\max \left\lbrace C',C'' \right\rbrace .
	$$
	Then Lemma \ref{lem:local-stable}, \ref{lem:cs-manifold}, \ref{lem:local-product} and \ref{lem:tubular} all hold for constants $\delta_0$ and $C_0$. The constant $\lambda$ is defined in (\ref{const-hyperbolic}).
\end{notation}

\subsection{Proof of main theorems}
\label{subsec:proof}

In this subsection, we state some results and deduce Theorem A and Theorem B from them. We want to point out that all our argument also works for $C^1$ partially hyperbolic diffeomorphisms with one-dimensional center.

Let $r\in\mathbb{N}_{\geq2}\cup\{\infty\}$ and $f\in{\rm PH}^r(M)$ with one-dimensional center bundle. Let $\delta_0$ be the constant in Notation \ref{notation}. Since $\exp_x:T_xM(\delta_0)\rightarrow M$ is a local diffeomorphism for every $x\in M$, the $(\delta_0/2)$-neighborhood 
$B(x,\delta_0/2)\subset  \exp_x\left( T_xM(\delta_0)\right)$
is diffeomorphic to a $d$-dimensional ball in $\mathbb{R}^d$. So we can define the \emph{local orientation of $E^c$ around $x$} as a fixed orientation of $E^c|_{B(x,\delta_0/2)}$. It is well defined since $B(x,\delta_0/2)$ is simply connected.

If the point $x\in\Omega(f)$, then there exist $y_n\rightarrow x$ and $f^{k_n}(y_n)\rightarrow x$ with $k_n>0$.
Firstly, we consider the case that $x$ admits a neighborhood such that every orbit return of this neighborhood preserves the local orientation of $E^c$. Let ${\cal X}^r(M)$ be the set of all $C^r$ vector fields on $M$. For every $X\in{\cal X}^r(M)$ and every $\tau\in\RR$, we denote $\|X\|=\max_{z\in M}\|X(z)\|$ and $X_{\tau}$ the time-$\tau$ map of $X$. 

\begin{theorem}\label{Thm:Local-Perturb}
	Let $r\in\mathbb{N}\cup\{\infty\}$ and $f\in{\rm PH}^r(M)$ with one-dimensional center bundle. For every $x\in\Omega(f)$, if there exists a neighborhood $U$ of $x$ with $E^c|_U$ is orientable, such that for every $y\in U$ and $f^k(y)\in U$ with $k>0$, $Df^k(y)$ preserves the orientation of $E^c|_U$, then there exist a vector field $X\in{\cal X}^{\infty}(M)$ whose support is contained in $U$,  a sequence of real numbers $\{\tau_n\}\subset\mathbb{R}$, and a sequence of points $\{p_n\}\subset M$, such that
	\begin{itemize}
		\item the sequence $\tau_n$ converges $0$ and $p_n$ converge to $x$ as $n$ tends to infinity;
		\item every $p_n$ is a periodic point of $f_n=X_{\tau_n}\circ f$.
	\end{itemize}
\end{theorem}

\begin{remark}
	Since the support of $X$ is contained in $U$, $f_n=X_{\tau_n}\circ f$ is a local perturbation of $f$. We will see that $X$ is a $C^\infty$ vector field satisfying
	\begin{itemize}
		\item The support of $X$ is contained in $U$, and the angle between $X$ and $E^c$ is sufficiently small everywhere;
		\item $X$ is non-negative with respect to the local orientation of $E^c$ and $X(x)\neq0$.
	\end{itemize}
	In particular, $X$ could not be divergence-free. So this perturbation is not  conservative. 
\end{remark}

If $x\in\Omega(f)$ admits orbit returns that reverse the local orientation of $x$, then the periodic points exist automatically.

\begin{theorem}\label{Thm:Topo-degree}
	Let $r\in\mathbb{N}\cup\{\infty\}$ and $f\in{\rm PH}^r(M)$ with one-dimensional center bundle. Given $x\in\Omega(f)$, if there exist a sequence of points $y_n\in M$ converging to $x$ and $f^{k_n}(y_n)\in M$ converging to $x$ with $k_n>0$, such that $Df^{k_n}(y_n)$ reverses the local orientation of the center bundle $E^c$ around $x$, then there exists a sequence of periodic points $p_n\in{\rm Per}(f)$ converging to $x$.
\end{theorem}

\begin{remark}
	Notice here we do not make any perturbations. The proof of this theorem relies on calculating the topological degrees of local return maps. Moreover, this theorem works for $C^1$ diffeomorphisms.
\end{remark}

Now we can prove Theorem A.

\begin{proof}[{\bf Proof of Theorem A}]
	Let $r\in\mathbb{N}_{\geq2}\cup\{\infty\}$ and $f\in{\rm PH}^r(M)$ with one-dimensional center bundle. For every $x\in\Omega(f)\setminus{\rm Per}(f)$, there are two possibilities:
	\begin{itemize}
		\item Either there exist $y_n\rightarrow x$, $k_n\rightarrow+\infty$, and $f^{k_n}(y_n)\rightarrow x$ as $n\rightarrow\infty$, such that $Df^{k_n}(y_n)$ reverses the local orientation of the center bundle $E^c$ around $x$;
		\item Or there exists a small neighborhood $U$ of $x$ with $E^c|_U$ is orientable, such that for every $y\in U$, if $f^k(y)\in U$ with $k>0$, then $Df^k(y)$ preserves the orientation of $E^c|_U$. 
	\end{itemize}	
	
	In the first case, Theorem \ref{Thm:Topo-degree} implies there exist $p_n\in{\rm Per}(f)$ such that $p_n\rightarrow x$. In the second case, we can apply Theorem \ref{Thm:Local-Perturb}, which shows there exist $p_n\rightarrow x$ and $f_n\rightarrow f$ in $C^r$-topology, such that $p_n\in{\rm Per}(f_n)$. Composing with a local conjugation which maps $x$ to $p_n$, it proves that $x$ is $C^r$-closable.	
\end{proof}

\vskip2mm

Now we begin to prove the $C^r$-closing lemma for conservative diffeomorphisms. We restrict ourselves in the case $r\in\mathbb{N}_{\geq2}\cup\{\infty\}$ because we need some results of stable ergodicity of conservative partially hyperbolic diffeomorphisms with one-dimensional center bundle.

Let $f\in{\rm PH}^r_m(M)$ with one-dimensional center bundle. It satisfies one of the following conditions:
\begin{enumerate}
	\item either the center bundle $E^c$ is orientable, and $Df$ preserves the orientation of $E^c$;
	\item or the center bundle $E^c$ is orientable, and $Df$ reverses the orientation of $E^c$;
	\item or the center bundle $E^c$ non-orientable.
\end{enumerate}
Each case in this classification is a $C^r$-open set in ${\rm Diff}^r_m(M)$, and they are mutually disjoint. We will prove the $C^r$-closing lemma case by case following this classification.

Recall that for $f\in{\rm PH}^r(M)$ and $x\in M$, the \emph{accessible class} $AC(x)$ of $x$ consists of all points that can be reached from $x$ along a piecewise smooth curve such that each smooth piece is tangent to either $E^s$ or $E^u$ at every point.
We say $f$ \emph{is accessible}, if it has only one accessible class which is equal to $M$. 

It was proved by Dider \cite{D} that accessibility is a $C^1$-open property for partially hyperbolic diffeomorphism with one-dimensional center bundle. Then Hertz-Hertz-Ures \cite{HHU2} proved that accessibility is a $C^r$-dense property for $C^r$ conservative partially hyperbolic diffeomorphisms with one-dimensional center bundle.

 The following lemma was essentially proved in \cite{HHU2}.

\begin{lemma}\label{lem:accessibility}
	Let $f\in{\rm PH}^r(M)$ with one-dimensional center bundle. Then
	\begin{itemize}
		\item either $f$ has an open accessible class;
		\item or the stable and unstable bundles of $f$ are jointly integrable and $f$ is not accessible.
	\end{itemize}
\end{lemma}

\begin{proof}
	Since dim$E^c=1$, Proposition A.4 of \cite{HHU2} shows that for every $x\in M$, its accessible class $AC(x)$ is open is equivalent to $AC(x)\cap\gamma^c(x)$ has non-empty interior, where $\gamma^c(x)$ is any local center curve centered at $x$. From the classical 4-legs argument, if $AC(x)$ is not open, then locally $AC(x)$ only intersects $\gamma^c(x)$ at $x$, which implies
	$E^s\oplus E^u$ is locally integrable at $x$. So $f$ has no open accessible class implies $E^s\oplus E^u$ is integrable everywhere and $f$ is not accessible.
\end{proof}

\begin{lemma}\label{lem:mixing}
	Let $r\in\mathbb{N}_{\geq2}\cup\{\infty\}$ and $f\in{\rm PH}^r_m(M)$ with one-dimensional center bundle, which satisfies one of the following:
	\begin{itemize}
		\item the center bundle $E^c$ is orientable, and $Df$ reverses the orientation of $E^c$;
		\item or the center bundle $E^c$ is non-orientable.
	\end{itemize}
	If $f$ is accessible, then for every $x\in M$, there exists a sequence of points $y_n\in M$ and integers $k_n>0$, such that
	\begin{itemize}
		\item $y_n\rightarrow x$, and $f^{k_n}(y_n)\rightarrow x$ as $n\rightarrow\infty$;
		\item for every $n$, $Df^{k_n}(y_n)$ reverses the local orientation of $E^c$ around $x$.
	\end{itemize}
\end{lemma}

\begin{proof}
	{\bf Case \uppercase\expandafter{\romannumeral1}.} $E^c$ is orientable, and $Df$ reverses the orientation of $E^c$.
	
	In this case, we can assume the local orientation of $E^c$ around $x$ coincides with the orientation of $E^c$.
	Since $f$ is accessible, it is Kolmogorov with respect to Lebesgue measure(\cite{BW,HHU2}), thus is topologically mixing. Then for every $x\in M$ and any $n\in\NN$, there exists $l_n>0$, such that
	$$
	f^{-k}(B(x,1/n))\cap B(x,1/n)\neq \emptyset, \qquad \forall k>l_n.
	$$
	We take an odd number $k_n>l_n$, and $y_n\in f^{-k_n}(B(x,1/n))\cap B(x,1/n)$. Then $f^{k_n}(y_n)\in B(x,1/n)$, and $Df^{k_n}$ reverses the orientation of $E^c$. Let $n$ tend to infinity, this proves the first case.
	
	\vspace{1mm}
	
	\noindent{\bf Case \uppercase\expandafter{\romannumeral2}.} $E^c$ is non-orientable.
	
	There exists a connected double covering $\widetilde{M}$ of $M$, such that the lifting bundle $\widetilde{E}^c$ is orientable. Denote $\pi:\widetilde{M}\rightarrow M$ the covering map. If we equip $\widetilde{M}$ with the lifting Riemannian metric, then $\pi$ is a local isometry.
	
	Let $T:\widetilde{M}\rightarrow\widetilde{M}$ be the nontrivial Deck transformation associated to this double covering. We have $T^2={\rm id}$, $\pi\circ T=\pi$, and $T$ is an isometry on $\widetilde{M}$. 
	In particular, since $\widetilde{M}$ is the double covering which makes lifting bundle $\widetilde{E}^c$ is orientable, and $T$ is the corresponding Deck transformation, $DT$ preserves $\widetilde{E}^c$ and reverses its orientation.
	
	The diffeomorphism $f$ has two lifting diffeomorphisms $f_+,f_-:\widetilde{M}\rightarrow\widetilde{M}$ satisfying 
	$$
	f_+=f_-\circ T=T\circ f_-.
	$$
	Since the covering map $\pi:\widetilde{M}\rightarrow M$ is a local isometry, $f_+$ and $f_-$ are also partially hyperbolic diffeomorphisms on $\widetilde{M}$ and preserves the Lebesgue measure of $\widetilde{M}$.
	Moreover, since they are commuting with $T$ which is a 2-periodic isometry, $f_+$ and $f_-$ have the same invariant bundles, which are lifting of the invariant bundle of $f$. So they also have the same stable and unstable foliations.
	We denote $\widetilde{{\cal F}}^s$ and $\widetilde{{\cal F}}^u$ the stable and unstable foliations of them respectively.
	
	Since $\widetilde{E}^c$ is the invariant center bundle of $f_+$ and $f_-$, and $f_+=f_-\circ T=T\circ f_-$, where $T$ reverses the orientation of $\widetilde{E}^c$. We can assume that $f_+$ preserves the orientation of $\widetilde{E}^c$ and 
	$f_-$ reverses the orientation of $\widetilde{E}^c$.
	
	\begin{claim}
		$f_-:\widetilde{M}\rightarrow\widetilde{M}$ is accessible, thus topologically mixing.
	\end{claim}
	
	\begin{proof}[Proof of the claim]
		Lemma \ref{lem:accessibility} implies that either $f_-$ has an open accessible class, or the union of stable and unstable bundles of $f_-$ is integrable to a foliation. The second case is impossible, since $\pi$ maps the stable and unstable foliations of $f_-$ into the stable and unstable foliations of $f$, and this contradicts to $f$ is accessible.    	
		
		Now assume $f_-$ is not accessible. Lemma \ref{lem:accessibility} shows that there exists $\widetilde{x}\in\widetilde{M}$, such that its accessible class $AC(\widetilde{x},f_-)$ is open and $AC(\widetilde{x},f_-)\neq\widetilde{M}$. Since the Deck transformation $T$ preserves the stable and unstable foliations of $f_-$, we have
		$$
		AC(T(\widetilde{x}),f_-)=T(AC(\widetilde{x},f_-))
		\neq\widetilde{M},
		$$ 
		which is also open. 
		
		Since $\pi$ maps the stable and unstable foliations of $f_-$ into the stable and unstable foliations of $f$, for $x=\pi(\widetilde{x})$, we have    	
		$$
		M=AC(x,f)=\pi(AC(\widetilde{x},f_-))=
		\pi(AC(T(\widetilde{x}),f_-)).
		$$
		We must have $AC(\widetilde{x},f_-)\cap AC(T(\widetilde{x}),f_-)\neq\emptyset$. Otherwise, we have
		$$
		\widetilde{M}=\pi^{-1}(M)=AC(\widetilde{x},f_-)\cup AC(T(\widetilde{x}),f_-)
		$$
		is the disjoint union of two open sets. This contradicts that $\widetilde{M}$ is connected. 
		
		However, $AC(\widetilde{x},f_-)\cap AC(T(\widetilde{x}),f_-)\neq\emptyset$ implies
		$$
		AC(\widetilde{x},f_-)=AC(T(\widetilde{x}),f_-).
		$$
		If there exists some $\widetilde{y}\in\widetilde{M}\setminus AC(\widetilde{x},f_-)$, then $T(\widetilde{y})\notin AC(\widetilde{x},f_-)$. This implies 
		$$
		y=\pi(\widetilde{y})\notin\pi(AC(\widetilde{x},f_-))=AC(x,f)=M.
		$$
		This is absurd. So we have $AC(\widetilde{x},f_-)=\widetilde{M}$ and  $f_-$ is accessible. Thus $f_-$ is topologically mixing.
	\end{proof}
	
	Now for any $x\in M$, let $\widetilde{x}\in\pi^{-1}(x)\subset\widetilde{M}$.
	From the proof of Case \uppercase\expandafter{\romannumeral1}, there exists a sequence of points $\widetilde{y}_n\in \widetilde{M}$ and integers $k_n>0$, such that
	\begin{itemize}
		\item $\widetilde{y}_n\rightarrow \widetilde{x}$, and $f_-^{k_n}(\widetilde{y}_n)\rightarrow \widetilde{x}$ as $n\rightarrow\infty$;
		\item for every $n$, $Df_-^{k_n}(\widetilde{y}_n)$ reverses the local orientation of $\widetilde{E}^c$ around $\widetilde{x}$.
	\end{itemize}
	
	Since $\pi$ projects the local orientation of $\widetilde{E}^c$ around $\widetilde{x}$ to the local orientation of $E^c$ around $x$. So the sequence of points $y_n=\pi(\widetilde{y}_n)$ and corresponding integers $k_n>0$ satisfy
	\begin{itemize}
		\item $y_n\rightarrow x$, and $f^{k_n}(y_n)\rightarrow x$ as $n\rightarrow\infty$;
		\item for every $n$, $Df^{k_n}(y_n)$ reverses the local orientation of center bundle $E^c$ around $x$.
	\end{itemize}
	This finishes the proof of second case.
\end{proof}

Combined with this lemma, Theorem \ref{Thm:Topo-degree} has the following corollary.

\begin{corollary}\label{Cor:Topo-degree}
	Let $r\in\mathbb{N}_{\geq2}\cup\{\infty\}$ and ${\rm PH}^{r,-}_m(M)$ be the set of all $f\in{\rm PH}^r_m(M)$ with one-dimensional center bundle which satisfies:
	\begin{itemize}
		\item either the center bundle $E^c$ is orientable, and $Df$ reverses the orientation of $E^c$;
		\item or the center bundle $E^c$ is non-orientable.
	\end{itemize}
	There exists a $C^r$-open dense subset ${\cal V}\subseteq {\rm PH}^{r,-}_m$, such that every $f\in{\cal V}$ satisfies ${\rm Per}(f)$ is dense in $M$.
\end{corollary}

When the center bundle $E^c$ is orientable and $Df$ preserves the orientation of $E^c$, we need to do global perturbations.

\begin{definition}
	Let $X\in{\cal X}^r(M)$ be a $C^r$ vector field. We say $X$ \emph{is transverse to} $E^s\oplus E^u$, if the vector $X(x)\in T_xM$ is transverse to $E^s_x\oplus E^u_x$ in $T_xM$ for every $x\in M$. It is clear that $E^c$ is orientable if and only if there exists a continuous vector field $X$ transverse to $E^s\oplus E^u$. 
	
	Assume the center bundle $E^c$ is oriented.
    We denote $\pi^c_x:T_xM\rightarrow E^c_x$ the projection for every $x\in M$. For every $v\in T_xM$, we say \emph{$v$ is positively transverse to $E^s\oplus E^u$ at $x$}, if $\pi^c_x(v)\in E^c_x$ is a non-zero positive vector corresponding to the orientation of $E^c$.
	We say $X\in{\cal X}^r(M)$ is \emph{positively transverse to} $E^s\oplus E^u$, if $X(x)$ is positively transverse to $E^s\oplus E^u$ for every $x\in M$.
\end{definition}

\begin{remark}
	It is clear that if a vector field is (positively) transverse to $E^s\oplus E^u$, then it is non-zero everywhere.
\end{remark}

\begin{theorem}\label{Thm:Global-Perturb}
	Let $r\in\mathbb{N}_{\geq2}\cup\{\infty\}$ and $f\in{\rm PH}^r(M)$ with one-dimensional center bundle. Assume the center bundle $E^c$ is orientable and $Df$ preserves the orientation of $E^c$. Let $X\in{\cal X}^r(M)$ which is transverse to $E^s\oplus E^u$. For every $x\in\Omega(f)$, there exist a sequence of real numbers $\{\tau_n\}$ and $\{p_n\}\subset M$, such that
	\begin{itemize}
		\item the sequence $\tau_n$ converges to $0$ and $p_n$ converge to $x$ as $n$ tends to infinity;
		\item every $p_n$ is a periodic point of $f_n=X_{\tau_n}\circ f$.
	\end{itemize}
\end{theorem}

\begin{remark}
	The previous Theorem provides another proof of Theorem A for the case where $E^c$ is orientable and $Df$ preserves the orientation of $E^c$.
\end{remark}

Since we want to make perturbations for conservative systems, the perturbating vector field $X$ need to be divergence-free.

\begin{theorem}\label{Thm:Div-free}
	Let $r\in\mathbb{N}_{\geq2}\cup\{\infty\}$ and $f\in{\rm PH}^r(M)$ with one-dimensional orientable center bundle $E^c$. If $f$ is topologically mixing, then there exists $X\in{\cal X}^{\infty}(M)$ which is divergence-free and transverse to $E^s\oplus E^u$.
\end{theorem}

Now we can prove Theorem B'.

\begin{proof}[Proof of Theorem B']
	Let $f\in{\rm PH}^r_m(M)$ with one-dimensional center bundle and $x\in M$. It has been showed \cite{HHU2} that there exists a $C^r$-open dense subset ${\cal U}$ of these diffeomorphisms, such that every $f\in{\cal U}$ is accessible. Thus $f$ is Kolmogorov with respect to Lebesgue measure \cite{BW,HHU2} and topologically mixing. 
	
	We denote ${\cal U}_+$ consisting of $f\in{\cal U}$ with $E^c$ is orientable and $Df$ preserves the orientation of $E^c$; ${\cal U}_-$ consisting of $f\in{\cal U}$ with $E^c$ is orientable and $Df$ reverses the orientation of $E^c$ or $E^c$ is non-orientable. Then ${\cal U}={\cal U}_+\cup{\cal U}_-$ are disjoint union of two open subsets.
	
	If $f\in{\cal U}_+$, then Theorem \ref{Thm:Div-free} implies there exists a divergence-free vector field $X\in{\cal X}^{\infty}(M)$ which is transverse to $E^s\oplus E^u$. So we can apply Theorem \ref{Thm:Global-Perturb}. This proves that for every $x\in M$, there exists $f_n\rightarrow f$ in $C^r$-topology and $p_n\rightarrow x$, such that $p_n\in{\rm Per}(f_n)$. The classical generic argument shows that $C^r$-generic $f\in{\cal U}_+$ has dense periodic points on $M$.
	
	If $f\in{\cal U}_-$, then we can apply Corollary \ref{Cor:Topo-degree} directly. This finishes the proof of Theorem B.
\end{proof}

\vskip 2mm

\noindent{\bf Organization of the paper:} In Section \ref{sec:center-curve}, we construct a family of complete periodic center curves approximating the non-wandering point, which is the first step for 
Theorem \ref{Thm:Local-Perturb}, \ref{Thm:Global-Perturb} and \ref{Thm:Div-free}. Section \ref{sec:perturb} is the main part of this paper. We prove Theorem \ref{Thm:Global-Perturb} in Section \ref{subsec:global}, and Theorem \ref{Thm:Local-Perturb} in Section \ref{subsec:local}. Finally, we prove Theorem \ref{Thm:Topo-degree} and Theorem \ref{Thm:Div-free} in Section \ref{sec:non-orientable} and Section \ref{sec:vector-field} respectively.

\section{Complete periodic center curves}\label{sec:center-curve}

In this section, we construct a family of complete periodic center curves approximating the non-wandering point, see Proposition \ref{prop:center-curve}.
This is the first step for proving Theorem \ref{Thm:Local-Perturb},  \ref{Thm:Global-Perturb}, and \ref{Thm:Div-free}.

\begin{proposition}\label{prop:center-curve}
	Let $f\in{\rm PH}^r(M)$ with one-dimensional oriented center bundle $E^c$, and $Df$ preserves the orientation of $E^c$. For every $x\in\Omega(f)\setminus {\rm Per}(f)$, there exists a sequence of complete $C^1$-curves $\theta_n:\RR\rightarrow M$, such that
	\begin{enumerate}
		\item For every $t\in\RR$, $\theta_n'(t)$ is the positive  (w.r.t. the orientation of $E^c$) unit vector in $E^c_{\theta_n(t)}$.
		\item There exist $k_n\rightarrow+\infty$ and $t_n\rightarrow0$, such that $\theta_n(0)=y_n\rightarrow x$, and $\theta_n(t_n)=f^{k_n}(y_n)\rightarrow x$ as $n\rightarrow\infty$.
		\item Let ~$\theta_n^*f^{k_n}:\RR\rightarrow\RR$ be the $C^1$-diffeomorphism defined as 
		$$
		\theta_n^*f^{k_n}(t)=t_n+
		\int_{0}^{t}\|Df^{k_n}|_{E^c_{\theta_n(s)}}\|{\rm d}s.
		$$
		Then we have
		$$
		\theta_n\circ \theta_n^*f^{k_n}=f^{k_n}\circ\theta_n.
		$$
	\end{enumerate}
    Moreover, if $f\in{\rm PH}^r_m(M)$ with one-dimensional oriented center bundle $E^c$ and $Df$ preserves the orientation of $E^c$, then the same conclusion holds for every $x\in M$.
\end{proposition}

\begin{remark}
	The curve $\theta_n:\RR\rightarrow M$ is $k_n$-periodic, and the following diagram commutes:
	\begin{displaymath}
	\xymatrix{
		\RR \ar[r]^{\theta_n^*f^{k_n}} \ar[d]_{\theta_n} & \RR \ar[d]^{\theta_n} \\
		M \ar[r]^{f^{k_n}} & M }
	\end{displaymath}
	Its image $\theta_n(\RR)\subset M$ may have self-intersections. 
\end{remark}

We will prove this proposition at the end of this section. 
The following lemma is the key step for proving Proposition \ref{prop:center-curve}.

\begin{lemma}\label{lem:center-connection}
	Let $f\in{\rm PH}^r(M)$ with one-dimensional center bundle and $x\in\Omega(f)\setminus{\rm Per}(f)$. There exist a sequence of points $y_n\rightarrow x$ and integers $k_n\rightarrow+\infty$ as $n\rightarrow\infty$, such that for every $n$,
	\begin{itemize}
		\item either $f^{k_n}(y_n)=y_n$;
		\item or there exists a $C^1$ curve $\sigma_n:[0,t_n]\rightarrow M$, such that
		$$
		\sigma_n(0)=y_n,  \quad
		\sigma_n(t_n)=f^{k_n}(y_n),  \quad
		\|\sigma_n'(t)\|\equiv1, \quad {\it and} \quad
		\sigma_n'(t)\in E^c_{\sigma_n(t)}, \quad 
		\forall t\in[0,t_n].
		$$
		Moreover, $t_n$ converge to zero as $n\rightarrow\infty$.
	\end{itemize}
\end{lemma}

\begin{proof}
	Since $x\in \Omega(f)\setminus {\rm Per}(f)$, there exist a sequence of points  $z_n\rightarrow x$ and integers $k_n\rightarrow+\infty$ as $n\rightarrow\infty$, such that $f^{k_n}(z_n)\rightarrow x$. 
	Let $\delta_0>0$ be the constant in Notation \ref{notation}.
	For every $n$, we take a $C^1$ curve $\rho_n:[-\delta_0,\delta_0]\rightarrow M$ satisfying
	\begin{itemize}
		\item $\rho_n(0)=z_n$;
		\item $\rho_n'(t)\in E^c_{\rho_n(t)}$ is a unit vector for every $t\in[-\delta_0,\delta_0]$.
	\end{itemize}
	For $\rho_n=\rho_n([-\delta_0,\delta_0])\subset M$,
    Lemma \ref{lem:cs-manifold} shows that 
	$$
	{\cal F}^{cs}_{\rho_n}(\delta_0)=\bigcup_{z\in\rho_n}{\cal F}^s_{z}(\delta_0)
	$$
	is a $C^1$ submanifold tangent to $E^s\oplus E^c$ everywhere.
	
	From the uniformly contracting of ${\cal F}^s$ and $\lim_{n\rightarrow\infty}k_n=+\infty$, the diameter of $f^{k_n}({\cal F}^s_{z_n}(\delta_0))$ satisfies
	$$
	\lim_{n\rightarrow\infty}{\rm diam}
	\left(f^{k_n}({\cal F}^s_{z_n}(\delta_0))\right)=0.
	$$
	Since $\lim_{n\rightarrow\infty}z_n=\lim_{n\rightarrow\infty}f^{k_n}(z_n)=x$, for $n$ large enough and every point $z\in {\cal F}^s_{z_n}(\delta_0)$, there exists a unique point
	$$
	z'\in{\cal F}^u_{f^{k_n}(z)}(\delta_0)\cap {\cal F}^{cs}_{\rho_n}(\delta_0).
	$$
	This defines a continuous map $z\mapsto z'$ from ${\cal F}^s_{z_n}(\delta_0)$ to ${\cal F}^{cs}_{\rho_n}(\delta_0)$. Moreover, there exist $\epsilon_n\rightarrow0$ as $n\rightarrow\infty$, such that
	$$
	z'\in B(z_n,\epsilon_n)\cap{\cal F}^{cs}_{\rho_n}(\delta_0), \qquad \forall z\in{\cal F}^s_{z_n}(\delta_0).
	$$
	
	\begin{claim}\label{claim:c-connection}
		For every $n$ large enough, there exist $w_n\in {\cal F}^s_{z_n}(\delta_0)$ and $w_n'\in{\cal F}^u_{f^{k_n}(w_n)}(\delta_0)\cap {\cal F}^{cs}_{\rho_n}(\delta_0)$, such that
		\begin{itemize}
			\item either $w_n=w_n'\in{\cal F}^s_{f^{k_n}(w_n)}(\delta_0)$.
			\item or there exists a $C^1$ curve $\tau_n$ contained in ${\cal F}^{cs}_{\rho_n}(\delta_0)$ with endpoints $w_n$ and $w_n'$, such that $\tau_n$ is tangent to $E^c$ and its length $|\tau_n|\rightarrow0$ as $n\rightarrow\infty$.
		\end{itemize}
	\end{claim}
	
	\begin{proof}[Proof of the claim]
		
		For every $n$ large enough, we can define a sequence of $C^1$ one-dimensional line field $E^c_{n,k}$ on the $C^1$ submanifold ${\cal F}^{cs}_{\rho_n}(\delta_0)$, such that
		\begin{itemize}
			\item $E^c_{n,k}(y)\subset E^s_y\oplus E^c_y=T_y{\cal F}^{cs}_{\rho_n}(\delta_0)$, for every $y\in{\cal F}^{cs}_{\rho_n}(\delta_0)$;
			\item $E^c_{n,k}(y)$ uniformly converges to $E^c_y$ on ${\cal F}^{cs}_{\rho_n}(\delta_0)$ as $k\rightarrow\infty$.
		\end{itemize}
		
		For $n,k$ large enough and every $z\in{\cal F}^s_{z_n}(\delta_0)$, there exists a $C^1$ curve which is tangent to $E^c_{n,k}$ everywhere, such that it has two endpoints
		$$
		z'={\cal F}^u_{f^{k_n}(z)}(\delta_0)\cap {\cal F}^{cs}_{\rho_n}(\delta_0) \qquad {\rm and} \qquad
		h_{n,k}(z)\in{\cal F}^{s}_{z_n}(\delta_0)
		$$
		This curve may be degenerated, i.e. $h_{n,k}(z)=z'$.
		This defines a continuous map $z\mapsto h_{n,k}(z)$ from ${\cal F}^s_{z_n}(\delta_0)$ to itself. 
		
		Since ${\cal F}^s_{z_n}(\delta_0)$ is diffeomorphic to a unit ball of dimension $s={\rm dim}E^s$, there exists at least one point $w_{n,k}\in{\cal F}^s_{z_n}(\delta_0)$, such that
		$$
		w_{n,k}=h_{n,k}(w_{n,k}).
		$$
		
		Let $k$ tend to infinity and take the subsequence, we can assume $w_{n,k}\rightarrow w_n$. The segments connecting $w_{n,k}'$ and $h_{n,k}(w_{n,k})$ also converge to a $C^1$-curve which is tangent to $E^c$ everywhere, and has two endpoints $w_n'$ and $w_n$. 
		
		If the converging segment is degenerated, then we are in the first item of the claim. Otherwise, we denote the converging segment by $\tau_n$. Since $z'\in B(z_n,\epsilon_n)$ for every $z\in{\cal F}^s_{z_n}(\delta_0)$ and $\epsilon_n\rightarrow0$, the local product structure in Lemma \ref{lem:local-product} implies the length of $\tau_n$ will tend to zero.
		This finishes the proof of our claim.	    
	\end{proof}

    Since $\lim_{n\rightarrow\infty}z_n=\lim_{n\rightarrow\infty}f^{k_n}(z_n)=x$, and
    $\lim_{n\rightarrow\infty}{\rm diam}
    \left(f^{k_n}({\cal F}^s_{z_n}(\delta_0))\right)=0$,
	we have 
	$$
	\lim_{n\rightarrow\infty}w_n=\lim_{n\rightarrow\infty}f^{k_n}(w_n)=\lim_{n\rightarrow\infty}w_n'=x,
	$$
	and the distance in the unstable manifold $d_{{\cal F}^u}\left(f^{k_n}(w_n),w_n'\right)\rightarrow0$ as $n\rightarrow\infty$.
	
	\vskip1mm
	
	Assuming the first case of Claim \ref{claim:c-connection} holds, i.e. $w_n=w_n'$. Since $k_n\rightarrow+\infty$ and ${\cal F}^u$ is uniformly contracting by $f^{-1}$, the distance in the unstable manifold $d_{{\cal F}^u}(w_n,f^{k_n}(w_n))\rightarrow0$ as $n\rightarrow\infty$. This implies 
	$$
	f^{-k_n}\left({\cal F}^u_{f^{k_n}(w_n)}(\delta_0)\right)
	~\subseteq~{\cal F}^u_{w_n}(\lambda^{k_n}\delta_0)
	~\subseteq~{\cal F}^u_{f^{k_n}(w_n)}
	\left(\lambda^{k_n}\delta_0+d_{{\cal F}^u}\left(w_n,f^{k_n}(w_n)\right)\right).
	$$
	So there exists a unique periodic point
	$$
	y_n\in
	{\cal F}^u_{f^{k_n}(w_n)}
	\left(\lambda^{k_n}\delta_0+d_{{\cal F}^u}\left(w_n,f^{k_n}(w_n)\right)\right)
	$$
	satisfying $f^{k_n}(y_n)=y_n$. Since $\lim_{n\rightarrow\infty}\lambda^{k_n}\delta_0+d_{{\cal F}^u}\left(w_n,f^{k_n}(w_n)\right)=0$, we have $y_n\rightarrow x$ as $n\rightarrow0$. 
	
	\vskip1.5mm
	
	Otherwise, the second case of Claim \ref{claim:c-connection} holds. We consider the set
	$$
	{\cal F}^{cu}_{\tau_n}(\delta_0)=\bigcup_{y\in\tau_n}{\cal F}^u_{y}(\delta_0).
	$$
	From Lemma \ref{lem:cs-manifold}, it is a $C^1$ submanifold tangent to $E^c\oplus E^u$ everywhere. We replay the argument in the proof of Claim \ref{claim:c-connection}.
	
	Since $d_{{\cal F}^u}\left(f^{k_n}(w_n),w_n'\right)\rightarrow0$ as $n\rightarrow\infty$, we have
	$$
	f^{-k_n}\left({\cal F}^u_{w_n'}(\delta_0)\right)
	\subseteq{\cal F}^u_{w_n}(\lambda^{k_n}\delta_0)
	\subseteq{\cal F}^u_{w_n}(\delta_0).
	$$
	Repeating the analysis in Claim \ref{claim:c-connection}, we consider a sequence $C^1$ line fields $E^c_{n,l}\subset T{\cal F}^{cu}_{\tau_n}(\delta_0)$ on ${\cal F}^{cu}_{\tau_n}(\delta_0)$, which converges to $E^c|_{{\cal F}^{cu}_{\tau_n}(\delta_0)}$ as $l\rightarrow\infty$.
	So for $n,l$ large enough and every $y\in{\cal F}^u_{w_n'}(\delta_0)$, there exists a $C^1$ curve tangent to $E^c_{n,l}$ everywhere and connecting 
	$f^{-k_n}(y)\in f^{-k_n}\left({\cal F}^u_{w_n'}(\delta_0)\right)
	\subseteq{\cal F}^u_{w_n}(\lambda^{k_n}\delta_0)$ 
	to a point $h_{n,l}(y)\in{\cal F}^u_{w_n'}(\delta_0)$. The continuous map $y\mapsto h_{n,l}(y)$ from ${\cal F}^u_{w_n'}(\delta_0)$ to itself has a fixed point $y_{n,l}$.
	
	Let $l\rightarrow\infty$ and assume $y_{n,l}\rightarrow y_n$. 
	The curve tangent to $E^c_{n,l}$ converges to a curve $\sigma_n$ tangent to $E^c$ everywhere and connecting $y_n$ to $f^{k_n}(y_n)$. Moreover, since $d(w_n,w_n')\rightarrow0$ as $n\rightarrow\infty$, the length of the curve $\sigma_n$ also tends to zero as $n\rightarrow\infty$. Since $k_n\rightarrow+\infty$, we also have $d(y_n,w_n')\rightarrow0$, thus
	$$
	\lim_{n\rightarrow\infty}y_n=\lim_{n\rightarrow\infty}f^{k_n}(y_n)=x.
	$$
	This finishes the proof of this lemma.
\end{proof}

The follwoing lemma has been proved by F. Rodriguez Hertz, J. Rodriguez Hertz and R. Ures in \cite{HHU1}.

\begin{theorem}[\cite{HHU1}, Theorem 2]\label{lem:center-mfd}
	Let $f\in{\rm PH}^r(M)$ with one-dimensional center bundle. There exists $K>0$, such that for any periodic point $p$ of $f$ with period $\pi(p)>K$, there exists, through $p$, an $f^{\pi(p)}$-invariant curve tangent to $E^c$ at every point. 
\end{theorem}

The following lemma is a corollary of this theorem.

\begin{lemma}\label{lem:complete-center-mfd}
	Let $f\in{\rm PH}^r(M)$ with one-dimensional oriented center bundle $E^c$, and $Df$ preserves the orientation of $E^c$. There exists $K>0$, such that for any periodic point $p$ of $f$ with period $\pi(p)>K$, there exists an $f^{\pi(p)}$-invariant $C^1$ curve $\xi_+(p):[0,+\infty)\rightarrow M$, such that
	\begin{itemize}
		\item $\xi_+(0)=p$;
		\item $\xi_+'(t)\in E^c_{\xi_+(t)}$ is the positive unit vector for every $t\in[0,+\infty)$.
	\end{itemize}
	Here we say $\xi_+(p)$ is complete in the positive center direction.
	
	Similarly, we have an $f^{\pi(p)}$-invariant $C^1$ curve $\xi_-(p):(-\infty,0]\rightarrow M$, such that
	\begin{itemize}
		\item $\xi_-(0)=p$;
		\item $\xi_-'(t)\in E^c_{\xi_-(t)}$ is the positive unit vector for every $t\in(-\infty,0]$.
	\end{itemize}
	We say $\xi_-(p)$ is complete in the negative center direction. 
	
	If we define $\xi:\RR\rightarrow M$ as
	$$
	\xi(t)=\left\{
	\begin{array}{ll}
	\xi_-(t), \qquad & t\in(-\infty,0], \\
	\xi_+(t), \qquad & t\in[0,+\infty).
	\end{array}
	\right.
	$$ 
	Then $\xi$ is a complete $f^{\pi(p)}$-invariant $C^1$ curve positively tangent to $E^c$ at every point.
\end{lemma}	
	
\begin{proof}
	Here we take the constant $K$ be the same constant in Theorem \ref{lem:center-mfd}. Then for the periodic point $p$ with period $\pi(p)>K$, it has an $f^{\pi(p)}$-invariant center manifold tangent to $E^c$ at every point. Moreover, since $Df$ preserves the orientation of $E^c$, if one removes $p$,  
	then this center manifold has two branches which are both $f^{\pi(p)}$-invariant.
	
	For the branch tangent to the positive direction of $E^c_p$, if it has infinite length, then it is complete in the positive center direction, and we parameterize it by its length, which defines $\xi_+$. 
	
	If the length of positive branch is finite, then it has another endpoint $q$. From $f^{\pi(p)}$-invariant of this branch, $q$ is a periodic point with period $\pi(q)=\pi(p)>K$. We can apply Theorem \ref{lem:center-mfd} to $q$, and consider its positive branch of $f^{\pi(p)}$-invariant center manifold. Taking the union of positive branchs of $f^{\pi(p)}$-invariant center manifold of $p$ and $q$, and repeating this procedure, we will get an $f^{\pi(p)}$-invariant curve starting from $p$, tangent to $E^c$ at every point, and complete in the positive center direction. We parameterize it by its length, which defines $\xi_+$. This completes the proof of $\xi_+$.
	
	The proof for $\xi_-$ is the same, and $\xi$ is the union of $\xi_+$ and $\xi_-$.	
\end{proof}

Now we can prove Proposition \ref{prop:center-curve}.

\begin{proof}[Proof of Proposition \ref{prop:center-curve}]
	Given $x\in\Omega(f)\setminus{\rm Per}(f)$,  Lemma \ref{lem:center-connection} shows that there exist a sequence of points $y_n\rightarrow x$ and integers $k_n\rightarrow\infty$, such that for every $n$,
	\begin{itemize}
		\item either $f^{k_n}(y_n)=y_n$;
		\item or there exists a $C^1$ curve $\sigma_n:[0,t_n]\rightarrow M$, such that
		$$
		\sigma_n(0)=y_n,  \quad
		\sigma_n(t_n)=f^{k_n}(y_n),  \quad
		\|\sigma_n'(t)\|\equiv1, \quad {\it and} \quad
		\sigma_n'(t)\in E^c_{\sigma_n(t)}, \quad 
		\forall t\in[0,t_n].
		$$
		Moreover, $t_n$ converge to zero as $n\rightarrow\infty$.
	\end{itemize}
	
	In the first case, since $k_n\rightarrow\infty$, we apply Lemma \ref{lem:complete-center-mfd} directly to $y_n$, which gives a complete $f^{k_n}$-invariant $C^1$ curve that positively tangent to $E^c$ everywhere. This proves the proposition in this case.
	
	In the second case, we assume $\sigma_n'(t)\in E^c_{\sigma_n(t)}$ is the positive unit vector. The proof for it is negative is the same.  The curve $f^{k_n}\left(\sigma_n([0,t_n])\right)$ is a curve positively tangent to $E^c$ with endpoints $f^{k_n}(y_n)$ and  $f^{2k_n}(y_n)$. Since $Df$ preserves the orientation of $E^c$, the curve
	$$
	\sigma_n\left([0,t_n]\right)\cup 
	f^{k_n}\left(\sigma_n([0,t_n])\right)
	$$
	is a curve positively tangent to $E^c$, and has endpoints $y_n$ and  $f^{2k_n}(y_n)$.
	
	Repeat the iteration and take the union. We consider the curve
	$$
	\bigcup_{l=-\infty}^{+\infty}f^{l\cdot k_n}
	\left(\sigma_n([0,t_n])\right).
	$$
	If the sum of the length
	$$
	\sum_{l=0}^{+\infty}{\rm length}
	\left(f^{l\cdot k_n}(\sigma_n([0,t_n]))\right)=+\infty,
	$$
	then this curve is complete in the positive direction.
	
	Otherwise, we have
	$$
	\sum_{l=0}^{+\infty}{\rm length}
	\left(f^{l\cdot k_n}(\sigma_n([0,t_n]))\right)<+\infty.
	$$
	This implies $f^{l\cdot k_n}\left(\sigma_n([0,t_n])\right)$ will converge to a point $p_n$ as $l\rightarrow\infty$. It is clear that $p_n$ is a periodic point with period $k_n$.
	
	Since $k_n\rightarrow\infty$ as $n\rightarrow\infty$, Lemma \ref{lem:complete-center-mfd} shows that the periodic point $p_n$ has an $f^{k_n}$-invariant $C^1$ curve tangent to $E^c$ everywhere, which is complete in the positive center direction. This gives us a curve tangent to $E^c$ and starting from $y_n$, which is complete in the positive direction. 
	
	\vskip1mm
	
	Similarly, we consider the sum
	$$
	\sum_{l=-\infty}^{-1}{\rm length}
	\left(f^{l\cdot k_n}(\sigma_n([0,t_n]))\right)
	$$
	is finite or not, and repeat this analysis. We have a curve tangent to $E^c$ and starting from $y_n$, which is complete in the negative direction. 
	
	So the union of these two curves is a complete $f^{k_n}$-invariant $C^1$ curve tangent to $E^c$ everywhere. Parameterizing it by its length, and setting the zero point to be $y_n$, we have the $C^1$ curve $\theta_n:\RR\rightarrow M$ satisfies the first and second items of the proposition.  Moreover, since $\theta_n$ is parameterized by its length and $k_n$-periodic, we have
	$$
	f^{k_n}\circ\theta_n(t)=
	\theta_n\left(t_n+
	\int_{0}^{t}\|Df^{k_n}|_{E^c_{\theta_n(s)}}\|{\rm d}s\right),
	\qquad \forall t\in\RR.
	$$
	This finishes the proof for $f\in{\rm PH}^r(M)$ and $x\in\Omega(f)\setminus{\rm Per}(f)$.
	
	\vskip1mm
	
	Finally, let $f\in{\rm PH}^r_m(M)$ with one-dimensional oriented  center bundle, and $Df$ preserves the orientation of $E^c$.
	Since $f$ is volume preserving, every $x\in M$ is a non-wandering point of $f$. If $x$ is not a periodic point, then we can apply the same proof directly. If $x$ is a periodic point, then we take sequence of points $x_k\in {\cal F}^s_x$, such that $x_k\rightarrow x$ as $k\rightarrow\infty$. We apply the previous analysis to $x_k$ getting a sequence of center curves satisfying desired properties. Let $k$ tend to infinity and take the subsequence, we get the desired sequence of periodic center curves approximating to $x$.
\end{proof}

\section{Lifting bundle dynamics and perturbations}\label{sec:perturb}

In this section, we prove Theorem \ref{Thm:Local-Perturb} and Theorem \ref{Thm:Global-Perturb}. We give a sketch of the proof and main ideas of this section:
\begin{itemize}
	\item Firstly, we consider the periodic center curves which approximates the non-wandering point. We lift the dynamics of $f$ on $M$ to the bundle dynamics on the these periodic curves (Lemma \ref{lem:inducing} and Proposition \ref{prop:exp-map}), and show that the lifting bundle dynamics are normally hyperbolic in these curves (Proposition \ref{prop:PH-local-diffeo}). 
	
	\item Secondly, we apply the normally hyperbolic invariant manifold theory of \cite{HPS} to the bundle dynamics (Theorem \ref{thm:HPS} and Theorem \ref{thm:permanence}). Moreover, we show that these periodic center curves have Lipschitz shadowing property under perturbations (Theorem \ref{thm:Lipschitz}). This Lipschitz shadowing property is important for us estimating the influence of perturbations.
	
	\item Thirdly, we use the partial leaf conjugacy to show that the global perturbation actually push the point in periodic center leaf moving forward along center direction. This is Proposition \ref{Prop:Moving-forward}, which is the key observation in this paper, and proves Theorem \ref{Thm:Global-Perturb}.
	
	\item Finally, for the local perturbation case, the idea is basic the same to Proposition \ref{Prop:Moving-forward}. But we need more delicate estimation about how much the perturbation influence the conjugacy point moving, see Section \ref{subsec:domination} and \ref{subsec:local}.   
\end{itemize}

Let $f\in{\rm PH}^r(M)$ with one-dimensional center bundle, and $TM=E^s\oplus E^c\oplus E^u$ be the partially hyperbolic splitting associated to $f$. We will lift $f:M\rightarrow M$ on the tangent bundles of periodic center curves in Section \ref{sec:center-curve}, and make the perturbation and estimation on these bundle dynamics. The main idea originates from the proof of structural stability of Anosov flows and the classical work of Hirsch-Pugh-Shub \cite{HPS}.

Let the Riemannian metric on $M$ be the adapted metric satisfying Lemma \ref{lem:metric}.
Let $F^s$ and $F^u$ be two $C^\infty$ subbundles of $TM$ with ${\rm dim}F^s=s$ and ${\rm dim}F^u=u$. Assume they are $C^0$-close to $E^s$ and $E^u$ respectively. Be precise, we assume
$$
\measuredangle(F^s,E^s)=
\max_{x\in M}
\max_{v\in F^s_x,\|v\|=1}d_{T_xM}(v,E^s_x)<10^{-3},
\quad {\rm and} \quad
\measuredangle(F^u,E^u)=
\max_{x\in M}
\max_{v\in F^u_x,\|v\|=1}d_{T_xM}(v,E^u_x)<10^{-3}.
$$
Since the distance between each pair of partially hyperbolic bundles are larger than $1-10^{-3}$ (Lemma \ref{lem:metric}), the decomposition
$$
TM=F^s\oplus E^c\oplus F^u
$$
is a direct sum decomposition. From now on, we fix these two bundles $F^s$ and $F^u$.

Assume the center bundle $E^c$ is orientable and oriented. Let $\theta_0:\RR\rightarrow M$ be a $C^1$ immersion and $k>0$, such that
\begin{itemize}
	\item $\theta_0'(t)\in E^c(\theta_0(t))$ is the positive unit vector for every $t\in\RR$;
	\item there exists $t_0\in\RR$ satisfying
	  \begin{align}\label{equ:center-per}
	  f^k\circ\theta_0(t)=
	  \theta_0\left(t_0+
	  \int_{0}^{t}\|Df^k|_{E^c_{\theta_0(s)}}\|{\rm d}s\right),
	  \qquad \forall t\in\RR.
	  \end{align}
\end{itemize}
These two properties imply $\theta_0$ is $k$-periodic, and $Df^k$ preserves the orientation of $E^c$.
However, we don't require $Df$ preserves the orientation of $E^c$.

For $i=1,\cdots,k-1$, the $C^1$-curve
$f^i\circ\theta_0:\RR\rightarrow M$
is non-degenerated and tangent to $E^c$ everywhere. We can reparameterize $f^i\circ\theta_0:\RR\rightarrow M$ by its length to satisfy the following lemma.

\begin{lemma}\label{lem:para}
	For every $i=1,\cdots,k-1$, there exist a $C^1$-increasing diffeomorphism $\alpha_i:\RR\rightarrow\RR$ satisfying $\alpha_i(0)=0$, and a $C^1$-curve $\theta_i:\RR\rightarrow M$, such that
	$$
	\theta_i(0)=f^i\circ\theta_0(0), \qquad
	\theta_i(t)=f^i\circ\theta_0\circ\alpha_i(t), 
	\qquad{\it and} \qquad
		 \|\theta_i'(t)\|\equiv1, \qquad
	\forall t\in\RR.
	$$
\end{lemma}

\begin{proof}
	For every $i=1,\cdots,k-1$, let $\bar{\alpha_i}:\RR\rightarrow\RR$ be defined as
	$$
	\bar{\alpha_i}(t)=\int_{0}^{t}\|Df^i|_{E^c_{\theta_0(s)}}\|{\rm d}s,
	\qquad \forall t\in\RR.
	$$
	Then $\bar{\alpha_i}$ is a $C^1$-smooth strictly increasing diffeomorphism of $\RR$ and  $\bar{\alpha_i}(0)=0$. 
	
	We define $\alpha_i:\RR\rightarrow\RR$ be the inverse function of $\bar{\alpha_i}$. It is a $C^1$-increasing diffeomorphism of $\RR$ and satisfies $\alpha_i(0)=0$. Let $\theta_i:\RR\rightarrow M$ be defined as
	$$
	\theta_i(t)=f^i\circ\theta_0\circ\alpha_i(t), \qquad \forall t\in\RR.
	$$
    It is a $C^1$-curve on $M$ satisfying $\theta_i(0)=f^i\circ\theta_0(0)$. Moreover, for every $t\in\RR$, let $t'=\alpha_i(t)$, then
    $$
    \|\theta_i'(t)\|~=~
    \|Df^i|_{E^c_{\theta_0(t')}}\|\cdot\|\theta_0'(t')\|\cdot\alpha_i'(t)~=~
    \|Df^i|_{E^c_{\theta_0(t')}}\|\cdot
    \frac{1}{\bar{\alpha_i}'(t')}
    ~=~\|Df^i|_{E^c_{\theta_0(t')}}\|\cdot
    \frac{1}{\|Df^i|_{E^c_{\theta_0(t')}}\|}~\equiv~ 1.
    $$
    This proves the lemma.
\end{proof}

For every $i=0,1,\cdots,k-1$, we use the symbol $\RR_i=\RR$ to emphasize $\theta_i:\RR_i=\RR\rightarrow M$.

\begin{lemma}\label{lem:inducing}
	For every $i=0,\cdots,k-1$, the diffeomorphism $f$ induces a diffeomorphism $\theta_i^*f: \RR_i\rightarrow\RR_{i+1}$   ($\RR_k=\RR_0$) defined as
	$$
	\theta_i^*f(t)=\int_{0}^{t}\|Df|_{E^c_{\theta_i(s)}}\|{\rm d}s,
	\qquad i=0\cdots,k-2, \qquad {\it and} \qquad
	\theta_{k-1}^*f(t)=
	t_0+\int_{0}^{t}\|Df|_{E^c_{\theta_{k-1}(s)}}\|{\rm d}s.
	$$ 
	They satisfy
	${\rm d}\theta_i^*f(t)/{\rm d}t=\|Df|_{E^c_{\theta_i(t)}}\|$ for every
	$t\in\RR_i$,
	and make the following diagram commuting:
\begin{displaymath}
\xymatrix{
\RR_0 \ar[r]^{\theta_0^*f} \ar[d]_{\theta_0} & \RR_1 \ar[r]^{\theta_1^*f} \ar[d]_{\theta_1} & \cdots \cdots \ar[r]^{\theta_{k-2}^*f} & \RR_{k-1} \ar[r]^{\theta_{k-1}^*f} \ar[d]^{\theta_{k-1}} & \RR_0 \ar[d]^{\theta_0} \\
M \ar[r]_{f} & M \ar[r]_{f} & \cdots \cdots \ar[r]_{f} & M \ar[r]_{f} & M }
\end{displaymath}
That is
$$
\theta_{i+1}\circ\theta_i^*f=f\circ\theta_i,   	
\qquad \forall i=0,1,\cdots,k-2, \qquad {\rm and} \qquad
\theta_{0}\circ\theta_{k-1}^*f=f\circ\theta_{k-1}.
$$
\end{lemma}

\begin{proof}
First we fix $i\in\{0,1,\cdots,k-2\}$. The induced diffeomorphism $\theta_i^*f:\RR_i\rightarrow \RR_{i+1}$ is defined as
\begin{align*}
\theta_i^*f(t)=\int_{0}^{t}\|Df|_{E^c_{\theta_i(s)}}\|{\rm d}s, 
\qquad {\rm thus} \quad
\frac{{\rm d}}{{\rm d}t}\theta_i^*f(t)=\|Df|_{E^c_{\theta_i(t)}}\|, \quad \forall t\in\RR_i.
\end{align*}
The definition of $\theta_i:\RR_i\rightarrow M$ in Lemma \ref{lem:para} implies
\begin{align*}
    f\circ\theta_i(t)~
       =~f\circ f^i\circ\theta_0\circ\alpha_i(t)
       ~=~f^{i+1}\circ\theta_0\circ\alpha_{i+1}\left( \alpha_{i+1}^{-1}\circ\alpha_i(t) \right) 
       ~=~\theta_{i+1}\left( \alpha_{i+1}^{-1}\circ\alpha_i(t) \right) .
\end{align*}
Since $\alpha_i^{-1}(t)=\int_{0}^{t}\|Df^i|_{E^c_{\theta_0(s)}}\|{\rm d}s$, we have
$\alpha_{i+1}^{-1}\circ\alpha_i(0)=0$. For every $t\in\RR$, if we denote $t'=\alpha_i(t)$, then
\begin{align*}
    \left[\alpha_{i+1}^{-1}\circ\alpha_i\right]'(t)~&=~
    \|Df^{i+1}|_{E^c_{\theta_0(t')}}\|\cdot
    \frac{1}{\|Df^i|_{E^c_{\theta_0(t')}}\|}
    ~=~\|Df|_{E^c_{f^i\circ\theta_0(t')}}\| \\
    &=~\|Df|_{E^c_{f^i\circ\theta_0\circ\alpha_i(t)}}\|
    ~=~\|Df|_{E^c_{\theta_i(t)}}\|.
\end{align*}
This implies
$$
\alpha_{i+1}^{-1}\circ\alpha_i(t)
~=~\int_{0}^{t}\|Df|_{E^c_{\theta_i(s)}}\|{\rm d}s
~=~\theta_i^*f(t).
$$
Thus $\theta_{i+1}\circ\theta_i^*f\equiv f\circ\theta_i$ for every $i=0,\cdots,k-2$.

\vskip3mm

When $i=k-1$, the induced diffeomorphism $\theta_{k-1}^*f:\RR_{k-1}\rightarrow \RR_0$ is defined as
\begin{align*}
   \theta_{k-1}^*f(t)=
   t_0+\int_{0}^{t}\|Df|_{E^c_{\theta_{k-1}(s)}}\|{\rm d}s, 
   \qquad   {\rm thus} \quad
   \frac{{\rm d}}{{\rm d}t}\theta_{k-1}^*f(t)=\|Df|_{E^c_{\theta_{k-1}(t)}}\|, \quad \forall t\in\RR_{k-1}.
\end{align*}
Recall (\ref{equ:center-per}) we have
  $$
  f^k\circ\theta_0(t)=
  \theta_0\left(t_0+
  \int_{0}^{t}\|Df^k|_{E^c_{\theta_0(s)}}\|{\rm d}s\right),
  \qquad \forall t\in\RR.
  $$
the definition of $\theta_{k-1}:\RR_i\rightarrow M$ in Lemma \ref{lem:para} implies
\begin{align}\label{equ:k-1}
  f\circ\theta_{k-1}(t)~
  &=~f\circ f^{k-1}\circ\theta_0\circ\alpha_{k-1}(t)
  ~=~f^k\circ\theta_0\circ\alpha_{k-1}(t) \notag\\
  &=~\theta_0\left(t_0+
  \int_{0}^{\alpha_{k-1}(t)}\|Df^k|_{E^c_{\theta_0(s)}}\|{\rm d}s\right)
\end{align}
Moreover, for every $t\in\RR$, if we denote $t'=\alpha_{k-1}(t)$, then
\begin{align*}
   \frac{{\rm d}}{{\rm d}t}
   \int_{0}^{\alpha_{k-1}(t)}\|Df^k|_{E^c_{\theta_0(s)}}\|{\rm d}s
   ~&=~
   \frac{{\rm d}\alpha_{k-1}(t)}{{\rm d}t}\cdot \|Df^k|_{E^c_{\theta_0(t')}}\| 
   ~=~\frac{1}{\|Df^{k-1}|_{E^c_{\theta_0(t')}}\|}
      \cdot \|Df^k|_{E^c_{\theta_0(t')}}\| \\
    &=~\|Df|_{E^c_{f^{k-1}\circ\theta_0(t')}}\|
      ~=~\|Df|_{E^c_{f^{k-1}\circ\theta_0\circ\alpha_{k-1}(t)}}\| \\
    &=~\|Df|_{E^c_{\theta_{k-1}(t)}}\|.
\end{align*}
Combining with (\ref{equ:k-1}), for every $t\in\RR_{k-1}$, we have
\begin{align*}
    f\circ\theta_{k-1}(t)~=~
    \theta_0\left(t_0+\int_{0}^{t}\|Df|_{E^c_{\theta_{k-1}(s)}}\|{\rm d}s\right)
    ~=~\theta_{0}\circ\theta_{k-1}^*f(t).
\end{align*}
\end{proof}

\begin{definition}\label{def:center-immer}
	Let $\bigsqcup_{i=0}^{k-1}\RR_i$ be the disjoint union of  $\RR_i$.
	The map $\theta:\bigsqcup_{i=0}^{k-1}\RR_i\rightarrow M$ and the induced diffeomorphism $\theta^*f:\bigsqcup_{i=0}^{k-1}\RR_i\rightarrow \bigsqcup_{i=0}^{k-1}\RR_i$ are defined as following, for every $i=0,\cdots,k-1$,
	$$
	\theta|_{\RR_i}=\theta_i:~\RR_i\longrightarrow M,
	\qquad {\rm and} \qquad
	\theta^*f|_{\RR_i}=\theta_i^*f:~\RR_i\longrightarrow \RR_{i+1}~(\RR_k=\RR_0).
	$$
\end{definition}

It is clear that the map $\theta:\bigsqcup_{i=0}^{k-1}\RR_i\rightarrow M$ is a $C^1$ leaf immersion, and it is a normally hyperbolic leaf immersion following the definition of \cite[Page 69]{HPS}. To be precise, we have the following proposition.
Its proof deduces from the fact that $f:M\rightarrow M$ is partially hyperbolic and Lemma \ref{lem:inducing} directly.

\begin{proposition}\label{prop:leaf-immersion}
	 The partially hyperbolic diffeomorphism $f:M\rightarrow M$ is {\it normally hyperbolic to the $C^1$ leaf immersion} $\theta:\bigsqcup_{i=0}^{k-1}\RR_i\rightarrow M$, i.e. they satisfy
	\begin{enumerate}
		\item Let $\Lambda=\theta(\bigsqcup_{i=0}^{k-1}\RR_i)$ and $\overline{\Lambda}$ be the closure of $\Lambda$. We have $f(\Lambda) =\Lambda$.
		\item The pull back diffeomorphism $\theta^*f$ satisfies $f\circ\theta=\theta\circ(\theta^*f)$ on $\bigsqcup_{i=0}^{k-1}\RR_i$.
		\item The $Df$-invariant splitting 
		$$
		T_{\overline{\Lambda}}M=(E^s|_{\overline{\Lambda}})\oplus T\overline{\Lambda}\oplus (E^u|_{\overline{\Lambda}})
		=(E^s|_{\overline{\Lambda}})\oplus (E^c|_{\overline{\Lambda}})\oplus (E^u|_{\overline{\Lambda}})
		$$
		is partially hyperbolic.
	\end{enumerate}
\end{proposition}

Since $TM=F^s\oplus E^c\oplus F^u$ is a direct sum decomposition, we can pull back the bundle $F^s\oplus F^u$ by $C^1$ leaf immersion $\theta:\bigsqcup_{i=0}^{k-1}\RR_i\rightarrow M$. There exists a fiber bundle $\theta^*(F^s\oplus F^u)$ over $\bigsqcup_{i=0}^{k-1}\RR_i$, such that the following diagram commutes:
\begin{displaymath}
\xymatrix{
	\theta^*(F^s\oplus F^u) \ar[r]^{\theta_*} \ar[d]_{\pi} & F^s\oplus F^u \ar[d]^{\pi} \\
	\bigsqcup_{i=0}^{k-1}\RR_i \ar[r]_{\theta} & M }
\end{displaymath}

For every $t\in\RR_i$, the fiber
$$
\theta^*(F^s\oplus F^u)_t=\theta^*(F^s_{\theta_i(t)}\oplus F^u_{\theta_i(t)})
$$
is a linear space equipped with an inner product structure by pulling back the Riemannian metric on $F^s_{\theta_i(t)}\oplus F^u_{\theta_i(t)}\subset T_{\theta_i(t)}M$.
This defines a metric $\|\cdot\|_t$ on each fiber $\theta^*(F^s\oplus F^u)_t$, which vary continuously with respect to $t$ in $\RR_i$. For every $\delta>0$, we denote 
$$
\theta^*(F^s\oplus F^u)_t(\delta)=\left\lbrace v\in\theta^*(F^s\oplus F^u)_t:~\|v\|_t<\delta \right\rbrace.
$$ 
For every $(a,b)\subseteq\RR_i$, we denote
$$
\theta^*(F^s\oplus F^u)_{(a,b)}(\delta)=
\bigcup_{t\in(a,b)}\theta^*(F^s\oplus F^u)_t(\delta),
\qquad {\rm and} \qquad
\theta^*(F^s\oplus F^u)(\delta)= \bigcup_{t\in\bigsqcup_{i=0}^{k-1}\RR_i}\theta^*(F^s\oplus F^u)_t(\delta).
$$
Now we show that the exponential map from $\theta^*(F^s\oplus F^u)$ to $M$ is a uniform local diffeomorphism.

\begin{proposition}\label{prop:exp-map}
	Let the exponential map $\Phi:\theta^*(F^s\oplus F^u)\rightarrow M$ be defined as:
	$$
	\Phi(v)=\exp_{\theta_i(t)}\left(\theta_{i*}(v)\right), 
	\qquad  \forall t\in\RR_i, 
	\quad\forall v\in\theta^*(F^s\oplus F^u)_t,
	$$
	then $\Phi$ is a uniform local diffeomorphism. 
	
	To be precise, there exists a constant $\delta_1=\delta_1(M,f,F^s,F^u)>0$, such that for every $0<\delta\leq\delta_1$ and every $t\in\RR_i$, the set $\Phi\left( \theta^*(F^s\oplus F^u)_{(t-\delta,t+\delta)}(\delta) \right)$ is an open subset in $M$, and 
	$$
	\Phi:~\theta^*(F^s\oplus F^u)_{(t-\delta,t+\delta)}(\delta)
	 \longrightarrow \Phi\left( \theta^*(F^s\oplus F^u)_{(t-\delta,t+\delta)}(\delta) \right)
	$$
	is a diffeomorohism.
\end{proposition}

\begin{proof}
	Let $\gamma:(-1,1)\rightarrow M$ be a $C^1$ curve which satisfies $\gamma'(t)\in E^c_{\gamma(t)}\subset T_{\gamma(t)}M$ is an unit vector for every $t\in(-1,1)$, and $\gamma^*(F^s\oplus F^u)$ be the lifting bundle. We denote $v^c=
	\gamma'(0)\in E^c_{\gamma(0)}$, which satisfies $\|v^c\|=1$. For every $\tau\in\RR$ and $v^{su}\in F^s_{\gamma(0)}\oplus F^u_{\gamma(0)}$, we denote 
	$\left(\tau,\gamma^*(v^{su})\right)\in 
	 T_{(0,0)}\gamma^*(F^s\oplus F^u)\cong T_0\RR\oplus\gamma^*(F^s_{\gamma(0)}\oplus F^u_{\gamma(0)})$ 
	a tangent vector. Then the tangent map 
	$D\Phi_{(0,0)}:T_{(0,0)}\gamma^*(F^s\oplus F^u)\rightarrow T_{\gamma(0)}M$ satisfies
	$$
	D\Phi_{(0,0)}\left(\tau,\gamma^*(v^{su})\right)=\tau\cdot v^c+v^{su}\in T_{\gamma(0)}M
	\cong E^c_{\gamma(0)}\oplus F^s_{\gamma(0)}\oplus F^u_{\gamma(0)}.
	$$
	
	This implies if we consider the following bundle isomorphism at the point $(0,0)\in\gamma^*(F^s\oplus F^u)$:
	$$
	T_{(0,0)}\gamma^*(F^s\oplus F^u)\cong T_0\RR\oplus\gamma^*(F^s_{\gamma(0)}\oplus F^u_{\gamma(0)})\cong E^c_{\gamma(0)}\oplus F^s_{\gamma(0)}\oplus F^u_{\gamma(0)}\cong T_{\gamma(0)}M,
	$$
	the tangent map
	$$
	D\Phi_{(0,0)}=D\exp|_{(\gamma(0),0)}:T_{(0,0)}\gamma^*(F^s\oplus F^u)\longrightarrow T_{\gamma(0)}M
	$$
	is the identity map.

	From the inverse function theorem, there exists a constant $\delta_1>0$, such that for every $0<\delta\leq\delta_1$, the map
	$$
	\Phi:~\gamma^*(F^s\oplus F^u)_{(-\delta,\delta)}(\delta)
	 \longrightarrow \Phi\left( \gamma^*(F^s\oplus F^u)_{(-\delta,\delta)}(\delta) \right)
	$$
	is a diffeomorphism. Moreover, the compactness of $M$ implies $\delta_1$ only depends on $M,f$ and the bundle $F^s,F^u$, which is independent of $\gamma$.
	Finally, we apply this analysis to $\theta_i:\RR_i\rightarrow M$ at $t\in\RR_i$, which proves the proposition.
\end{proof}

Let $C_1=100\cdot\sup_{z\in M}\|Df(z)\|$. Since $\Phi$ is a local diffeomorphism,
for every $t\in\bigsqcup_{i=0}^{k-1}\RR_i$ and every $v\in\theta^*(F^s\oplus F^u)_t(\delta_1/C_1)$, there exists a unique point 
$$
\theta^*f(v):=\Phi^{-1}\circ f\circ\Phi(v)\in\theta^*(F^s\oplus F^u)(\delta_1/2).
$$
Here $\Phi^{-1}$ is defined from $\Phi\left( \theta^*(F^s\oplus F^u)_{(\theta^*f(t)-\delta,\theta^*f(t)+\delta)}(\delta_1) \right)$ to the region $\theta^*(F^s\oplus F^u)_{(\theta^*f(t)-\delta,\theta^*f(t)+\delta)}(\delta)$. Then we have the following diagram commutes:
\begin{displaymath}
\xymatrix{
	\theta^*(F^s\oplus F^u)(\delta_1/C_1) \ar[r]^{\theta^*f} \ar[d]_{\Phi} & \theta^*(F^s\oplus F^u)(\delta_1/2) \ar[d]^{\Phi} \\
	M \ar[r]_{f} & M }
\end{displaymath}
Moreover, the lifting map $\theta^*f$ is a diffeomorphism from $\theta^*(F^s\oplus F^u)(\delta_1/C_1)$ to its image 
$$
\theta^*f\left(\theta^*(F^s\oplus F^u)(\delta_1/C_1)\right)~\subset~\theta^*(F^s\oplus F^u)(\delta_1/2).
$$

\begin{proposition}\label{prop:PH-local-diffeo}
	The local diffeomorphism $\theta^*f:\theta^*(F^s\oplus F^u)(\delta_1/C_1)\rightarrow\theta^*(F^s\oplus F^u)(\delta_1/2)$ is partially hyperbolic at the invariant set $\bigsqcup_{i=0}^{k-1}\RR_i$. The partially hyperbolic splitting at $\bigsqcup_{i=0}^{k-1}\RR_i$ is
	\begin{align*}
	   T_{\bigsqcup_{i=0}^{k-1}\RR_i}\theta^*(F^s\oplus F^u) 
	   &=\theta^*(E^s)|_{\bigsqcup_{i=0}^{k-1}\RR_i} \oplus T(\bigsqcup_{i=0}^{k-1}\RR_i)\oplus \theta^*(E^u)|_{\bigsqcup_{i=0}^{k-1}\RR_i} \\
	   &=\left(\theta^*(E^s)\oplus \theta^*(E^c)\oplus \theta^*(E^u)\right)|_{\bigsqcup_{i=0}^{k-1}\RR_i}.
	\end{align*}	
\end{proposition}

\begin{proof}
	The diffeomorphism $f$ is partially hyperbolic on $M$, and $\theta_i(\RR_i)$ is tangent to $E^c$ everywhere for every $i=0,\cdots,k-1$. Moreover, we have
	$
	\Phi\circ\theta^*f=f\circ\Phi
	$
	on $\theta^*(F^s\oplus F^u)(\delta_1/C_1)$, and the tangent map $D\Phi$ satisfies
	$$
	D\Phi|_{T_{\bigsqcup_{i=0}^{k-1}\RR_i}\theta^*(F^s\oplus F^u)} \equiv
	\theta^*|_{\bigsqcup_{i=0}^{k-1}\RR_i}.
	$$
	This shows that $\theta^*f$ is partially hyperbolic on  $\bigsqcup_{i=0}^{k-1}\RR_i$ with partially hyperbolic splitting
	\begin{align*}
	    T_{\bigsqcup_{i=0}^{k-1}\RR_i}\theta^*(F^s\oplus F^u)
	    &=D\Phi^{-1}|_{\bigsqcup_{i=0}^{k-1}\RR_i}(E^s)\oplus D\Phi^{-1}(\theta(\bigsqcup_{i=0}^{k-1}\RR_i))\oplus D\Phi^{-1}|_{\bigsqcup_{i=0}^{k-1}\RR_i}(E^u) \\
	    &=\left(\theta^*(E^s)\oplus \theta^*(E^c)\oplus \theta^*(E^u)\right)|_{\bigsqcup_{i=0}^{k-1}\RR_i}.
	\end{align*}	
\end{proof}
    
The following theorem was established in \cite[Theorem 6.1]{HPS}, which shows the existence and properties of local invariant manifolds of $\theta^*f$.

\begin{theorem}[Theorem 6.1 of \cite{HPS}]\label{thm:HPS}
	The diffeomorphism $\theta^*f:\theta^*(F^s\oplus F^u)(\delta_1/C_1)\rightarrow \theta^*(F^s\oplus F^u)(\delta_1/2)$ satisfies:
	\begin{enumerate}
		\item Through $\bigsqcup_{i=0}^{k-1}\RR_i$, there exists $C^1$ manifolds $W^{cu}, W^{cs}\subset\theta^*(F^s\oplus F^u)(\delta_1/C_1)$ satisfying
		      $$
		      \theta^*f(W^{cs})\subset W^{cs} \qquad {\it and} \qquad 
		      W^{cu}\subset\theta^*f(W^{cu}).
		      $$
		      The boundary of manifolds $\partial W^{cs}\cup\partial W^{cu}\subset\partial \theta^*(F^s\oplus F^u)(\delta_1/C_1)$, and
		      $$
		      T_{\bigsqcup_{i=0}^{k-1}\RR_i}W^{cs}
		      =\theta^*(E^s)\oplus T(\bigsqcup_{i=0}^{k-1}\RR_i)
		       \qquad {\it and} \qquad
		      T_{\bigsqcup_{i=0}^{k-1}\RR_i}W^{cu}
		      =T(\bigsqcup_{i=0}^{k-1}\RR_i)\oplus \theta^*(E^u).
		      $$
		\item The manifold $W^{cs}$ consists of points whose positive orbit staying in $\theta^*(F^s\oplus F^u)(\delta_1/C_1)$, and $W^{cu}$ consists of points whose negative orbit staying in $\theta^*(F^s\oplus F^u)(\delta_1/C_1)$:
		      $$
		      W^{cs}=\bigcap_{n=0}^{+\infty}\theta^*f^{-n}\left(\theta^*(F^s\oplus F^u)(\delta_1/C_1)\right) \qquad {\it and} \qquad
		      W^{cu}=\bigcap_{n=0}^{+\infty}\theta^*f^n\left(\theta^*(F^s\oplus F^u)(\delta_1/C_1)\right).
		      $$
		\item The manifolds $W^{cs}$ and $W^{cu}$ are foliated by $C^1$ strong stable and unstable discs:
		$$
		W^{cs}=\bigcup_{t\in\bigsqcup_{i=0}^{k-1}\RR_i}W^s_t
		\qquad {\it and} \qquad
		W^{cu}=\bigcup_{t\in\bigsqcup_{i=0}^{k-1}\RR_i}W^u_t,
		$$
		which satisfy $T_tW^s_t=\theta^*(E^s)_t$ and $T_tW^u_t=\theta^*(E^u)_t$.
		For every $t\in\bigsqcup_{i=0}^{k-1}\RR_i$, they satisfy
		\begin{align*}
		     \theta^*f(W^s_t)\subset W^s_{\theta^*f(t)}; 
		     \qquad {\it and} \qquad
		     W^u_{\theta^*f(t)}\subset \theta^*f(W^u_t).
		\end{align*}
	    Actually, for every $t\in\bigsqcup_{i=0}^{k-1}\RR_i\subset\theta^*(F^s\oplus F^u)(\delta_1/C_1)$, we have
		\begin{align*}
		    W^s_t=
		    \Phi^{-1}\left({\cal F}^s_{\theta(t)}(\delta_1)\right)
		      \cap \theta^*(F^s\oplus F^u)(\delta_1/C_1), 
		    \quad
		    W^u_t=
		    \Phi^{-1}\left({\cal F}^u_{\theta(t)}(\delta_1)\right)
		      \cap \theta^*(F^s\oplus F^u)(\delta_1/C_1).
		\end{align*}
		Here $\Phi^{-1}$ is defined as the inverse of the local diffeomorphim $\Phi$ around $t$.     
	\end{enumerate}
\end{theorem}
\begin{remark}\label{rk:tangent}
	$\Phi(W^{cu})$ and $\Phi(W^{cs})$ are tangent to $E^c\oplus E^u$ and $E^s\oplus E^c$ everywhere, respectively. 
\end{remark}

We denote 
$d_{W^{cu}}(\cdot,\cdot)$ the distance induced by pulling back the Riemannian metric restricted on $\Phi(W^{cu})$. It is locally well defined on $W^{cu}$. The same for $d_{W^{cs}}(\cdot,\cdot)$ on $W^{cs}$.
Since  the expanding rate of $\theta^*f$ on $W^u_t$ and  contracting rate on $W^s_t$ are uniform for every $t\in\bigsqcup_{i=0}^{k-1}\RR_i$, we have the following corollary of Theorem \ref{thm:HPS}.

\begin{corollary}\label{cor:boundary}
	There exists a constant $c_1=c_1(f,X,F^s,F^u)>0$, such that
	\begin{itemize}
		\item for every $v\in \theta^*f^{-1}(W^{cu})$, we have
		$B_{W^{cu}}(v,c_1\delta_1)=
		\left\lbrace v'\in W^{cu}:~ d_{W^{cu}}(v,v')\leq c_1\delta_1 \right\rbrace 
		\subseteq W^{cu}$;
		\item for every $v\in\theta^*f(W^{cs})$, we have
		$B_{W^{cs}}(v,c_1\delta_1)=
		\left\lbrace v'\in W^{cs}:~ d_{W^{cs}}(v,v')\leq c_1\delta_1 \right\rbrace 
		\subseteq W^{cs}$. 
	\end{itemize}
	Moreover, $\Phi(B_{W^{cu}}(v,c_1\delta_1))$ and $\Phi(B_{W^{cs}}(v,c_1\delta_1))$ are imbedded disks in $M$ tangent to $E^c\oplus E^u$ and $E^s\oplus E^c$ respectively. 
\end{corollary}

\subsection{Lipschitz shadowing of invariant manifolds}
\label{subsec:Lipschitz}

Let $X\in{\cal X}^r(M)$ be a $C^r$ vector field on $M$, and denote $X_{\tau}$ the time-$\tau$ map generated by $X$.  Let $\widetilde{X}$ be the lifting vector field on $\theta^*(F^s\oplus F^u)(\delta_1)$, which is defined as 
$$
\widetilde{X}(v)=D\Phi^{-1}\left( X(\Phi(v))\right), \qquad \forall v\in \theta^*(F^s\oplus F^u)(\delta_1).
$$
Since the leaf immersion $\theta$ is $C^1$-smooth, the lifting vector field $\widetilde{X}$ is a $C^0$ vector field on $\theta^*(F^s\oplus F^u)(\delta_1)$. However, for $\tau$ small enough, $\widetilde{X}$ generates a family of $C^1$ diffeomorphisms 
$$
\widetilde{X}_{\tau}:\theta^*(F^s\oplus F^u)(\delta_1/2)\rightarrow \theta^*(F^s\oplus F^u)(\delta_1),
$$
which satisfies
$\widetilde{X}_{\tau}=\Phi^{-1}\circ X_{\tau}\circ\Phi$. 

Since the exponential map $\Phi:\theta^*(F^s\oplus F^u)_{(t-\delta_1,t+\delta_1)}(\delta_1)\rightarrow M$ is a local diffeomorphism for every $t\in\RR_i$, we can define a local metric on $\theta^*(F^s\oplus F^u)(\delta_1)$ via pulling back the Riemannian metric on $M$ by $\Phi$. For every $v_1,v_2\in \theta^*(F^s\oplus F^u)(\delta_1/2)$ which are close enough, we denote $\tilde{d}(v_1,v_2)$ the distance between $v_1$ and $v_2$ defined by this induced metric, i.e.
$\tilde{d}(v_1,v_2)=d(\Phi(v_1),\Phi(v_2))$. For every set $K\subset\theta^*(F^s\oplus F^u)(\delta_1)$ and $\delta>0$, we denote
$$
B(K,r)=\left\lbrace v\in\theta^*(F^s\oplus F^u)(\delta_1):~
\tilde{d}(v,K)\leq \delta  \right\rbrace 
$$

\begin{lemma}\label{lem:C1-perturb}
	There exists $\tau'=\tau'(f,X,F^s,F^u)>0$, such that for every for every $\tau$ satisfying $|\tau|<\tau'$, the family of $C^1$ diffeomorphisms
	$$
	F_{\tau}=\widetilde{X}_{\tau}\circ\theta^*f:
	~\theta^*(F^s\oplus F^u)(\delta_1/C_1)\longrightarrow \theta^*(F^s\oplus F^u)(\delta_1)
	$$
	is well defined. The diffeomorphism $F_{\tau}$ converge to $\theta^*f$ in the $C^1$-topology as $\tau$ tends to $0$. Moreover, there exists a constant $L_0>1$ which depends on $X$ only, such that $\tilde{d}(v,\widetilde{X}_{\tau}(v))\leq L_0|\tau|$ for every $v\in\theta^*(F^s\oplus F^u)(\delta_1/2)$. This implies	
	$$
	\tilde{d}\left(F_{\tau}(v),\theta^*f(v)\right)
	=\tilde{d}\left(F_{\tau}(v),F_{0}(v)\right)
	\leq L_0\cdot|\tau|, 
	\qquad \forall v\in \theta^*(F^s\oplus F^u)(\delta_1/C_1).
	$$
\end{lemma}

The following theorem was established in \cite[Theorem 6.8]{HPS}, which states the permanence of invariant manifolds after perturbations.

\begin{theorem}[Theorem 6.8 of \cite{HPS}]\label{thm:permanence}
	There exists a contant $\tau''=\tau''(f,X,F^s,F^u)\in(0,\tau']$, such that for every $\tau$ satisfying $|\tau|<\tau''$, there exists an $F_{\tau}$-invariant section $\sigma_{\tau}:\bigsqcup_{i=0}^{k-1}\RR_i\rightarrow\theta^*(F^s\oplus F^u)(\delta_1)$, such that 
	$$
	\sigma_{\tau}\left(\bigsqcup_{i=0}^{k-1}\RR_i\right)=
	\bigcap_{n\in\mathbb{Z}}F_{\tau}^n\left(\theta^*(F^s\oplus F^u)(\delta_1/C_1)\right).
	$$
	The section $\sigma_{\tau}$ is $C^1$-smooth and vary continuously with respect to $\tau$. The sets
	$$
	W^{cs}_{\tau}=\bigcap_{n=0}^{+\infty}F_{\tau}^{-n}\left(\theta^*(F^s\oplus F^u)(\delta_1/C_1)\right) \qquad {\it and} \qquad
	W^{cu}_{\tau}=\bigcap_{n=0}^{+\infty}F_{\tau}^n\left(\theta^*(F^s\oplus F^u)(\delta_1/C_1)\right)
	$$
	are $C^1$-manifolds intersecting transversely at $\sigma_{\tau}\left(\bigsqcup_{i=0}^{k-1}\RR_i\right) 
	=W^{cs}_{\tau}\cap W^{cu}_{\tau}$.
	
	Moreover, the section $\sigma_{\tau}$ converges to $0$-section in $C^1$-topology as $\tau\rightarrow0$. This means for every $\e>0$, there exists $\tau_\e=\tau_\e(f,X,F^s,F^u,\e)>0$, such that  if $|\tau|<\tau_\e$, then the $F_{\tau}$-invariant section $\sigma_{\tau}$ satisfies $\|\sigma_{\tau}-0\|_{C^1}<\e$. In particular, this implies if $|\tau|<\tau_\e$, then
	$$
	\sigma_{\tau}\left(\bigsqcup_{i=0}^{k-1}\RR_i\right)
	~\subset~\theta^*(F^s\oplus F^u)(\e).
	$$
\end{theorem}

\begin{remark}\label{rk:permanence}
Notice here for every $\e>0$, the constant $\tau_\e$ is independent of the choice of periodic center curves $\theta$.
In the proof of \cite[Theorem 6.8]{HPS}, the authors proved that $W^{cs}_{\tau}$ and $W^{cu}_{\tau}$ converge to $W^{cs}$ and $W^{cu}$ in $C^1$-topology when $\tau$ tends to zero. Thus $\sigma_{\tau}$ converge to $0$-section in $C^1$-topology.
\end{remark}

The following theorem shows that the section $\sigma_{\tau}$ is Lipschitz with respect to the parameter $\tau$. This means there exists some constant $L>1$, such that when $\tau$ is small enough in Theorem \ref{thm:permanence}, the new section $\sigma_{\tau}$ is contained in $(L\cdot|\tau|)$-neighborhood of zero section. In particular, if we want $\|\sigma_\tau\|_{C^0}<\e$, then we need to take $\tau_\e<\e/L$.

\begin{theorem}\label{thm:Lipschitz}
	There exist two constants $0<\tau_0=\tau_0(f,X,F^s,F^u)<\tau''$ and $L=L(f,X,F^s,F^u)>1$, such that for every $\tau\in(-\tau_0,\tau_0)$, every $t\in\RR_i$ and any $C^1$ map $\psi:E^s_{\theta_i(t)}(\delta_1)\oplus E^u_{\theta_i(t)}(\delta_1)\rightarrow E^c_{\theta_i(t)}(\delta_1)$ which satisfies $\psi(0)=0$, $\|\partial\psi/\partial s\|<10^{-3}$, and $\|\partial\psi/\partial u\|<10^{-3}$, then the $C^1$-submanifold
	$$
	D^{su}_{\theta_i(t)}
	=\exp_{\theta_i(t)}\left({\rm Graph}(\psi)\right)
	=\exp_{\theta_i(t)}\left(\left\lbrace v^{su}+\psi(v^{su}):v^{su}\in E^s_{\theta_i(t)}(\delta_1)\oplus E^u_{\theta_i(t)}(\delta_1)\right\rbrace \right)
	$$
	intersects $\Phi\circ\sigma_{\tau}((t-\delta_1,t+\delta_1))$ with a unique point $q=q(t,\tau,\psi)$. Moreover, if we denote
	$$
	\exp_{\theta_i(t)}^{-1}(q)=q^{su}+\psi(q^{su})\in 
	\left(E^s_{\theta_i(t)}(\delta_1)\oplus E^u_{\theta_i(t)}(\delta_1)\right)
	\oplus E^c_{\theta_i(t)}(\delta_1),
	$$
	then 
	$$
	\|q^{su}\|+\|\psi(q^{su})\|\leq L\cdot|\tau|.
	$$
	In particular, this implies
	$\sigma_{\tau}\left(\bigsqcup_{i=0}^{k-1}\RR_i\right)\subset\theta^*_i(F^s\oplus F^u)(L\cdot|\tau|)$.
\end{theorem}

\begin{proof}
	From Theorem \ref{thm:HPS}, we know that
	$$
	W^{cu}
	=\bigcup_{t\in\bigsqcup_{i=0}^{k-1}\RR_i}W^u_t
	=\bigcup_{t\in\bigsqcup_{i=0}^{k-1}\RR_i}
	  \Phi^{-1}\left({\cal F}^u_{\theta(t)}(\delta_1)\right)
	  \cap\theta^*(F^s\oplus F^u)(\delta_1/C_1).
	$$ 
	Thus $\Phi(W^{cu})$ is tangent to $E^c\oplus E^u$ everywhere.
	
	For every $v\in\theta^*(F^s\oplus F^u)(\delta_1/C_1)$, we denote $t_v\in\bigsqcup_{i=0}^{k-1}\RR_i$, such that $v\in\theta^*(F^s\oplus F^u)_{t_v}(\delta_1/C_1)$.
	For every $0<\delta<\delta_1/C_1$, we denote
	$$
	W^s_{v}(\delta)=\Phi^{-1}\left({\cal F}^s_{\Phi(v)}(\delta)\right) 
	\qquad {\rm and} \qquad
	W^u_{v}(\delta)=\Phi^{-1}\left({\cal F}^u_{\Phi(v)}(\delta)\right).
	$$
	We can see that $W^s_{v}(\delta)$ and $W^u_{v}(\delta)$ are the strong stable and unstable manifolds of $v$ with respect to $\theta^*f$.
	Let $K^{cu}\subseteq W^{cu}$ and $K^{cs}\subseteq W^{cs}$, we denote 
	$$
	B^s_{\delta}(K^{cu})=\bigcup_{v\in K^{cu}}W^s_v(\delta)
	\qquad {\rm and} \qquad
	B^u_{\delta}(K^{cs})=\bigcup_{v\in K^{cs}}W^u_v(\delta).
	$$
	
	\begin{claim}
		Let $c_1$ be the constant in Corollary \ref{cor:boundary}.
		There exists $\tau'''=\tau'''(f,X,F^s,F^u)>0$, such that for every $\tau\in(-\tau''',\tau''')$, we have
		$$
		\sigma_{\tau}\left(\bigsqcup_{i=0}^{k-1}\RR^i\right)~\subset~
		B^s_{c_1\delta_1}\left(\theta^*f^{-2}(W^{cu})\right)
		~\cap~
		B^u_{c_1\delta_1}\left(\theta^*f^2(W^{cs})\right).
		$$
	\end{claim}
	\begin{proof}[Proof of the claim]	
	From the local product structure in Lemma \ref{lem:local-product}, both sets
	$B^s_{c_1\delta_1}\left(\theta^*f^{-2}(W^{cu})\right)$ and $B^u_{c_1\delta_1}\left(\theta^*f^2(W^{cs})\right)$ contain a neighborhood of $\bigsqcup_{i=0}^{k-1}\RR^i$. Theorem \ref{thm:permanence} shows that there exists $\tau'''>0$, such that for every $\tau\in(-\tau''',\tau''')$, we have
	$\sigma_{\tau}\left(\bigsqcup_{i=0}^{k-1}\RR^i\right)\subset
	B^s_{c_1\delta_1}\left(\theta^*f^{-2}(W^{cu})\right)
	\cap
	B^u_{c_1\delta_1}\left(\theta^*f^2(W^{cs})\right)$.
	\end{proof}
	
	Let $\rho>0$ be the constant in Lemma \ref{lem:tubular} and $L_0$ be the constant in Lemma \ref{lem:C1-perturb}, which both depend only on $f,X,F^s,F^u$. We denote 
	$$
	\tau_0=\min\left\lbrace \tau''', \frac{\rho c_1\delta_1}{L_0} \right\rbrace 
	\qquad {\rm and} \qquad
	L_1=\frac{L_0}{\rho}.
	$$

	\begin{claim}\label{claim:Lip-Tubular-nbhd}
	   For every $\tau$ satisfying $|\tau|<\tau_0$, we have
	   $$
	   \sigma_{\tau}\left(\bigsqcup_{i=0}^{k-1}\RR^i\right)~\subset~
	   B^s_{L_1|\tau|}\left(\theta^*f^{-2}(W^{cu})\right)\bigcap
	   B^u_{L_1|\tau|}\left(\theta^*f^2(W^{cs})\right).
	   $$
	\end{claim}
	
	\begin{proof}[Proof of the claim]
	   We first show that $\sigma_{\tau}\left(\bigsqcup_{i=0}^{k-1}\RR^i\right)\subset B^s_{L_1|\tau|}\left(\theta^*f^{-2}(W^{cu})\right)$. 
	   Since $\tau_0\leq\frac{\rho c_1\delta_1}{L_0}$, then for every $\tau\in(-\tau_0,\tau_0)\setminus\{0\}$($\tau=0$ is trivial), there exists $n>0$ and a sequence numbers $\left\lbrace r_j \right\rbrace_{j=0}^n$ satisfying 
	   $$
	   r_0=c_1\delta_1
	   ~>~r_1=\lambda^{\frac{1}{2}}c_1\delta_1
	   ~>~\cdots\cdots
	   ~>~r_{n-1}=\lambda^{\frac{n-1}{2}}c_1\delta_1
	   ~>~\frac{L_0|\tau|}{\rho}=L_1|\tau|
	   ~\geq~
	   r_n=\lambda^{\frac{n}{2}}c_1\delta_1.
	   $$
	   For every $j=0,\cdots,n-1$, we have
	   $$
	   \theta^*f\left(
	   B^s_{r_j}\left(\theta^*f^{-2}(W^{cu})\right)
	   \right) \subseteq 
	   B^s_{\lambda r_j}\left(\theta^*f^{-1}(W^{cu})\right)
	   $$
	   Since $L_0|\tau|\leq\rho r_{n-1}\leq \rho r_j$, Corollary \ref{cor:boundary} and Lemma \ref{lem:tubular} shows that  every $v\in \theta^*f^{-1}(W^{cu})$ satisfies
	   $$
	   B\left(W^s_v(\lambda r_j),L_0|\tau|\right)
	   \subseteq B^s_{\sqrt{\lambda} r_j}\left( W^{cu} \right).
	   $$
	   Thus 
	   \begin{align*}
	       \tilde{X}_{\tau}\left( 
	       B^s_{\lambda r_j}\left(\theta^*f^{-1}(W^{cu})\right) \right)
	       ~&\subseteq~
	       B\left( B^s_{\lambda r_j}\left(\theta^*f^{-1}(W^{cu})\right),
	       L_0|\tau| \right) \\
	       ~&\subseteq~ 
	       B\left( B^s_{\lambda r_j}\left(\theta^*f^{-1}(W^{cu})\right),
	       \rho r_j \right) 
	       ~\subseteq~
	        B^s_{r_{j+1}}\left(W^{cu}\right).
	   \end{align*}
	   This implies for every $j=0,\cdots,n-1$, we have
	   \begin{align*}
	       F_{\tau}\left( 
	       B^s_{r_j}\left(\theta^*f^{-2}(W^{cu})\right) 
	       \right)
	       ~\cap~
	       B^s_{c_1\delta_1}\left(\theta^*f^{-2}(W^{cu})\right)
	       ~\subseteq~
	       B^s_{r_{j+1}}\left(\theta^*f^{-2}(W^{cu})\right)
	   \end{align*}
	   Since $r_n\leq L_1|\tau|$, we have
	   $$
	   \sigma_{\tau}\left(\bigsqcup_{i=0}^{k-1}\RR^i\right)\subset B^s_{L_1|\tau|}\left(\theta^*f^{-2}(W^{cu})\right).
	   $$
	   The proof for $\sigma_{\tau}\left(\bigsqcup_{i=0}^{k-1}\RR^i\right)\subset B^u_{L_1|\tau|}\left(\theta^*f^2(W^{cs})\right)$
	   is the same.
	\end{proof}
	
    From Theorem \ref{thm:permanence}, the center-stable manifold $W^{cs}_{\tau}$ intersects the center-unstable manifold $W^{cu}_{\tau}$ transversely in $\sigma_{\tau}\left(\bigsqcup_{i=0}^{k-1}\RR_i\right)$. 
	This implies $\sigma_{\tau}(\RR_i)$ is contained in the $(2L_1|\tau|)$-neighborhood of $\RR_i$ following the metric $\tilde{d}(\cdot,\cdot)$ on $\theta^*(F^s\oplus F^u)(\delta_1/C_1)$. From the local product structure, there exists a unique point $q_{\tau}$ satisfying
	$$
	q_{\tau}^{su}~=~
	\exp_{\theta_i(t)}^{-1}(q_{\tau})~=~
	\exp_{\theta_i(t)}^{-1}\circ\Phi\circ\sigma_{\tau}
	\left((t-\delta_1,t+\delta_1)\right)~
	\pitchfork~ E^s_{\theta_i(t)}(\delta_1)\oplus E^u_{\theta_i(t)}(\delta_1).
	$$ 
	Moreover, it satisfies $\|q_{\tau}^{su}\|\leq 4L_1\cdot|\tau|$. 
	
	Now let $\psi:E^s_{\theta_i(t)}(\delta_1)\oplus E^u_{\theta_i(t)}(\delta_1)\rightarrow E^c_{\theta_i(t)}(\delta_1)$ be a $C^1$-map satisfying 
	$$
	\psi(0)=0, \qquad \|\partial\psi/\partial s\|<10^{-3},
	\qquad {\rm and} \qquad
	\|\partial\psi/\partial u\|<10^{-3}.
	$$
	The disk $D^{su}_{\theta_i(t)}=\exp_{\theta_i(t)}({\rm Graph}(\psi))$ is tangent to $10^{-2}$-cone field of $E^s\oplus E^u$. From the local product the  segment $\sigma_{\tau}\left((t-\delta_1,t+\delta_1)\right)$ intersects $D^{su}_{\theta_i(t)}$ with a unique point $q$. Then for $\exp_{\theta_i(t)}^{-1}(q)=q^{su}+\psi(q^{su})$, since $q^{su}_{\tau}$ is connected to $q$ by a sub-arc of $\Phi\circ\sigma_{\tau}\left((t-\delta_1,t+\delta_1)\right)$, we have the following estimation
	$$
	\|q^{su}_{\tau}-q^{su}\|<\psi(q^{su})\leq
	10^{-3}\cdot\|q^{su}\|.
	$$
	So we have
	$\|q^{su}\|\leq\|q^{su}_{\tau}\|/(1-10^{-3})$, and 
	$\|\psi(q^{su})\|\leq10^{-3}\cdot\|q^{su}\|\leq10^{-3}\cdot\|q^{su}_{\tau}\|/(1-10^{-3})$.
	Thus for $L=10L_1$, we have
	$$
	\|q^{su}\|+\|\psi(q^{su})\|\leq \frac{1+10^{-3}}{1-10^{-3}}\cdot4L_1\cdot|\tau|\leq L\cdot|\tau|,
	$$
	which proves the theorem.
\end{proof}

The following corollary of Theorem \ref{thm:Lipschitz} will be needed in Section \ref{subsec:local}.

\begin{corollary}\label{cor:Lipschitz}
	Let the map $\varphi^s_{\theta_i(t)}:E^s_{\theta_i(t)}(\delta_0)\rightarrow E^c_{\theta_i(t)}(\delta_0)\oplus E^u_{\theta_i(t)}(\delta_0)$
	define the local stable manifold of $\theta_i(t)$ as 
	${\cal F}^s_{\theta_i(t)}(\delta_0)=\exp_{\theta_i(t)}\left({\rm Graph}(\varphi^s_{\theta_i(t)})\right)$.
	With the same assumption as Theorem \ref{thm:Lipschitz}, we can  enlarge $L$, such that for every $\tau$ satisfying $|\tau|<\tau_0$,
	\begin{enumerate}
		\item 
		Let the map $\varphi^{su}_{\theta_i(t)}:E^s_{\theta_i(t)}(\delta_1)\oplus E^u_{\theta_i(t)}(\delta_1)\rightarrow E^c_{\theta_i(t)}(\delta_1)$ be defined as
		$$
		\varphi^{su}_{\theta_i(t)}(v^s+v^u)=\pi^c_{\theta_i(t)}\circ \varphi^s_{\theta_i(t)}(v^s),
		\qquad \forall v^s+v^u\in E^s_{\theta_i(t)}(\delta_1)\oplus E^u_{\theta_i(t)}(\delta_1).
		$$
		For every $|\tau|<\tau_0$, the $C^r$-submanifold $\exp_{\theta_i(t)}\left({\rm Graph}(\varphi^{su}_{\theta_i(t)})\right)$ intersects $\Phi\circ\sigma_{\tau}((t-\delta_1,t+\delta_1))$ with a unique point $q=q(t,\tau)$.
		\item 
		If we denote
		$$
		\exp_{\theta_i(t)}^{-1}(q)=q^s+\pi^c_{\theta_i(t)}\circ \varphi^s_{\theta_i(t)}(q^s)+q^u\in 
		E^s_{\theta_i(t)}(\delta_1)\oplus E^c_{\theta_i(t)}(\delta_1)\oplus E^u_{\theta_i(t)}(\delta_1),
		$$
		then it satisfies
		$$
		\|q^u-\pi^u_{\theta_i(t)}\circ \varphi^s_{\theta_i(t)}(q^s)\|\leq L\cdot|\tau|.
		$$
		
	\end{enumerate}

\end{corollary}

\begin{proof}
	Since $\delta_1\leq\delta_0$,
	Item 2 of Lemma \ref{lem:metric} shows $\|\partial\varphi^s_{\theta_i(t)}/\partial s(v^s)\|<10^{-3}$ for every $v^s\in E^s_{\theta_i(t)}(\delta_1)$. 
	So the map $\varphi^{su}_{\theta_i(t)}$ satisfies
	$$
	\|\partial \varphi^{su}_{\theta_i(t)}/\partial s\|<10^{-3},
	\qquad {\rm and} \qquad
	\|\partial \varphi^{su}_{\theta_i(t)}/\partial u\|\equiv0.
	$$
	Theorem \ref{thm:Lipschitz} shows that $\exp_{\theta_i(t)}({\rm Graph}(\varphi^{su}_{\theta_i(t)}))$ intersects $\Phi\circ\sigma_{\tau}((t-\delta_1,t+\delta_1))$ with a unique point $q=q(t,\tau)$, and it satisfies
	$$
	\|q^s+q^u\|+\|\pi^c_{\theta_i(t)}\circ \varphi^s_{\theta_i(t)}(q^s)\|<L\cdot\tau.
	$$
	From the property of adapted metric we have chosen, we have $\|q^s\|<2L\tau$ and $\|q^u\|<2L\tau$.
	
	On other hand, we have
	$$
	\|\pi^u_{\theta_i(t)}\circ \varphi^s_{\theta_i(t)}(q^s)\|
	\leq \|\partial \varphi^s_{\theta_i(t)}/\partial s\|\cdot\|q^s\|
	\leq 10^{-3}\cdot 2L\tau.
	$$
	Thus
	$$
	\|q^u-\pi^u_{\theta_i(t)}\circ \varphi^s_{\theta_i(t)}(q^s)\|\leq
	\|q^u\|+\|\pi^u_{\theta_i(t)}\circ \varphi^s_{\theta_i(t)}(q^s)\| \leq 2(1+10^{-3})L\tau.
	$$
	So we enlarge $L$ to $2(1+10^{-3})L$ proving the corollary.
\end{proof}

\subsection{A dynamical $su$-foliation}
\label{subsec:su-foliation}

In this subsection, we introduce a dynamical $su$-foliation in the tubular neighborhood of $\bigsqcup_{i=0}^{k-1}\RR_i$ in  $\theta^*(F^s\oplus F^u)(\delta_1)$. We will define the leaf conjugacy between $\bigsqcup_{i=0}^{k-1}\RR_i$ and $\bigsqcup_{i=0}^{k-1}\sigma_{\tau}(\RR_i)$ by this $su$-foliation.

Let $F$ be a $C^\infty$ $u$-dimensional plane field which $C^0$-approximates $E^u$. We assume $0<\eta<10^{-3}$, and
$$
\measuredangle(F,E^u)=\max_{x\in M}\max_{v\in F_x,\|v\|=1}d_{T_xM}(v,E^u_x)<\eta.
$$
For every $y\in M$, we denote $D^u_y(\delta_1)=\exp_y(F_y(\delta_1))$ which is a $u$-dimensional imbedded disk in $M$, and
\begin{align}\label{Def:Dsu}
     D^{su}_z(\delta_1)=\bigcup_{y\in{\cal F}^s_z(\delta_1)}D^u_y(\delta_1) =\bigcup_{y\in{\cal F}^s_z(\delta_1)}\exp_y(F_y(\delta_1)).
\end{align}

For every $\eta>0$, we say a $C^1$-curve $\gamma:(a,b)\rightarrow M$ is tangent to the $\eta$-cone field of $E^c$, if for every $t\in(a,b)$, the vector $\gamma'(t)\in T_{\gamma(t)}M$ satisfies
$$
\|\pi^s_{\gamma(t)}(\gamma'(t))+\pi^u_{\gamma(t)}(\gamma'(t))\|
\leq\|\pi^c_{\gamma(t)}(\gamma'(t))\|.
$$

\begin{lemma}\label{lem:su-disk}
	Let $0<\eta<10^{-3}$ and $F$ be a $C^\infty$ $u$-dimensional plane field satisfying $\measuredangle(F,E^u)<\eta$. Then there exists a constant $0<\delta_2=\delta_2(M,f,F,\eta)\leq\delta_1$ satisfying the following properties:
	\begin{enumerate}
		\item For every $z_0\in M$, the set $D^{su}_{z_0}(\delta_2)=\bigcup_{y\in{\cal F}^s_{z_0}(\delta_2)}D^u_y(\delta_2)$ is a $C^r$-smooth local manifold and satisfies
		$$
		\measuredangle(T_wD^{su}_{z_0}(\delta_2),E^s_w\oplus E^u_w)<2\eta, \qquad \forall w\in D^{su}_{z_0}(\delta_2).
		$$
		Moreover, there exists a $C^r$-function $\psi^{su}_{z_0}:E^s_{z_0}(\delta_2/2)\oplus E^u_{z_0}(\delta_2/2)\rightarrow E^c_{z_0}(\delta_2/2)$ satisfies:
		\begin{itemize}
			\item $\exp^{-1}_{z_0}(D^{su}_{z_0}(\delta_2))
			\cap T_{z_0}M(\delta_2/2)
			={\rm Graph}(\psi^{su}_{z_0})
			=\left\lbrace v^{su}+\psi^{su}_{z_0}(v^{su}): 
			v^{su}\in E^s_{z_0}(\delta_2/2)\oplus E^u_{z_0}(\delta_2/2)\right\rbrace $;
			\item $\psi^{su}_{z_0}(0^{su})=0^c$, $\|\partial\psi^{su}_{z_0}/\partial s\|<2\eta$, and $\|\partial\psi^{su}_{z_0}/\partial u\|<2\eta$.
		\end{itemize}
		
		\item If $\gamma^c$ is a $C^1$-curve tangent to $E^c$ everywhere and centered at $z_0$ with radius $\delta_2$, then the set
		$$
		B^{su}_{z_0}(\gamma^c,\delta_2)=\bigcup_{z\in\gamma^c}D^{su}_z(\delta_2)
		$$
		is a neighborhood of $z_0$, and the family
		$$
		{\cal D}^{su}(\gamma^c)=\left\lbrace D^{su}_z(\delta_2):z\in\gamma^c\right\rbrace 
		$$
		is a $C^0$-foliation on $B^{su}_{z_0}(\gamma^c,\delta_2)$. For any $C^1$-curve $\gamma\subset B^{su}_{z_0}(\gamma^c,\delta_2)$ that is tangent to the $10^{-3}$-cone field of $E^c$ in $B^{su}_{z_0}(\gamma^c,\delta_2)$, the foliation ${\cal D}^{su}(\gamma^c)$ is transverse to $\gamma$ everywhere.
		
		\item For every $0<\tau<\delta_2$, there exists $0<\tau'<\delta_2$, such that for any center curve $\gamma^c$ centered at a point $z_0$ with radius $\delta_2$, if $\gamma_1$ and $\gamma_2$ are two $C^1$-curves tangent to the $10^{-3}$-cone field of $E^c$ in $B^{su}_{z_0}(\gamma^c,\delta_2)$, and $x_i,y_i$ are endpoints of $\gamma_i$ for $i=1,2$ which satisfying
		$$
		x_1,x_2\in D^{su}_{z'}(\delta_2),\qquad y_1,y_2\in D^{su}_{z''}(\delta_2),\qquad{\it and}\qquad z',z''\in\gamma^c,
		$$
		then the length $l(\gamma_1)\geq\tau$ implies $l(\gamma_2)\geq\tau'$.
	\end{enumerate}
\end{lemma}

\begin{proof}
	Recall that ${\cal F}^s_{z_0}(\delta_1)$ is a $C^r$-submanifold for every $z_0\in M$. The bundle $F$ is $C^0$-close to $E^u$, which implies $F_z$ is quasi-transverse to $E^s_z$ for every $z\in M$. Since the bundle $F$ is $C^\infty$-smooth, the set $D^{su}_{z_0}(\delta_1)$ is a $C^r$-submanifold. Moreover, we have
	$$
	T_zD^{su}_{z_0}(\delta_1)=E^s_z\oplus F_z, \qquad \forall z\in{\cal F}^s_{z_0}(\delta_1).
	$$
	From the uniform continuity, there exists $\delta_2\in(0,\delta_1]$, such that
	$$
	\measuredangle(T_zD^{su}_{z_0}(\delta_2),E^s_z\oplus E^u_z)<2\eta, \qquad \forall z\in D^{su}_{z_0}(\delta_2).
	$$
	The existence of function $\psi^{su}_{z_0}$ is a direct consequence of the property of $D^{su}_{z_0}(\delta_2)$. This proves the first item.
	
	If $\gamma^c$ is $C^1$-curve tangent to $E^c$ everywhere and centered at $z_0$ with radius $\delta_2$, then 
	$$
	{\cal F}^{cs}_{\gamma^c}(\delta_2)=\bigcup_{z\in\gamma^c}{\cal F}^s_z(\delta_2)
	$$ 
	is an imbedded $C^1$-submanifold tangent to $E^s\oplus E^c$ everywhere.
	
	Now we consider the map
	$$
	\Psi: 
	\left(F|_{{\cal F}^{cs}_{\gamma^c}(\delta_2)}\right)(\delta_2)
	=\left\lbrace v_y:~y\in{\cal F}^{cs}_{\gamma^c}(\delta_2), 
	     ~v_y\in F_y(\delta_2)\right\rbrace  
	\longrightarrow M
	$$
	which is defined as $\Psi(v_y)=\exp_y(v_y)\in M$. For every $0_y\in \left(F|_{{\cal F}^{cs}_{\gamma^c}(\delta_2)}\right)(\delta_2)$, we have
	$$
	D\Psi_{0_y}:T_y{\cal F}^{cs}_{\gamma^c}(\delta_2)\oplus F_y=E^{cs}_y\oplus F_y\longrightarrow T_yM=E^{cs}_y\oplus F_y
	$$
	satisfies $D\Psi_{0_y}=(id_{E^{cs}_y},id_{F_y})$.
	
	From the inverse function theorem and shrinking $\delta_2$ if necessary, $\Psi$ is a diffeomorphism from $\left(F|_{{\cal F}^{cs}_{\gamma^c}(\delta_2)}\right)(\delta_2)$ to its image in $M$. So $B^{su}_{z_0}(\gamma^c,\delta_2)$ is a neighborhood of $z_0$.
	Moreover, for every two points $y_1\neq y_2\in{\cal F}^{cs}_{\gamma^c}(\delta_2)$, we have
	$$
	D^u_{y_1}(\delta_2)\cap D^u_{y_2}(\delta_2)
	=\Psi(F^u_{y_1}(\delta_2))\cap\Psi(F^u_{y_1}(\delta_2))
	=\emptyset.
	$$
	So for every $z_1\neq z_2\in\gamma^c$, since ${\cal F}^s_{z_1}(\delta_2)\cap {\cal F}^s_{z_2}(\delta_2)=\emptyset$, we have
	$$
	D^{su}_{z_1}(\delta_2)\cap D^{su}_{z_2}(\delta_2)=\emptyset.
	$$
	That is ${\cal D}^{su}(\gamma^c)=\{D^{su}_z(\delta_2):z\in\gamma^c\}$ is a $C^0$-foliation on $B^{su}_{z_0}(\gamma^c,\delta_2)$. This proves the second item.
	
	Finally, we show that this local foliation is uniformly continuous. We prove it by contradiction. Assume there exist $0<\tau<\delta_2$, a sequence of $C^1$-curves $\gamma^c_n$ tangent to $E^c$ and centered at $z_n$ with radius $\delta_2$, and 
	two sequences of $C^1$-curves $\mu_n^1$ and $\mu_n^2$ tangent to $10^{-3}$-cone field in $B^{su}_{z_n}(\gamma^c_n,\delta_2)$, such that
	$$
	l(\mu_n^1)=\tau, \qquad {\rm and}\qquad l(\mu_n^2)<\frac{1}{n}.
	$$
	By taking the subsequence, we can assume that
	$$
	z_n\rightarrow z, \qquad \gamma^c_n\rightarrow\gamma^c,
	\qquad {\rm and} \qquad
	\mu_n^1\rightarrow\mu^1, \qquad  \mu_n^2\rightarrow y\in M,
	$$
	where $\{y\}\cup\mu^1\subset B^{su}_z(\gamma^c,\delta_2)$. The curve $\mu^1$ satisfies $l(\mu^1)=\tau$ and is tangent to the $10^{-3}$-cone field of $E^c$. So the curve $\mu^1$ is transverse to every $D^{su}_z(\delta_2)$ it intersects.
	
	Denote two endpoints of $\mu^1$ be $y^1$ and $y^2$. Since ${\cal D}^{su}(\gamma^c)$ is a $C^0$-foliation transverse to $E^c$ everywhere, it is also transverse to $\mu^1$, so we have $z_1\neq z_2\in\gamma^c$, such that
	$y^i\in D^{su}_{z_i}(\delta_2)$ for $i=1,2$.
	However, the point $y\in D^{su}_{z_1}(\delta_2)\cap D^{su}_{z_2}(\delta_2) $, which contradicts that ${\cal D}^{su}(\gamma^c)$ is a foliation. This proves the lemma.
\end{proof}

\begin{remark}
	The third item shows that the foliation ${\cal D}^{su}(\gamma^c)$ is uniformly continuous. The constant $l'$ depends only on $F$ and $l$, and is independent on $\gamma^c$.
	Actually, the foliation ${\cal D}^{su}(\gamma^c)$ is uniformly $C^1$-smooth in $B^{su}_{z_0}(\gamma^c,\delta_2)$. Becasuse the stable foliation ${\cal F}^s$ is uniformly $C^1$-smooth in ${\cal F}^{cs}_{\gamma^c}(\delta_2)$ and the bundle $F$ is $C^{\infty}$-smooth, see \cite{PSW1}. We won't need this property in the future. 
\end{remark}

\begin{corollary}\label{cor:su-foliation}
	Let $F$ be a $C^\infty$ $u$-dimensional plane field satisfying $\measuredangle(F,E^u)<\eta<10^{-3}$, and $\delta_2\in(0,\delta_1]$ be the constant in Lemma \ref{lem:su-disk}. For every $t\in\bigsqcup_{i=0}^{k-1}\RR_i$, we denote
	$$
	\tilde{D}^{su}_t(\delta_2)=
	\Phi^{-1}\left(D^{su}_{\theta(t)}(\delta_2)\right).
	$$
	Then
	$$
	{\cal D}^{su}(\delta_2/10)=
	\left\lbrace  \tilde{D}^{su}_t(\delta_2)\cap\theta^*(F^s\oplus F^u)(\delta_2/10):~t\in\bigsqcup_{i=0}^{k-1}\RR_i \right\rbrace 
	$$                        
	is a $C^0$-foliation of $\theta^*(F^s\oplus F^u)(\delta_2/10)$. This foliation satisfies
	\begin{enumerate}
		\item It is transverse to $(\Phi^{-1})^*(E^c)$ everywhere.
		\item For every $0<\tau<\delta_2$, there exists $0<\tau'<\delta_2$, such that if $\tilde{\gamma}$ is a $C^1$-curves tangent to the $10^{-3}$-cone field of $(\Phi^{-1})^*(E^c)$ in $\theta^*(F^s\oplus F^u)(\delta_2/10)$, and $x_1,x_2$ are endpoints of $\tilde{\gamma}$ for $i=1,2$ which satisfying
		$$
		l(\tilde{\gamma})\geq\tau \qquad {\it and} \qquad
		x_i\in\tilde{D}^{su}_{t_i}(\delta_2),~~i=1,2,
		$$
		then $|t_1-t_2|\geq\tau'$.
	\end{enumerate}

\end{corollary}

\subsection{Moving forward by global perturbations}
\label{subsec:global}

In this subsection, we prove Theorem \ref{Thm:Global-Perturb}. During this subsection, we assume $f\in{\rm PH}^r(M)$ with  one-dimensional oriented center bundle $E^c$, and $Df$ preserves the orientation of $E^c$.
Let $X\in{\cal X}^r(M)$ be a $C^r$ vector field which is positively transverse to $E^s\oplus E^u$. Let $F^s$ and $F^u$ are two $C^\infty$ bundles that are $C^0$-close to $E^s$ and $E^u$ respectively, where $TM=F^s\oplus E^c\oplus F^u$ is a direct sum. Let $\theta:\bigsqcup_{i=0}^{k-1}\RR_i\rightarrow M$ be a periodic normally hyperbolic $C^1$-leaf immersion, and
$$
\theta^*f:\theta^*(F^s\oplus F^u)(\delta_1/C_1)\rightarrow \theta^*(F^s\oplus F^u)(\delta_1/2)
$$
be the lifting bundle dynamics.

 Denote by $\tilde{X}$ the lifting vector field of $X$ on $\theta^*(F^s\oplus F^u)(\delta_1)$, and
$$
F_{\tau}=\tilde{X}_{\tau}\circ\theta^*f: \theta^*(F^s\oplus F^u)(\delta_1/C_1)\rightarrow \theta^*(F^s\oplus F^u)(\delta_1)
$$
be the lifting dynamics of $X_{\tau}\circ f:M\rightarrow M$ for every $\tau\in(-\tau_0,\tau_0)$.

We want to show that if the vector field $X$ is positively transverse to $E^s\oplus E^u$, then the perturbation generated by $X_{\tau}$ actually pushes the point moving forward, i. e.
there exists $\eta>0$, such that if $\measuredangle(F,E^u)<\eta$, for every $y\in M$ and let $D^{su}_{y}$ be the local disk defined in (\ref{Def:Dsu}), then for $\tau>0$ small enough,  we have
	$$
	F_{\tau}\big(D^{su}_{y}\big)~>~D^{su}_{f(y)},
	$$ 
in a $(L\cdot\tau)$-neighborhood of $f(y)$, where $L$ is give by Theorem \ref{thm:Lipschitz}. Here the order $``>''$ comes from the orientation of $E^c$ around $f(y)$, where $F_{\tau}\big(D^{su}_{y}\big)$ and $D^{su}_{f(y)}$ are graphs of two smooth function from $E^s_{f(y)}\oplus E^u_{f(y)}$ to $E^c_{f(y)}$.

To describe the moving forward phenomenon accurately, we need to use the language of leaf conjugacy and comparing the order of conjugacy point in bundle dynamics.

\begin{proposition}\label{Prop:Moving-forward}
	Let $X$ be a $C^r$-vector field which is positively transverse to $E^s\oplus E^u$ on $M$. There exist $\eta=\eta(M,f,X)>0$, a $C^\infty$ $u$-dimensional plane field $F$ satisfying $\measuredangle(F,E^u)<\eta$, and $\tau_1=\tau_1(M,f,X,\eta,F)>0$, such that
	\begin{itemize}
		\item let $\delta_2=\delta_2(M,f,\eta,F)>0$ be the constant in Lemma \ref{lem:su-disk} and Corollary \ref{cor:su-foliation}, and
		$$
		{\cal D}^{su}(\delta_2/10)=\left\lbrace \tilde{D}^{su}_{t}(\delta_2)\cap 
		\theta^*(F^s\oplus F^u)(\delta_2/10):
		t\in\bigsqcup_{i=0}^{k-1}\RR_i \right\rbrace 
		$$ 
		be the $su$-foliation on $\theta^*(F^s\oplus F^u)(\delta_2/10)$ defined as Corollary \ref{cor:su-foliation};
		
		\item for every $\tau\in(-\tau_1,\tau_1)$, let $\sigma_{\tau}:\bigsqcup_{i=0}^{k-1}\RR_i\rightarrow \theta^*(F^s\oplus F^u)(\delta_2/10)$ be the $F_{\tau}$-invariant section defined in Theorem \ref{thm:permanence} and Theorem \ref{thm:Lipschitz};
		
		\item let the leaf conjugacy
		$h_{\tau}:\bigsqcup_{i=0}^{k-1}\RR_i\rightarrow \bigsqcup_{i=0}^{k-1}\sigma_{\tau}(\RR_i)$
	    be defined as 
		$$
		h_{\tau}(t)=\sigma_{\tau}(\RR_i)\cap\tilde{D}^{su}_{t}(\delta_2),
		\qquad \forall t\in\RR_i.
		$$
	\end{itemize}
	We have
	\begin{enumerate}
		\item for every $0<\tau<\tau_1$, there exists $\Delta_{\tau}>0$ such that
		$$
		F_{\tau}\circ h_{\tau}(t)~>~
		h_{\tau}\left(\theta^*f(t)+\Delta_{\tau}\right),
		\qquad \forall t\in\RR_i,\quad i=0,\cdots,k-1.
		$$
		
		\item for every $-\tau_1<\tau<0$, there exists $\Delta_{\tau}>0$ such that
		$$
		F_{\tau}\circ h_{\tau}(t)~<~
		h_{\tau}\left(\theta^*f(t)-\Delta_{\tau}\right),
		\qquad \forall t\in\RR_i,\quad i=0,\cdots,k-1.
		$$
	\end{enumerate}
    Here the order on $\sigma_{\tau}(\RR_{i+1})$ inherits from $\RR_{i+1}$~$(\RR_k=\RR_0)$.
\end{proposition}

\begin{remark}
	The leaf conjugacy $h_{\tau}$ can be very flexible. For instance, if the bundles $F^s,F^u$ satisfy $\measuredangle(F^s,E^s)<\eta$ and $\measuredangle(F^u,E^u)<\eta$, then we can take 
	$$
	h_{\tau}(t)=\sigma_{\tau}(\RR_i)\cap\theta^*(F^s\oplus F^u)_t(\delta_2),
	\qquad \forall t\in\RR_i,\quad i=0,\cdots,k-1,
	$$
	This is the classical way for constructing leaf conjugacy in \cite{HPS}.
	The result of Proposition \ref{Prop:Moving-forward} still holds for this $h_{\tau}$, which is Theorem \ref{Thm:push} in the introduction.
    Here we take the special foliation ${\cal D}^{su}(\delta_2/10)$ for constructing the leaf conjugacy $h_{\tau}$ in order to prove Theorem \ref{Thm:Local-Perturb}.
\end{remark}

\begin{figure}[htbp]
	\centering
	\includegraphics[width=15cm]{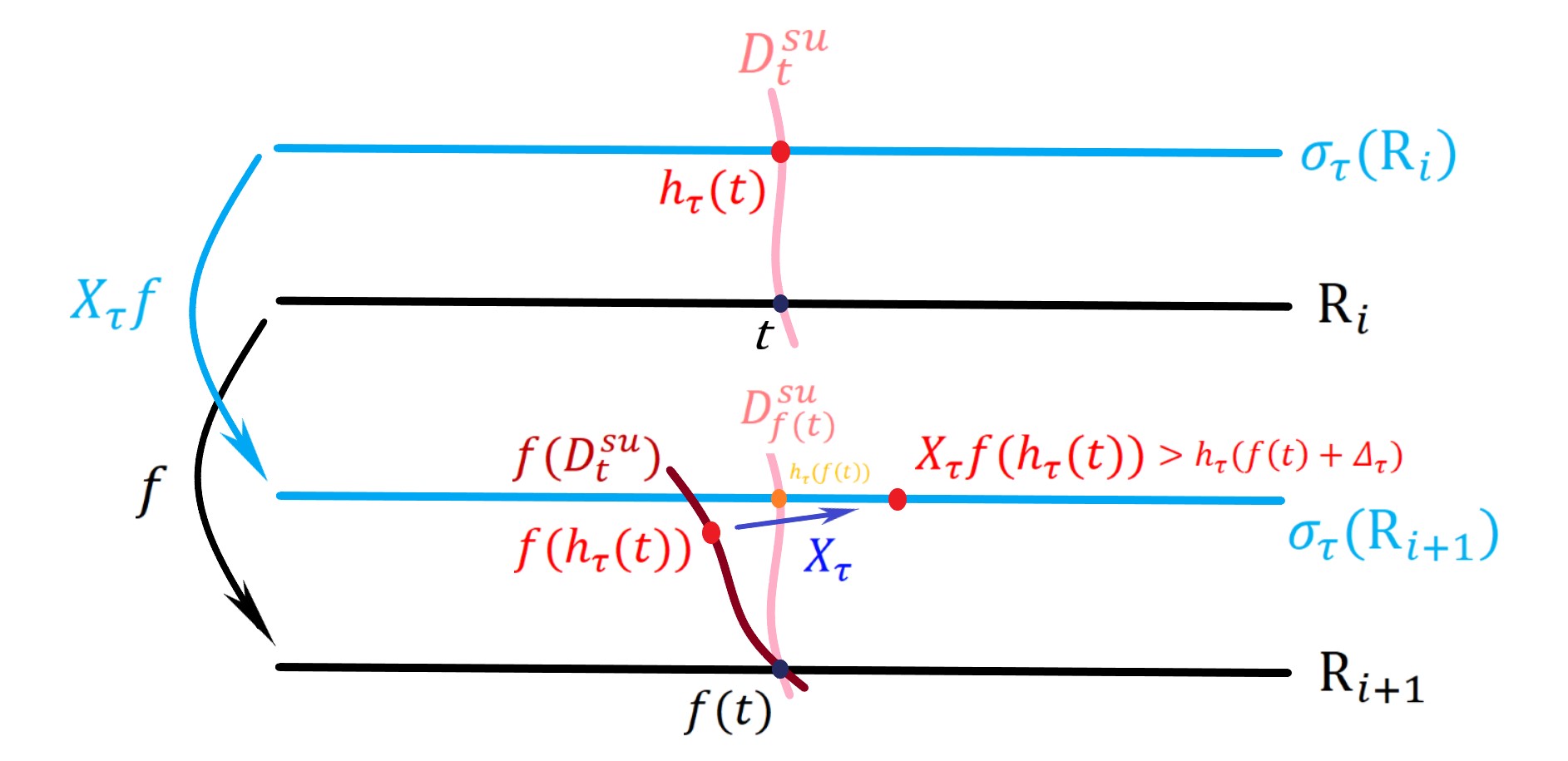}
	\caption{Moving forward by $F_{\tau}$}
\end{figure}

\begin{proof}[Proof of Proposition \ref{Prop:Moving-forward}]
	We only prove the first item with $\tau>0$. The proof of the case $\tau<0$ is the same.

	Let $z\in M$ and we consider the lifting vector field $X_z=D\exp_z^{-1}(X)$ which is locally well defined around the zero of $T_zM$. For every $v\in T_zM$ which is sufficiently close to $0$, we denote
	$$
	X_z(v)=X_z^s(v)+X_z^c(v)+X_z^u(v)\in E^s_z(v)\oplus E^c_z(v)\oplus E^u_z(v)
	$$
	and $X_z^{su}(v)=X_z^s(v)+X_z^u(v)\in E^s_z(v)\oplus E^u_z(v)$.
	We define the order in $E^c_z$ following the orientation of $E^c_z$. Thus for $0_z\in T_zM$ be the zero, we have $X_z^c(0_z)>0$ in this order.
	
	Since the vector field $X\in{\cal X}^r(M)$ is positively transverse to $E^s\oplus E^u$,  there exist three constants $\delta_3>0$, $a_0>0$ and $b_0>0$, such that for every $z\in M$,
	\begin{itemize}
		\item for every $v^{su}\in E^s_z(\delta_3)\oplus E^u_z(\delta_3)$, the curve $\exp_z^{-1}(v^{su}\times E^c(\delta_3))$ is tangent to the $10^{-3}$-cone field of $E^c$ everywhere;
		\item the lifting vector field $X_z$ satisfies, for every $v\in T_zM(\delta_3)=E^s_z(\delta_3)\oplus E^c_z(\delta_3)\oplus E^u_z(\delta_3)$,
		$$
		\|X_z^c(v)\|\geq b_0, \qquad {\rm and} \qquad
		\|X_z^s(v)\|+\|X_z^u(v)\|<a_0\|X_z^c(v)\|
		\leq 2a_0\|X\|,
		$$
		where $\|X\|=\sup_{z\in M}\|X(z)\|$.
	\end{itemize}
	Notice here these three constants only depend on $M$, $f$ and $X$.
	
	Let $\tau_0$  and $L$ be constants in Theorem \ref{thm:Lipschitz}, and we take the constant
	$$
	\eta=\min\left\lbrace 10^{-3}, 
	~\frac{b_0}{8\left(\sup_{z\in M}\|Df(z)\|\cdot L +L+2a_0\|X\|\right)} \right\rbrace
	$$
	Take a $C^\infty$ $u$-dimensional plane field $F$ satisfying $\measuredangle(F,E^u)<\eta$. Let $\delta_2=\delta_2(M,f,F,\eta)>0$ be the constant in Lemma \ref{lem:su-disk}. 
	
	Let $\delta'=10^{-2}\cdot\min\{\delta_2,\delta_3\}$, $\tau_0$ be the constant in Theorem \ref{thm:Lipschitz}, and
	$$
	\tau'=\min\left\lbrace \tau_0,~
	\frac{\delta'}{100L\cdot\|X\|\cdot \sup_{x\in M}\|Df(x)\|}  \right\rbrace .
	$$
	
	For every $t\in\RR_i$ be fixed, we denote $z_0=\theta_i(t)$.
	 If $\tau$ satisfies $|\tau|<\tau'$, then Theorem \ref{thm:Lipschitz} implies that the point 
	$
	q_{\tau}=\Phi(h_{\tau}(t))\in D^{su}_{z_0}(\delta')
	$
	satisfies
	$$
	\exp^{-1}_{z_0}(q_{\tau})
	=q_{\tau}^{su}+q_{\tau}^c
	\in E^{su}_{z_0}\oplus E^c_{z_0}, 
	\qquad {\rm and} \qquad 
	\|q_{\tau}^{su}\|
	\leq L\cdot\tau<
	\frac{\delta'}{100\cdot\|X\|\cdot\sup_{z\in M}\|Df(z)\|}.
	$$
    Moreover, the first item of Lemma \ref{lem:su-disk} implies
	$$
	\|q_{\tau}^c\|=
	\|\psi^{su}_{z_0}(q_{\tau}^{su})\|<
	2\eta\cdot\|q_{\tau}^{su}\|\leq 
	2\eta\cdot L\cdot\tau<
	\frac{2\eta\cdot\delta'}{100\cdot\|X\|\cdot\sup_{x\in M}\|Df(x)\|},
	$$
	where $\psi^{su}_{z_0}$ is the function whose graph defines $D^{su}_{z_0}(\delta')$ by the exponential map in Lemma \ref{lem:su-disk}.
	Recall that $z_0=\theta_i(t)$ and $f(z_0)=\theta_{i+1}(\theta^*f(t))$. From the definition of $\tau'$ and $|\tau|<\tau'$, we have 
	$$
	f(q_{\tau})\in B_{f(z_0)}(\delta'/10),
	\qquad {\rm and} \qquad
	X_{\tau}\circ f(q_{\tau})\in B_{f(z_0)}(\delta').
	$$
	
	If we denote
	$$
	\exp^{-1}_{f(z_0)}(f(q_{\tau}))=f(q_\tau)^{su}+f(q_\tau)^c
	\in E^{su}_{f(z_0)}\oplus E^c_{f(z_0)},
	$$
	then we have the following claim:
	
	\begin{claim}
		There exists a continuous function $\epsilon(\cdot):\RR\rightarrow\RR_+$ which only depends on $M$, $f$, and $X$, satisfying $\epsilon(\tau)\rightarrow 0$ as $\tau\rightarrow0$, such that
		\begin{itemize}
			\item $\|f(q_\tau)^{su}\|<
			\left(\|Df(z_0)\|+\epsilon(\tau)\right)\cdot
			    \|q_{\tau}^{su}\|<
			\left(\|Df(z_0)\|+\epsilon(\tau)\right)\cdot L\tau$;
			\item $\|f(q_\tau)^c\|<
			\left(\|Df|_{E^c_{z_0}}\|+\epsilon(\tau)\right)\cdot
			    \|q_{\tau}^c\|<
			2\eta\cdot\left(\|Df|_{E^c_{z_0}}\|+\epsilon(\tau)\right)
			    \cdot L\tau$.
		\end{itemize}
		In particular, this implies
		$$
		f(q_\tau)^c>
		-2\eta\cdot\left(\|Df|_{E^c_{z_0}}\|+\epsilon(\tau)\right)
		\cdot L\tau.
		$$
    \end{claim}	
    
    \begin{proof}[Proof of the Claim]
    	It is clear that
    	$$
    	f(q_\tau)^{su}+f(q_\tau)^c=\exp_{f(z_0)}^{-1}\circ f\circ\exp_{z_0}(q_{\tau}^{su}+q_{\tau}^c): T_{z_0}M\longrightarrow T_{f(z_0)}M.
    	$$
    	Let $0_{z_0}$ and $0_{f(z_0)}$ be origins of $T_{z_0}M$ and $T_{f(z_0)}M$ respectively. The derivative of the local diffeomorphism $\exp_{f(z_0)}^{-1}\circ f\circ\exp_{z_0}$ satisfies
    	$$
    	D\left(\exp_{f(z_0)}^{-1}\circ f\circ\exp_{z_0}\right)(0_{z_0})
    	=Df(z_0):~
    	T_{0_{z_0}}\left(T_{z_0}M\right)\cong T_{z_0}M
    	\longrightarrow
    	T_{0_{f(z_0)}}\left(T_{f(z_0)}M\right)\cong T_{f(z_0)}M.
    	$$
    	
    	Since $q_{\tau}$ uniformly converges to $z_0$ as $\tau\rightarrow 0$, we have
    	$$
    	f(q_\tau)^{su}=(Df|_{E^{su}_{z_0}})(q_{\tau}^{su})+ o(\|q_{\tau}^{su}\|+\|q_{\tau}^c\|),
    	\qquad{\rm and}\qquad
    	f(q_\tau)^c=(Df|_{E^c_{z_0}})(q_{\tau}^c)+ o(\|q_{\tau}^{su}\|+\|q_{\tau}^c\|).
    	$$
    	Here $o(\|q_{\tau}^{su}\|+\|q_{\tau}^c\|)/(\|q_{\tau}^{su}\|+\|q_{\tau}^c\|)$ uniformly converges to $0$ as $\tau\rightarrow0$. The facts that $\|q_{\tau}^c\|<2\eta\|q_{\tau}^{su}\|$ and $\|q_{\tau}^{su}\|<L\tau$ imply the claim.
    \end{proof}
    
	Now we consider the action of $X_{f(z_0),\tau}$ on the point $(f(q_\tau)^{su}+f(q_\tau)^c)$. Denote 
	$$
	X_{f(z_0),\tau}\left(f(q_\tau)^{su}+f(q_\tau)^c\right)
	=\exp_{f(z_0)}^{-1}\left(\Phi\circ F_{\tau}(h_{\tau}(t))\right)
	=p_{\tau}^{su}+p_{\tau}^c\in 
	E^{su}_{f(z_0)}\oplus E^c_{f(z_0)}.
	$$
	For every $\tau>0$, we have
	\begin{itemize}
		\item $\|p_{\tau}^{su}\|\leq
		  \|f(q_\tau)^{su}\|+\sup_{v\in T_{f(z_0)}M(\delta_3)}
		  \left(\|X^s_{f(z_0)}(v)\|+\|X^u_{f(z_0)}(v)\|\right)
		   \cdot\tau<
		  \|f(q_\tau)^{su}\|+2a_0\|X\|\cdot\tau$;
		
		\item $p_{\tau}^c\geq -\|f(q_\tau)^c\|+\inf_{v\in T_{f(z_0)}M(\delta_3)}\|X^c_{f(z_0)}(v)\|\cdot\tau
		\geq -\|f(q_\tau)^c\|+b_0\cdot\tau$.
	\end{itemize}
	This implies
	$$
	\|p_{\tau}^{su}\|\leq
	\left[\left(\|Df(z_0)\|+\epsilon(\tau)\right)\cdot L+2a_0\|X\|\right]\cdot\tau,
	\qquad {\rm and} \qquad
	p_{\tau}^c>b_0\cdot\tau-
	 2\eta\left[\cdot\left(\|Df|_{E^c_{z_0}}\|+\epsilon(\tau)\right)
	  \cdot L\right]\cdot\tau.
	$$
	
	There exists $0<\tau_1<\tau'$, such that for every $|\tau|\leq\tau_1$, it satisfies
	$$
	0\leq\epsilon(\tau)<1,
	\qquad {\rm and} \qquad
	\left[\left(\sup_{z\in M}\|Df(z)\|+1\right)\cdot L+2a_0\|X\|\right]\cdot|\tau|<\delta'.
	$$
	Let $D^{su}_{f(z_0)}(\delta')$ be the disk defined by $F$ as Lemma \ref{lem:su-disk}, and $\psi^{su}_{f(z_0)}$ be the function whose graph defines $D^{su}_{f(z_0)}(\delta')$ by the exponential map.  
	Since 
	$$
	\eta\leq 
	\frac{b_0}{8\left(\sup_{z\in M}\|Df(z)\|\cdot L
		+L+2a_0\|X\|\right)},
	\qquad {\rm and} \qquad |\tau|<\tau_1,
	$$
	we have
	\begin{align*}
	    p_{\tau}^c-\psi^{su}_{f(z_0)}(p_{\tau}^{su})
	    &>b_0\cdot\tau
	     -2\eta\left[\cdot\left(\|Df|_{E^c_{z_0}}\|+1\right)\cdot L\right]\cdot\tau-2\eta\cdot\|p_{\tau}^{su}\|  \\
	    &>b_0\cdot\tau
	     -2\eta\cdot\left[2\left(\sup_{z\in M}\|Df(z)\|+1\right)\cdot L+2a_0\|X\|\right]\cdot\tau \\
	    &> b_0\cdot\tau-\frac{b_0}{2}\cdot\tau= \frac{b_0}{2}\cdot\tau>0.
	\end{align*}
	This implies the point 
	$$
	\exp_{f(z_0)}\left(p_{\tau}^{su}+p_{\tau}^c\right)
	 =X_{\tau}\circ f(h_{\tau}(t))
	 =\Phi(F_{\tau}(h_{\tau}(t)))
	$$
	belongs to $D^{su}_{\theta_{i+1}(t')}$ for some $t'\in\RR_{i+1}$ with $t'>\theta^*f(t)$.
	
	Moreover, since the segment 
	$$
	\exp_{f(z_0)}\left(\left\lbrace  p_{\tau}^{su}+v^c:~\varphi^{su}_{f(z_0)}(p_{\tau}^{su})\leq v^c\leq p_{\tau}^c\right\rbrace\right)
	$$
	is tangent to the $10^{-3}$-cone field of $E^c$ with length uniformly bounded from zero, the second item of Corollary \ref{cor:su-foliation} shows that there exists $\Delta_{\tau}>0$ satisfying $t'-\theta^*f(t)>\Delta_{\tau}$. So we have
	$$
	F_{\tau}\circ h_{\tau}(t)~>~
	h_{\tau}\left(\theta^*f(t)+\Delta_{\tau}\right).
	$$
\end{proof}

\begin{remark}
	Here the constant $\Delta_{\tau}$ is independent of the choice of periodic center curves $\theta:\bigsqcup_{i=0}^{k-1}\RR_i\rightarrow M$ and its period $k$.
\end{remark}

Applying this proposition, we can prove Theorem \ref{Thm:Global-Perturb}.

\begin{proof}[{\bf Proof of Theorem \ref{Thm:Global-Perturb}}]
	For every $x\in\Omega(f)$, if $x\in{\rm Per}(f)$, then we can choose $\tau_n\equiv0$ and $p_n\equiv x$. Otherwise, for the vector field $X\in{\cal X}^r(M)$ which is positively transverse to $E^s\oplus E^u$, let $F$ be the $C^{\infty}$ plane field on $M$ and $\tau_1$ be the constant both decided by Proposition \ref{Prop:Moving-forward}. Then for every $n>1/\tau_1$, let $\Delta_{1/n}$ be the number associated to $\tau=1/n$ as Proposition \ref{Prop:Moving-forward}.
	
	For $x\in\Omega(f)\setminus{\rm Per}(f)$,
	Proposition \ref{prop:center-curve} implies there exists a complete $C^1$-curve $\theta:\RR\rightarrow M$ tangent to $E^c$ everywhere, which is $k$-periodic for some $k>0$, and satisfies
	$$
	y=\theta(0)\in B_x(1/n), \qquad {\rm and} \qquad
	f^k(y)=\theta(t_0)~~{\rm with}~~|t_0|<\Delta_{1/n}.
	$$
	If $t_0=0$, then $y$ is a periodic point. We assume $-\Delta_{1/n}<t_0<0$. The proof for $0<t_0<\Delta_{1/n}$ is the same.
	
	Now we consider the fiber bundle dynamics $\theta^*f$ and its perturbation $F_{1/n}=\tilde{X}_{1/n}\circ\theta^*f$ defined on $\theta^*(F^s\oplus F^u)(\delta_1/C_1)$. Let $h_{1/n}:\bigsqcup_{i=0}^{k-1}\RR_i\rightarrow \bigsqcup_{i=0}^{k-1}\sigma_{1/n}(\RR_i)$ be the leaf conjugacy defined as Proposition \ref{Prop:Moving-forward}. Then
	for $0\in\RR=\RR_0$,  Proposition \ref{Prop:Moving-forward} shows that on $\sigma_{1/n}(\RR_0)$,
	\begin{align*}
	        F_{1/n}^k\left(h_{1/n}(0)\right)
	        &>F_{1/n}^{k-1}
	          \left(h_{1/n}\left(\theta^*f(0)\right)\right)
	         >\cdots\cdots>
	         F_{1/n}
	          \left(h_{1/n}\left(\theta^*f^{k-1}(0)\right)\right) \\
	        &>h_{1/n}\left(\theta^*f^k(0)+\Delta_{1/n}\right)
	         =h_{1/n}(t_0+\Delta_{1/n}) \\
	        &>h_{1/n}(0).
	\end{align*}
	
	Recall that $F_{0}=\theta^*f$ and $\theta^*f^k(0)<0$. Since $\sigma_{\tau}$ vary continuously with respect to $\tau$, there exists $\tau_n\in(0,1/n)$, such that
	$F_{\tau_n}^k(0)=0$. Let
	$p_n=\Phi\left(h_{\tau_n}(0)\right)$, then we have
	$$
	\left(X_{\tau_n}\circ f\right)^k(p_n)=p_n.
	$$
	Moreover, Theorem \ref{thm:Lipschitz} shows that $d(y,p_n)<L\cdot\tau_n$, which implies $d(x,p_n)<(L+1)/n$. Let $n$ tend to infinity and we prove the theorem. 
\end{proof}

\subsection{Local dominated dynamics}\label{subsec:domination}

We prove Theorem \ref{Thm:Local-Perturb} in Section \ref{subsec:domination} and \ref{subsec:local}. The difficulty for proving Theorem \ref{Thm:Local-Perturb} is that we don't have the uniformly moving forward estimation in Proposition \ref{Prop:Moving-forward} by the local perturbation. When the orbit goes through the region where  the vector field vanishes, the leaf conjugacy point $F_{\tau}(h_{\tau}(t))=f(h_{\tau}(t))$ may move backward comparing with $h_{\tau}(f(t))$. So we need an estimation how far the point $F_{\tau}(h_{\tau}(t))$ move backward.

In this subsection, we prove Lemma \ref{Lem:Doimination}, which shows that the distance of $F_{\tau}(h_{\tau}(t))$ moving backward in the center direction is controlled by the distance between conjugacy point $h_{\tau}(t)$ and $t$ in the unstable direction. This is the key observation for proving Theorem \ref{Thm:Local-Perturb}. During this subsection, we only assume $f\in{\rm PH^r}(M)$ with ${\rm dim}E^c=1$.

Recall that for every $z\in M$, the function $\varphi_z^s:E^s_z(\delta_0)\rightarrow E^c_z(\delta_0)\oplus E^u_z(\delta_0)$ defines the local stable manifold of $z$ as ${\cal F}^s_{z}(\delta_0)=\exp_z\left({\rm Graph}(\varphi_z^s)\right)$.
Lemma \ref{lem:local-stable} shows that $\varphi_z^s(0^s_z)=0^{cu}_z$, and for every $\eta>0$, there exists $0<\delta_{\eta}\leq\delta_0$, such that $\|\partial\varphi_z^s/\partial s(v^s)\|<\eta$ for every $v^s\in E^s_z(\delta_\eta)$. 

Moreove, let the map $\varphi^{su}_z:E^s_z(\delta_0)\oplus E^u_z(\delta_0)\rightarrow E^c_z(\delta_0)$ be defined as
$$
\varphi^{su}_z(v^{su})
=\pi^c_z\circ \varphi^s_z\circ\pi^s_z(v^{su}),
\qquad \forall v^{su}\in E^s_z(\delta_0)\oplus E^u_z(\delta_0).
$$
From the definition of $\varphi^{su}_z$, it satisfies 
$$
\varphi^{su}_z(0^{su}_z)=0^c_z,
\qquad
\|\partial\varphi^{su}_z/\partial u\|\equiv0,
\qquad {\rm and} \qquad 
\|\partial\varphi^{su}_z/\partial s(v^{su})\|\equiv\|\partial\varphi_z^s/\partial s(\pi^s_z(v^{su}))\|.
$$

The following lemma is the key for proving Theorem \ref{Thm:Local-Perturb}. It is deduced from the partial hyperbolicity of $f$, and the action of $Df|_{E^u}$ dominates $Df|_{E^c}$.

\begin{lemma}\label{Lem:Doimination}
	There exist two constants $\delta_4>0$ and $\eta_1>0$, which both depend only on $M$ and $f$, satisfying the following properties:
	\begin{enumerate}
	  \item For every $y\in M$, let
		$$
		\gamma_y:E^c_y(\delta_4)\rightarrow E^s_y(\delta_4)\oplus E^u_y(\delta_4), 
		\qquad {\it and} \qquad
		\gamma_{f(y)}:E^c_{f(y)}(\delta_4)\rightarrow E^s_{f(y)}(\delta_4)\oplus E^u_{f(y)}(\delta_4)
		$$
		be two $C^1$-maps which satisfy $\|\partial\gamma_y/\partial c\|<\eta_1$ and $\|\partial\gamma_{f(y)}/\partial c\|<\eta_1$, then
		\begin{itemize}
			\item the graph of $\gamma_y$ intersect with the graph of $\varphi^{su}_y$ with a unique point $w_1$, i.e.
			there exist a unique point $w_1=w_1^s+w_1^c+w_1^u\in T_yM(\delta_4)$ and a unique vector $v_1^u\in E^u_y(\delta_4)$, such that
			$$
			w_1^s+w_1^u=\gamma_y(w_1^c), \qquad {\it and} \qquad
			w_1^c+w_1^u=\varphi_y^s(w_1^s)+v_1^u;
			$$
			\item the graph of $\gamma_{f(y)}$ intersect with the graph of $\varphi^{su}_{f(y)}$ with a unique point $w_2$, 
			i.e.
			there exist a unique point $w_2=w_2^s+w_2^c+w_2^u\in T_{f(y)}M(\delta_4)$ and a unique vector $v_2^u\in E^u_{f(y)}(\delta_4)$, such that
			$$
			w_2^s+w_2^u=\gamma_{f(y)}(w_2^c), \qquad {\it and} \qquad
			w_2^c+w_2^u=\varphi_{f(y)}^s(w_2^s)+v_2^u.
			$$
		\end{itemize}

	  \item We fix two orientations of $E^c_y$ and $E^c_{f(y)}$, such that $Df(y)$ preserves the orientations. The orders on $E^c_y$ and $E^c_{f(y)}$ are defined by these orientations.
	  For every $0<\theta\leq1$, let $z_1=z_1^s+z_1^c+z_1^u\in T_yM(\delta_4)$ which is defined as 
	    $$
	    z_1^c=w_1^c-\theta\eta_1\cdot\|v_1^u\| =w_1^c-\theta\eta_1\cdot\|w_1^u-\pi^u_y\circ\varphi^s_y(w_1^s)\|,
	    \qquad {\it and} \qquad
	    z_1^s+z_1^u=\gamma_y(z_1^c).
	    $$
	    If we assume the point $\exp_{f(y)}^{-1}\circ f\circ\exp_y(z_1)$ is contained in $T_{f(y)}M(\delta_4)$, and denote $\exp_{f(y)}^{-1}\circ f\circ\exp_y(z_1)=f(z_1)^s+f(z_1)^c+f(z_1)^u$, then it satisfies
	    \begin{align*}
	      f(z_1)^c-\pi^c_{f(y)}\circ\varphi^s_{f(y)}(f(z_1)^s)
	       &\geq -\theta\eta_1\cdot\frac{1-\eta_1}{1+\eta_1}\cdot \left\|f(z_1)^u-  
	         \pi^u_{f(y)}\circ\varphi^s_{f(y)}(f(z_1)^s)\right\| \\
	       &\geq -\theta\eta_1\cdot \|Df(y)\|\cdot\|v_1^u\|.
	    \end{align*}

	  \item Furthermore, let $X_{f(y)}$ be a vector field defined on  
	    $T_{f(y)}M(\delta_4)$ which satisfies $X^c_{f(y)}(v)\geq0$, and
	    $$
	    \|X^s_{f(y)}(v)\|+\|X^u_{f(y)}(v)\|\leq3\eta_1\cdot X^c_{f(y)}(v), 
	    \qquad\forall v\in T_{f(y)}M(\delta_4).
	    $$
	    Here
	    $X_{f(y)}(v)=X_{f(y)}^s(v)+X_{f(y)}^c(v)+X_{f(y)}^u(v)\in E^s_{f(y)}(v)\oplus E^c_{f(y)}(v)\oplus E^u_{f(y)}(v)$ is the canonical decomposition.
	    If for some $\tau>0$, the point 
	    $$
	    z_2=z_2^s+z_2^c+z_2^u=X_{f(y),\tau}\circ\exp_{f(y)}^{-1}\circ f\circ\exp_y(z_1)
	    $$ 
	    is also contained in $T_{f(y)}M(\delta_4)$, and satisfies $z_2^s+z_2^u=\gamma_{f(y)}(z_2^c)$, then we have
	    $$
	    z_2^c\geq w_2^c-\theta\eta_1\|v_2^u\| 
	      =w_2^c-\theta\eta_1
	     \left\|w_2^u-\pi^u_{f(y)}\circ\varphi^s_{f(y)}(w_2^s)\right\|.
	    $$
	\end{enumerate}	
\end{lemma}

\begin{remark}
	Notice here that the  vector field $X_{f(y)}$ may vanish at some region of $T_{f(y)}M(\delta_2)$ or vanish at the whole $T_{f(y)}M(\delta_2)$.
\end{remark}

\begin{proof}
	Let 
	  $$
	  \eta_1=\min\left\lbrace 
	   \inf_{z\in M}\frac{\|Df|_{E^s_z}\|}{1000}, ~
	   \sqrt[4]{\inf_{z\in M}\frac{m(Df|_{E^u_z})}{\|Df|_{E^c_z}\|}}-1
	   \right\rbrace 
	  $$
	and choose $\delta_4$ small enough such that $\|\partial\varphi^s_z/\partial s(v^s)\|<\eta_1$ for every $z\in M$ and $v^s\in E^s_z(\delta_4)$. This imples
	$$
	\|\partial\varphi^{su}_z/\partial s(v^{su})\|<\eta_1,
	\qquad \forall z\in M, \quad \forall v^{su}\in E^s_z(\delta_4)\oplus E^u_z(\delta_4).
	$$ 
	
	The graph of $\gamma_y$ is contained in $T_yM(\delta_4)$ and tangent to the $\eta_1$-cone field of $E^c_y$ everywhere. So the transversality between $E^c_y$ and $E^s_y\oplus E^u_y$ implies there exists a unique point $w_1=w_1^s+w_1^c+w_1^u\in{\rm Graph}(\psi^{su}_y)\cap{\rm Graph}(\gamma_y)$. Thus we have $w_1^s+w_1^u=\gamma_y(w_1^c)$ and $w_1^c=\pi^c_y\circ\varphi^s_y(w_1^s)$. The vector $v_1^u=w_1^u-\pi^u_y\circ\varphi^s_y(w_1^s)$. The proof for $w_2$ is the same. This proves the first item of the lemma. 
	
	\vskip2mm

    We begin to prove the second item, and we will shrink $\delta_4$ during the proof.
	Let $0<\theta\leq1$ be fixed, and $z_1=z_1^s+z_1^c+z_1^u\in T_yM(\delta_4)$ be the point satisfying
	$$
	z_1^c=w_1^c-\theta\eta_1\cdot\|v_1^u\| =w_1^c-\theta\eta_1\cdot\|w_1^u-\pi^u_y\circ\varphi^s_y(w_1^s)\|,
	\qquad {\it and} \qquad
	z_1^s+z_1^u=\gamma_y(z_1^c).
	$$

	Considering the point $\exp_{f(y)}^{-1}\circ f\circ\exp_y(z_1)$, Lemma \ref{lem:local-stable} shows that
	$$
	\exp_{f(y)}^{-1}\circ f\circ\exp_y\left( \left\lbrace v^s+\varphi^s_y(v^s):v^s\in E^s_y(\delta_4)\right\rbrace \right) 
	\subseteq \left\lbrace v^s+\varphi^s_{f(y)}(v^s):v^s\in E^s_{f(y)}(\lambda\delta_4)\right\rbrace .
	$$ 
	If we denote $r_1=r_1^s+r_1^c+r_1^u=w_1-v_1^u$, then it satisfies $r_1^s=w_1^s$, $r_1^c=w_1^c$, and  $r_1^c+r_1^u=\varphi^s_y(r_1^s)$. Moreover, the point
	$$
	r_2=r_2^s+r_2^c+r_2^u=\exp_{f(y)}^{-1}\circ f\circ\exp_y(r_1)
	$$
	satisfies $\|r_2^s\|<\lambda\delta_4$ and $r_2^c+r_2^u=\varphi_{f(y)}^s(r_2^s)$. 
	
	The lemma assumed that $\exp_{f(y)}^{-1}\circ f\circ\exp_y(z_1)=f(z_1)^s+f(z_1)^c+f(z_1)^u\in T_{f(y)}M(\delta_4)$. We have the following claim.
	
	\begin{claim}\label{claim:Df}
		There exists a function $\varepsilon(\cdot):\RR_+\rightarrow\RR_+$ which only depends on $M$ and $f$, satisfying $\varepsilon(\delta_4)\rightarrow 0$ as $\delta_4\rightarrow 0$, such that 
		\begin{itemize}
			\item $f(z_1)^c-r_2^c\geq -\|Df|_{E^c_y}\|\cdot(\theta\eta_1\|v_1^u\|)-\varepsilon(\delta_4)\cdot \|v_1^u\|$;
			\item $\|f(z_1)^s-r_2^s\|\leq 
			\|Df|_{E^s_y}\|\cdot(\theta\eta_1^2\|v_1^u\|)
			+ \varepsilon(\delta_4)\cdot\|v_1^u\|$;
			\item $(1-\theta\eta_1^2)m(Df|_{E^u_y})\cdot\|v_1^u\|-
			\varepsilon(\delta_4)\cdot\|v_1^u\|\leq \|f(z_1)^u-r_2^u\|
			\leq(1+\theta\eta_1^2)\|Df|_{E^u_y}\|\cdot\|v_1^u\|+
			\varepsilon(\delta_4)\cdot\|v_1^u\|$.
		\end{itemize} 
	\end{claim}

    \begin{proof}[Proof of the Claim]
    	From the definition of points $r_1,w_1,z_1$, and $\|\partial\gamma_y/\partial c\|<\eta_1$, we have
    	$$
    	z_1^c-r_1^c=-\theta\eta_1\|v_1^u\|, \qquad 
    	\|z_1^s-r_1^s\|\leq\theta\eta_1^2\|v_1^u\|, \qquad
    	\|v_1^u\|-\theta\eta_1^2\|v_1^u\|\leq \|z_1^u-r_1^u\|\leq
    	\|v_1^u\|+\theta\eta_1^2\|v_1^u\|.
    	$$
    	In particular, we have $\|z_1-r_1\|<10\|v_1^u\|$.
    	
    	Since
    	$
    	D(\exp_{f(y)}^{-1}\circ f\circ\exp_y)(0_y)=Df_y:
    	T_{0_y}(T_yM)\cong T_yM\longrightarrow
    	T_{0_{f(y)}}(T_{f(y)}M)\cong T_{f(y)}M,
    	$
    	there exists a function $\varepsilon(\cdot):\RR_+\rightarrow\RR_+$ which satisfies $\varepsilon(\delta_4)\rightarrow 0$ as $\delta_4\rightarrow 0$, such that for every $y_1,y_2\in T_yM(\delta_4)$, and every $i=s,c,u$, we have
    	$$
    	\left\|\pi^i_{f(y)} 
    	\left( \exp_{f(y)}^{-1}\circ f\circ\exp_y(y_1)- \exp_{f(y)}^{-1}\circ f\circ\exp_y(y_2)\right) - (Df|_{E^i_y})(\pi^i_y(y_1-y_2))\right\|
    	\leq\frac{\varepsilon(\delta_4)}{10}\|y_1-y_2\|.
    	$$
    	The function $\varepsilon(\cdot)$ only depends on $M$ and $f$. Thus we apply this estimation to $z_1$ and $r_1$, which deduces the claim.    	 
    \end{proof}
	
	Now we shrink $\delta_4$ sufficiently small, such that for every $0<\delta\leq\delta_4$, it satisfies
	$$
	0<\varepsilon(\delta)<10^{-2}\cdot\eta_1^4.
	$$
	Then from this claim, we have the following estimation about the position of $\exp_{f(y)}^{-1}\circ f\circ\exp_y(z_1)$, which proves the second item of this lemma.

	\begin{figure}[htbp]
		\centering
		\includegraphics[width=15cm]{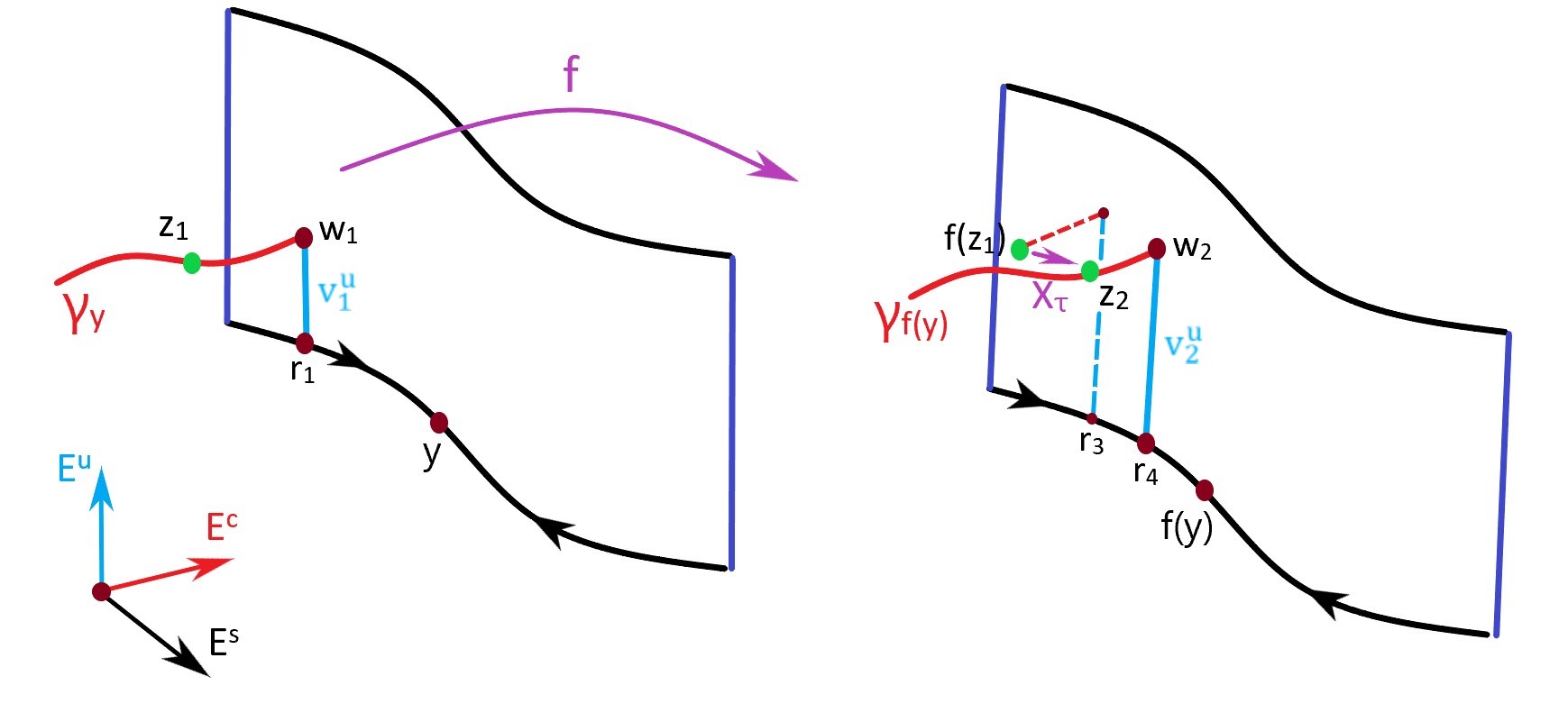}
		\caption{Domination of local action}
	\end{figure}

    \begin{claim}\label{claim:r3}
    	Let the point $r_3=r_3^s+r_3^c+r_3^u\in T_{f(y)}M(\delta_4)$ be defined as
    	$r_3^s=f(z_1)^s$ and $r_3^c+r_3^u=\varphi^s_{f(y)}(r_3^s)$. That is
    	$$
    	r_3^c=\pi^c_{f(y)}\circ\varphi^s_{f(y)}(f(z_1)^s)
    	\qquad {\it and } \qquad
    	r_3^u=\pi^u_{f(y)}\circ\varphi^s_{f(y)}(f(z_1)^s).
    	$$
    	Then we have the following estimation
    	$$
    	f(z_1)^s-r_3^c\geq
    	-\theta\eta_1\cdot\frac{1-\eta_1}{1+\eta_1}\cdot
    	\|f(z_1)^u-r_3^u\|\geq -\theta\eta_1\cdot \|Df(y)\|\cdot\|v_1^u\|.
    	$$ 
    \end{claim}

    \begin{proof}[Proof of the Claim]
    	From the definition  $r_3^s=f(z_1)^s$ and Claim \ref{claim:Df}, we have
    	$$
        \|r_3^s-r_2^s\|\leq\|Df|_{E^s_y}\|\cdot(\theta\eta_1^2\|v_1^u\|)
    	+ \varepsilon(\delta_4)\cdot\|v_1^u\|.
    	$$
    	Since both $r_2$ and $r_3$ are on the graph of $\varphi^s_{f(y)}$, and $\|\partial\varphi^s_{f(y)}/\partial s\|<\eta_1$, we have
    	$$
    	\|r_2^c-r_3^c\|<
    	\left(\theta\eta_1^3\|Df|_{E^s_y}\|
    	  +\eta_1\varepsilon(\delta_4)\right)\cdot\|v_1^u\|, 
    	\qquad{\rm and} \qquad
    	\|r_2^u-r_3^u\|<
    	\left(\theta\eta_1^3\|Df|_{E^s_y}\|
    	  +\eta_1\varepsilon(\delta_4)\right)\cdot\|v_1^u\|.
    	$$
    	
    	This implies
    	\begin{align*}
    	   \|f(z_1)^u-r_3^u\|&\leq\|f(z_1)^u-r_2^u\|+\|r_2^u-r_3^u\|\\
    	    &\leq \left(1+\theta\eta_1^2\right)\|Df|_{E^u_y}\|\cdot\|v_1^u\|
    	      +\varepsilon(\delta_4)\cdot\|v_1^u\|+
    	      \left(\theta\eta_1^3\|Df|_{E^s_y}\|+
    	        \eta_1\varepsilon(\delta_4)\right)\cdot\|v_1^u\| \\
    	    &\leq\left(1+\theta\eta_1^2+\theta\eta_1^3
    	        +\eta_1^4+\eta_1^5\right) 
    	        \cdot\|Df(y)\|\cdot\|v_1^u\|   \\
    	    &\leq\frac{1+\eta_1}{1-\eta_1}\cdot\|Df(y)\|\cdot\|v_1^u\|.
    	\end{align*}
    	
    	We apply the estimation between $r_2$ and $\exp_{f(y)}^{-1}\circ f\circ\exp_y(z_1)$ in Claim \ref{claim:Df}. Since $\eta_1\leq10^{-3}\|Df|_{E^s_y}\|$ and $\varepsilon(\delta_4)<10^{-2}\eta_1^4$, we have
    	\begin{align*}
           	f(z_1)^c-r_3^c 
           &\geq f(z_1)^c-r_2^c-\|r_2^c-r_3^c\| \\	
           &\geq -\|Df|_{E^c_y}\|\cdot(\theta\eta_1\|v_1^u\|) -\varepsilon(\delta_4)\cdot \|v_1^u\|-
            \left(\theta\eta_1^3\|Df|_{E^s_y}\|+
             \eta_1\varepsilon(\delta_4)\right)\cdot\|v_1^u\| \\
           &\geq-\left(\theta\eta_1\cdot\|Df|_{E^c_y}\|+
            (\theta+1)\eta_1^3\|Df|_{E^s_y}\|\right)
            \cdot\|v_1^u\| \\
           &\geq-\theta\eta_1(1+\eta_1)\|Df|_{E^c_y}\|\cdot\|v_1^u\|
    	\end{align*}
        and
        \begin{align*}
           \|f(z_1)^u-r_3^u\|  
           &\geq \|f(z_1)^u-r_2^u\|-\|r_2^u-r_3^u\| \\
           &\geq m(Df|_{E^u_y})\cdot
             \left(\|v_1^u\|-\theta\eta_1^2\|v_1^u\|\right)-
             \varepsilon(\delta_4)\cdot\|v_1^u\|-
             \left(\theta\eta_1^3\|Df|_{E^s_y}\|+
                \eta_1\varepsilon(\delta_4)\right)\cdot\|v_1^u\| \\
           &\geq \left[m(Df|_{E^u_y})\left(1-(1+\theta)\eta_1^2\right)
               -2\varepsilon(\delta_4)\right]\cdot\|v_1^u\| \\
           &\geq (1-\eta_1)m(Df|_{E^u_y})\cdot\|v_1^u\|
        \end{align*}
        
    	Recall that we have $0<\eta_1\leq\sqrt[4]{\inf_{z\in M}[m(Df|_{E^u_z})/\|Df|_{E^c_z}\|]}-1$, which implies
    	$$
    	\frac{\|Df|_{E^c_y}\|\cdot(1+\eta_1)}{m(Df|_{E^u_y})(1-\eta_1)}\leq \frac{1-\eta_1}{1+\eta_1}.
    	$$
    	Thus we have
    	$$
    	f(z_1)^c-r_3^c\geq
    	-\theta\eta_1\cdot\frac{1-\eta_1}{1+\eta_1}\cdot
    	\|f(z_1)^u-r_3^u\|\geq -\theta\eta_1\cdot \|Df(y)\|\cdot\|v_1^u\|.
    	$$
    \end{proof} 
	
	\vskip 2mm

	Now we prove the third item of the lemma. Let $X$ be the vector field and consider the point $z_2=X_{f(y),\tau}(\exp_{f(y)}^{-1}\circ f\circ\exp_y(z_1))$. Since we have assumed that
	$$
	\|X^s_{f(y)}(v)\|+\|X^u_{f(y)}(v)\|\leq3\eta_1\cdot X^c_{f(y)}(v), 
	\qquad\forall v\in T_{f(y)}M(\delta_4),
	$$
	it implies
	$$
	   z_2^c-f(z_1)^c\geq 
	   \frac{1}{3\eta_1}\cdot\left( \|z_2^s-f(z_1)^s\|+
	   \|z_2^u-f(z_1)^u\|\right).
	$$
	
	\begin{claim}\label{claim:r4}
		Let the point $r_4=r_4^s+r_4^c+r_4^u\in T_{f(y)}M(\delta_4)$ be defined as
		$r_4^s=z_2^s$ and
		$r_4^c+r_4^u=\varphi^s_{f(y)}(r_4^s)$. Then it satisfies
		$$
		z_2^c-r_4^c\geq-\theta\eta_1\cdot\frac{1-\eta_1}{1+\eta_1}\cdot
		\|z_2^u-r_4^u\|.
		$$
	\end{claim}
	
	\begin{proof}[Proof of the Claim]

		In the center direction, it satisfies
		\begin{align*}
		z_2^c-r_4^c =& (z_2^c-f(z_1)^c)+(f(z_1)^c-r_3^c)+(r_3^c-r_4^c) \\
		\geq& (z_2^c-f(z_1)^c)+(f(z_1)^c-r_3^c)-\|\partial\varphi^s_{f(y)}/\partial s\| \cdot\|r_3^s-r_4^s\|\\
		\geq& (z_2^c-f(z_1)^c)+(f(z_1)^c-r_3^c)-\eta_1\cdot \|f(z_1)^s-z_2^s\| \\
		\geq& (1-3\eta_1^2)\cdot(z_2^c-f(z_1)^c)+(f(z_1)^c-r_3^c).
		\end{align*}
		Meanwhile, in the unstable direction, it satisfies
		\begin{align*}
		\|z_2^u-r_4^u\|\geq & -\|z_2^u-f(z_1)^u\|+\|f(z_1)^u-r_3^u\|-\|r_3^u-r_4^u\| \\
		\geq &-(3\eta_1)\cdot(z_2^c-f(z_1)^c)+\|f(z_1)^u-r_3^u\| -\eta_1\cdot\|f(z_1)^s-z_2^s\| \\
		\geq &-(3\eta_1+3\eta_1^2)\cdot(z_2^c-f(z_1)^c)+\|f(z_1)^u-r_3^u\|.
		\end{align*}
		Thus we have
		\begin{align*}
		-\theta\eta_1\cdot\frac{1-\eta_1}{1+\eta_1}\cdot
		\|z_2^u-r_4^u\| 
		\leq & 
		-\theta\eta_1\cdot\frac{1-\eta_1}{1+\eta_1}\cdot\left[ -(3\eta_1+3\eta_1^2)\cdot\left(z_2^c-f(z_1)^c\right)+\|f(z_1)^u-r_3^u\| \right]  \\
		\leq & \eta_1\cdot\left(z_2^c-f(z_1)^c\right)
		-\theta\eta_1\cdot\frac{1-\eta_1}{1+\eta_1}\|f(z_1)^u-r_3^u\|
		\end{align*}
		Notice that $1-3\eta_1^2>\eta_1$ and $z_2^c-\pi^c_{f(y)}\circ\exp_{f(y)}^{-1}\circ f\circ\exp_y(z_1)\geq0$. Thus the estimation of $z_2^c-r_4^c$ and Claim \ref{claim:r3} implies
		$$
		-\theta\eta_1\cdot\frac{1-\eta_1}{1+\eta_1}\cdot
		\|z_2^u-r_4^u\|\leq z_2^c-r_4^c,
		$$
		which proves the claim.
	\end{proof}
	
	There are two possibilities for the position of $z_2$:
	
	\vskip 2mm
	
	\noindent{\bf Case 1.} 
	$z_2^c\geq\pi^c_{f(y)}\circ\varphi^s_{f(y)}(z_2^s)$.
	
	In this case, the curve defined by the graph of $\gamma_{f(y)}:E^c_{f(y)}(\delta_4)\rightarrow E^s_{f(y)}(\delta_4)\oplus E^u_{f(y)}(\delta_4)$ that passing through $z_2$ will intersects with the set
	$$
	\left\lbrace v^s+v^u+\pi^c_{f(y)}\circ\varphi^s_{f(y)}(v^s): ~v^s+v^u\in E^s_{f(y)}(\delta_4)\oplus E^u_{f(y)}(\delta_4)\right\rbrace 
	$$
	with a unique point $w_2$. The fact that $w_2^c=\pi^c_{f(y)}\circ\varphi^s_{f(y)}(w_2^s)$ implies $z_2^c\geq w_2^c$. Thus we must have
	$$
	z_2^c\geq w_2^c-\theta\eta_1\cdot\left\|w_2^u
	  -\pi^u_{f(y)}\circ\varphi^s_{f(y)}(w_2^s)\right\|,
	$$
	which proves the lemma in this case.
	
	\vskip 2mm
	
	\noindent {\bf Case 2.} 
	$z_2^c<\pi^c_{f(y)}\circ\varphi^s_{f(y)}(z_2^s)$.
	
	In this case, we need a more delicate estimation. 
	The curve which is defined by the graph of $\gamma_{f(y)}:E^c_{f(y)}(\delta_4)\rightarrow E^s_{f(y)}(\delta_4)\oplus E^u_{f(y)}(\delta_4)$, and passing through $z_2$, intersects the set
	$$
	\left\lbrace v^s+v^u+\pi^c_{f(y)}\circ\varphi^s_{f(y)}(v^s): ~v^s+v^u\in E^s_{f(y)}(\delta_4)\oplus E^u_{f(y)}(\delta_4)\right\rbrace 
	$$
	with a unique point $w_2=w_2^s+w_2^c+w_2^u$ satisfying $w_2^c=\pi^c_{f(y)}\circ\varphi^s_{f(y)}(w_2^s)$. 
	
	In particular, since  $z_2^c<\pi^c_{f(y)}\circ\varphi^s_{f(y)}(z_2^s)$,
	we have $w_2^c>z_2^c$. To prove the lemma, we only need to show that 
	$$
	z_2^c-w_2^c\geq-\theta\eta_1\left\|w_2^u
	  -\pi^u_{f(y)}\circ\varphi^s_{f(y)}(w_2^s)\right\|.
	$$
	
	We have the following estimations:
	\begin{align*}
	  \|w_2^u-\pi^u_{f(y)}\circ\varphi^s_{f(y)}(w_2^s)\|
	   \geq &-\|w_2^u-z_2^u\|+\|z_2^u-r_4^u\|
	      -\|r_4^u-\pi^u_{f(y)}\circ\varphi^s_{f(y)}(w_2^s)\| \\
	  \geq &-\eta_1(w_2^c-z_2^c)+\|z_2^u-r_4^u\|
	      -\eta_1\cdot\|r_4^s-w_2^s\| 
	\end{align*}
	Since $\|r_4^s-w_2^s\|=\|z_2^s-w_2^s\|\leq\eta_1(w_2^c-z_2^c)$, we have
	$$
	\|w_2^u-\pi^u_{f(y)}\circ\varphi^s_{f(y)}(w_2^s)\|~\geq~ \|z_2^u-r_4^u\|-(\eta_1+\eta_1^2)\cdot(w_2^c-z_2^c).
	$$
	So we only need to prove that
    $$
      z_2^c-w_2^c~\geq~ \frac{-\theta\eta_1}{1+\theta\eta_1^2+\theta\eta_1^3}\|z_2^u-r_4^u\|,
    $$
    which implies
    $z_2^c-w_2^c\geq  -\theta\eta_1\left( \|z_2^u-r_4^u\|-(\theta\eta_1+\theta\eta_1^2)\cdot(w_2^c-z_2^c) \right)$ and the lemma.

	On the other hand, we have
	\begin{align*}
	   z_2^c-w_2^c~\geq~ &z_2^c-r_4^c-\|r_4^c-w_2^c\|~\geq~ z_2^c-r_4^c-\eta_1\cdot\|r_4^s-w_2^s\| \\
	   ~=~ & z_2^c-r_4^c-\eta_1\cdot\|z_2^s-w_2^s\|~\geq~ z_2^c-r_4^c-\eta_1^2\cdot(w_2^c-z_2^c)
	\end{align*}
    Notice that $1-\eta_1<1-\eta_1^2$ and $1+\theta\eta_1^2+\theta\eta_1^3<1+\eta_1$. Then Claim \ref{claim:r4} implies
    \begin{align*}
       z_2^c-w_2^c~\geq~ &\frac{1}{1-\eta_1^2}\cdot(z_2^c-r_4^c)
       ~\geq~ \frac{1}{1-\eta_1^2}\cdot(-\theta\eta_1)\cdot \frac{1-\eta_1}{1+\eta_1}\cdot\|z_2^u-r_4^u\| 
       ~\geq~  \frac{-\theta\eta_1}{1+\theta\eta_1^2+\theta\eta_1^3}\|z_2^u-r_4^u\| \\
       ~\geq~ & -\theta\eta_1\cdot\left\|w_2^u-
         \pi^u_{f(y)}\circ\varphi^s_{f(y)}(w_2^s)\right\|.
    \end{align*}
    This finishes the proof of the lemma.
\end{proof}

\subsection{Moving forward by local perturbations}
\label{subsec:local}

In this subsection, we prove Theorem \ref{Thm:Local-Perturb}. We need to construct the local perturbation vector field $X$ and re-choose the $u$-dimensional plane field $F$, such that angles between $X$ and $E^c$, $F$ and $E^u$ are both sufficiently small. In this subsection, we only assume $f\in{\rm PH}^r(M)$ with ${\rm dim}E^c=1$.

\begin{proof}[{\bf Proof of Theorem \ref{Thm:Local-Perturb}}]
	From the assumption of Theorem \ref{Thm:Local-Perturb}, for $x\in\Omega(f)\setminus{\rm Per}(f)$, there exists an open neighborhood $U$ of $x$ with $E^c|_U$ is orientable, such that for every pair of points $y,f^k(y)\in U$, $Df^k(y)$ preserves the orientation of $E^c|_U$.  
	Since $U$ is open, there exists $\varepsilon_0>0$, such that $B_x(5\varepsilon_0)\subset U$.

	There are two possibilities: either $E^c$ is orientable, or $E^c$ is non-orientable.
	
	\vskip 2mm
	
	\noindent{\bf Case 1.} $E^c$ is orientable.
	
	\vskip 1mm
	
	We fix an orientation of $E^c$, and the orientation of $E^c|_U$ coincides with this orientation. 
	In this case, there are also two possibilities, either $Df$ preserves the orientation of $E^c$, or $Df$ reverses the orientation of $E^c$. If $Df$ reverses the orientation of $E^c$, then the assumption that $Df^k(y)$ preserves the orientation of $E^c|_U$ for every $y,f^k(y)\in U$ implies $k$ must be even. Thus $x\in\Omega(f^2)\setminus{\rm Per}(f^2)$. We can always apply Proposition \ref{prop:center-curve} to the point $x$.
	
	\vskip1mm
	
	Let the constants $\delta_4$ and $\eta_1$ decided in Lemma \ref{Lem:Doimination}, we shrink $\delta_4$ such that $0<\delta_4\leq\varepsilon_0$. 
	Let $X$ be a $C^{\infty}$ vector field defined on $M$, which satisfies:
		\begin{enumerate}
			\item The support of $X$ is contained in $B_x(3\varepsilon_0)$: ${\it supp}(X)\subset B_x(3\varepsilon_0)\subset U$;
			\item For every $z\in{\it supp}(X)$, $X(z)$ is positively transverse to $E^s_z\oplus E^u_z$, and the decomposition $X(z)=X^s(z)+X^c(z)+X^u(z)\in E^s_z\oplus E^c_z\oplus E^u_z$ satisfies
			$$
			\|X^s(z)\|+\|X^u(z)\|<\eta_1\|X^c(z)\|
			\leq 2\eta_1\|X\|.
			$$
			\item There exists $b_1>0$, such that $2b_1\leq\|X^c(z)\|$ for every $z\in B_x(2\varepsilon_0)$.
		\end{enumerate}

	For the vector field $X\in{\cal X}^{\infty}(M)$, let $L$ be the constant in Theorem \ref{thm:Lipschitz} associated to $X$.
	We want to show that for every $n$ large enough, there exist $\tau_n\in(-1/n,1/n)$ and $p_n\in{\rm Per}(X_{\tau_n}\circ f)$, such that $d(x,p_n)<(1+L)/n$.

	Let
	\begin{align}\label{def:eta}
	   \eta=\frac{1}{100}\cdot\min\left\lbrace  \eta_1,~\frac{b_1}{L\cdot\sup_{z\in M}\|Df(z)\|+\|X\|} \right\rbrace  \in\left(0,10^{-3}\right),
	\end{align}
	and $F$ be a $C^{\infty}$ $u$-dimensional plane field on $M$ which satisfying 
	\begin{align}\label{def:F}
	\measuredangle(F,E^u)=\max_{x\in M}\max_{v\in F_x,\|v\|=1}d_{T_xM}(v,E^u_x)<\eta.
	\end{align}
	Let $\delta_2=\delta_2(M,f,F,\eta)$ be the constant in Lemma \ref{lem:su-disk}. We still denote 
	$D^{su}_z(\delta_2)=\bigcup_{y\in{\cal F}^s_z(\delta_2)}D^u_y(\delta_2) =\bigcup_{y\in{\cal F}^s_z(\delta_2)}\exp_y(F_y(\delta_2))$.

	For every family of complete periodic center curves $\theta:\bigsqcup_{i=0}^{k-1}\RR_i\rightarrow M$ and the fiber dynamics $\theta^*f$, we consider the invariant section $\sigma_{\tau}$ of $F_{\tau}=\tilde{X}_{\tau}\circ\theta^*f$. 
	Theorem \ref{thm:permanence} shows that invariant section $\sigma_{\tau}$ converge to the zero-section in $C^1$-topology as $\tau$ tends to zero. So we have the following claim.

	\begin{claim}\label{claim:def-tau2}
		There exists $\tau_2=\tau_2(M,f,X,\eta_1)>0$, 
		such that for every family of complete periodic center curves $\theta$ and every $\tau\in(-\tau_2,\tau_2)$, the shadowing curves $\Phi\circ\sigma_{\tau}\left(\RR_i\right), i=0,\cdots,k-1$ are tangent to the $\eta_1/2$-cone field of $E^c$ everywhere.
	\end{claim}
	
	There exists $\delta_5=\delta_5(M,f,X,\eta_1)>0$, such that for every $|\tau|<\min\{\tau_2,\delta_5/100L\}$ and $t\in\RR_i$, there exists a $C^1$-map $\gamma^c:E^c_{\theta_i(t)}(\delta_5)\rightarrow E^s_{\theta_i(t)}(\delta_5)\oplus E^u_{\theta_i(t)}(\delta_5)$ satisfying $\|\partial\gamma^c/\partial c\|\leq\eta_1$, such that
	$$
	\exp_{\theta_i(t)}^{-1}\circ\Phi\circ\sigma_{\tau}\left( (t-3\delta_5,t+3\delta_5)\right)
	\cap T_{\theta_i(t)}M(\delta_5)={\rm Graph}(\gamma^c)=
	\left\lbrace v^c+\gamma^c(v^c):~v^c\in E^c_{\theta_i(t)}(\delta_5) \right\rbrace.
	$$
	This means the shadowing curve is tangent to $\eta_1$-cone field of $E^c$ in the local chart of exponential maps.

	\vskip1mm
	
	Let $0<\delta< 10^{-2}\cdot\min\{\delta_2,\delta_4,\delta_5\}$
	    be a constant, which satisfies
	    \begin{enumerate}
		    \item for every $z\in B_x(\varepsilon_0)$ and  $v\in T_zM(\delta)$, the vector field $X_z(v)=D\exp_z^{-1}(X)(v)=X^s_z(v)+X^c_z(v)+X^u_z(v)$ satisfies
		    $$
		    \|X^s_z(v)\|+\|X^u_z(v)\|\leq 3\eta_1\|X^c_z(v)\|\leq 9\eta_1\|X\|,
		    \qquad {\rm and} \qquad
		    b_1\leq\|X^c_z(v)\|\leq 3\|X\|;
		    $$
		    \item for every $z\in M$ and $v^s\in E^s_z(\delta)$, we have $\|\partial\varphi^s_z(v)/\partial s\|<\eta$, where the graph of $\varphi^s_z$ defines the local stable manifold of $z$.
	    \end{enumerate} 
	For every $n\in\NN$ satisfying $1/n<\delta/(100L\cdot\sup_{z\in M}\|Df(z)\|)$, we consider the perturbation diffeomorphism $X_{1/n}\circ f$.
	
	Let $\Delta_n>0$ be the constant in Corollary \ref{cor:su-foliation} associated to $b_1/2n$, such that if $\gamma$ is a segment tangent to $10^{-3}$-cone field of $E^c$ with length $b_1/2n$, and $x_1,x_2$ be two endpoints of $\gamma$ satisfying $\Phi^{-1}(x_i)\in\tilde{D}^{su}_{t_i}(\delta_2)$ for $i=1,2$, then we have $|t_1-t_2|>\Delta_n$ in $\RR_i$.
    
    \vskip2mm

	Now we apply Proposition \ref{prop:center-curve} to the point $x\in\Omega(f)\setminus{\rm Per}(f)$ or $\Omega(f^2)\setminus{\rm Per}(f^2)$. In either way, we have the following claim.
	
	\begin{claim}
		For every $n>0$,
		there exist a complete periodic center curves $\theta:\bigsqcup_{i=0}^k\RR_i\rightarrow M$ and corresponding induced diffeomorphism $\theta^*f:\bigsqcup_{i=0}^k\RR_i\rightarrow\bigsqcup_{i=0}^k\RR_i$ defined as Lemma \ref{lem:inducing}, such that $f\circ\theta=\theta\circ\theta^*f$. Moreover, the zero $0_0\in\RR_0$ and $y=\theta(0_0)$ satisfy 
		$$
		d(x,y)<1/n,\qquad f^k(y)=\theta(s_n), \qquad {\rm and} \qquad |s_n|<\Delta_n.
		$$
	\end{claim}
	
	If $s_n=0$, then $y$ is a periodic point and we are done. We assume $-\Delta_n<s_n<0$. The proof for $0<s_n<\Delta_n$ is the same.
	
	\begin{definition}\label{def:conjugacy}
		For every $t\in\RR_i$, let 
		$\tilde{D}^{su}_{t}(\delta_2)=
		\Phi^{-1}\left(D^{su}_{\theta(t)}(\delta_2)\right)$  and
		$$
		{\cal D}^{su}(\delta)=\left\lbrace \tilde{D}^{su}_{t}(\delta_2)\cap 
		\theta^*(F^s\oplus F^u)(\delta):
		t\in\bigsqcup_{i=0}^{k-1}\RR_i \right\rbrace 
		$$ 
		be the $su$-foliation on $\theta^*(F^s\oplus F^u)(\delta)$ defined as Corollary \ref{cor:su-foliation}. For every $\tau\in(-\tau_2,\tau_2)$, we define the leaf conjugacy
		$h_{\tau}:\bigsqcup_{i=0}^{k-1}\RR_i\rightarrow \bigsqcup_{i=0}^{k-1}\sigma_{\tau}(\RR_i)$
	    as 
		$$
		h_{\tau}(t)~=~
		\sigma_{\tau}(\RR_i)\cap\tilde{D}^{su}_{t}(\delta_2),
		\qquad \forall t\in\RR_i;
		$$
	\end{definition}

	For every $t\in\RR_i$, we denote 
	$\psi^{su}_{\theta(t)}:
	E^s_{\theta(t)}(\delta)\oplus E^u_{\theta(t)}(\delta)
	\rightarrow E^c_{\theta(t)}(\delta)$
	the $C^1$-map whose graph satisfies
	$$
	\left\lbrace v^{su}+\psi^{su}_{\theta(t)}(v^{su}):
	v^{su}\in E^s_{\theta(t)}(\delta)\oplus E^u_{\theta(t)}(\delta) \right\rbrace 
	~=~
	\exp_{\theta(t)}^{-1}\left(D^{su}_{\theta(t)}(\delta_2)\right) 
	\cap T_{\theta(t)}M(\delta).
	$$
	The definition and property of $D^{su}_{\theta(t)}(\delta_2)$ implies that
	\begin{itemize}
		\item $\Phi\circ h_{\tau}(t)=
		       D^{su}_{\theta(t)}(\delta_2)\cap
		     \Phi\circ\sigma_{\tau}((t-3\delta,t+3\delta))$;
		\item $\|\partial\psi^{su}_{\theta(t)}/\partial s\|<2\eta$, and $\|\partial\psi^{su}_{\theta(t)}/\partial u\|<2\eta$.
	\end{itemize}

    Let $0_i\in\RR_i$ be the zero point of $\RR_i$ for $i=0,\cdots,k-1$.
    We denote $\gamma_i:E^c_{\theta(0_i)}(\delta)\rightarrow E^s_{\theta(0_i)}(\delta)\oplus E^u_{\theta(0_i)}(\delta)$ be the $C^1$-map satisfying
    $$
    \left( \exp_{\theta(0_i)}^{-1}\circ\Phi\circ\sigma_{1/n}((0_i-3\delta,0_i+3\delta))
    \right) 
    \bigcap T_{\theta(0_i)}M(\delta)={\rm Graph}(\gamma_i)
    =\left\lbrace v^c+\gamma_i(v^c):~v^c\in E^c_{\theta(0_i)}(\delta) \right\rbrace. 
    $$
    It satisfies $\|\partial\gamma_i/\partial c\|\leq\eta_1$. The graph of $\gamma_i$ intersects the set
    $$
    \left\lbrace v^s+v^u+\pi^c_{\theta(0_i)}\circ\varphi^s_{\theta(0_i)}(v^s): ~v^s+v^u\in E^s_{\theta(0_i)}(\delta)\oplus E^u_{\theta(0_i)}(\delta)  \right\rbrace \subset T_{\theta(0_i)}M(\delta),
    $$
    with a unique point $w_i=w_i^s+w_i^c+w_i^u$.
    Here the graph of $\varphi^s_{\theta(0_i)}$ defines the local stable manifold of $\theta(0_i)$ in $T_{\theta(0_i)}M(\delta)$. Corollary \ref{cor:Lipschitz} shows that
    $$
    \|w_i^u-\pi^u_{\theta(0_i)}\circ\varphi^s_{\theta(0_i)}(w_i^s)\|\leq L/n.
    $$
    
    There exists a unique point $z_i\in\Phi\circ\sigma_{1/n}((0_i-3\delta,0_i+3\delta))$, where $\exp_{\theta(0_i)}^{-1}(z_i)=z_i^s+z_i^c+z_i^u\in T_{\theta(0_i)}M(\delta)$ satisfies
    $$
    z_i^c=w_i^c- 3\eta\cdot\|w_i^u-\pi^u_{\theta(0_i)}\circ\varphi^s_{\theta(0_i)}(w_i^s)\|,
    \qquad {\rm and} \qquad
    z_i^s+z_i^u=\gamma_i(z_i^c).
    $$
    For every $i=0,\cdots,k-1$,
    we denote $t_i\in(0_i-3\delta,0_i+3\delta)\subset\RR_i$ the point satisfying $\Phi\circ h_{1/n}(t_i)=z_i$, i.e.
    $$
    z_i=D^{su}_{\theta(t_i)}(\delta_2)\pitchfork \Phi\circ\sigma_{1/n}((0_i-3\delta,0_i+3\delta)).
    $$
	
	For the zero $0_0\in\RR_0$ and the point $h_{1/n}(0_0)$, we have the following claim.
	
	\begin{claim}\label{claim:disk}
		Let the orientation of $\sigma_{1/n}(\RR_0)$ be induced by $\RR_0$, 
		then $z_0\leq h_{1/n}(0_0)$ in $\sigma_{1/n}(\RR_0)$. This implies $t_0\leq 0_0$ in $\RR_0$.
	\end{claim}
	
	\begin{proof}[Proof of the Claim]
		We only need to show that $z_0^c\leq\psi^{su}_{\theta(0_0)}(z_0^s+z_0^u)$. Since $\|\partial \gamma_0/\partial c\|<\eta_1$, we have
		$$
		\|z_0^s-w_0^s\|\leq\eta_1\cdot3\eta\|w_0^u-\pi^u_{\theta(0_0)}\circ\varphi^s_{\theta(0_0)}(w_0^s)\|, \qquad
		\|z_0^u-w_0^u\|\leq\eta_1\cdot3\eta\|w_0^u-\pi^u_{\theta(0_0)}\circ\varphi^s_{\theta(0_0)}(w_0^u)\|. 
		$$

		From the fact that $\|\partial\psi^{su}_{\theta(t)}/\partial u\|<2\eta$, and $\psi^{su}_{\theta(0_0)}(w_0^s+\pi^u_{\theta(0_0)}\circ\varphi^s_{\theta(0_0)}(w_0^s))=\pi^c_{\theta(0_0)}\circ\varphi^s_{\theta(0_0)}(w_0^s)$, we have the estimation
		\begin{align*}
		    \psi^{su}_{\theta(0_0)}(z_0^s+z_0^u) &\geq 
		    \psi^{su}_{\theta(0_0)}\left(w_0^s+
		      \pi^u_{\theta(0_0)}\circ\varphi^s_{\theta(0_0)}(w_0^s)\right)
		     -2\eta\cdot\|z_0^s-w_0^s\|-2\eta\cdot\|z_0^u
		     -\pi^u_{\theta(0_0)}\circ\varphi^s_{\theta(0_0)}(w_0^s)\| \\
		    &\geq \pi^c_{\theta(0_0)}\circ\varphi^s_{\theta(0_0)}(w_0^s)
		    -2\eta\cdot\|z_0^s-w_0^s\|-2\eta\cdot\|z_0^u-w_0^u\|-2\eta\cdot\|w_0^u-\pi^u_{\theta(0_0)}\circ\varphi^s_{\theta(0_0)}(w_0^s)\| \\
		    &\geq w_0^c-2\eta\cdot(1+3\eta_1)\cdot \|w_0^u-\pi^u_{\theta(0_0)}\circ\varphi^s_{\theta(0_0)}(w_0^s)\|
		\end{align*}
		Since 
		$$
		\eta_1\ll1/6, \qquad {\rm and } \qquad
		z_0^c=w_0^c-3\eta\cdot \|w_0^u-\pi^u_{\theta(0_0)}\circ\varphi^s_{\theta(0_0)}(w_0^s)\|,
		$$
		This shows $\psi^{su}_{\theta(0_0)}(z_0^s+z_0^u)\geq z_0^c$.	   
	\end{proof}
	
	Now we let the orientation of $E^c_{\theta(0_i)}$ in $T_{\theta(0_i)}M(\delta)$ is induced $Df^i_{\theta(0_0)}(E^c_{\theta(0_0)})$ for every $i=1,\cdots,k-1$. With this orientation, we can see that the vector field $X_{\theta(0_i)}=D\exp_{\theta(0_i)}^{-1}(X)$ is non-negative in the center direction and satisfies the assumption in Lemma \ref{Lem:Doimination}. (Notice here this orientation may be not fit the orientation of $E^c$). Lemma \ref{Lem:Doimination} shows that 
	$$
	\pi^c_{\theta(0_{i+1})}\circ\exp_{\theta(0_{i+1})}^{-1}\circ X_{1/n}\circ f(z_i)>z_{i+1}^c, \qquad i=0,1,\cdots,k-2.
	$$
    This implies
    \begin{align}\label{local dominated}
        F_{1/n}\circ h_{1/n}(t_i)\geq h_{1/n}(t_{i+1}), 
        \qquad \forall i=0,1,\cdots,k-2.
    \end{align}
    Here the order on $\sigma(\RR_{i+1})$ is induced by the order of $\RR_{i+1}$.  Again, if $Df$ reverses the orientation of $E^c$, then the order of $\RR_i$ is reversing to the order of $E^c$ when $i$ is odd, and the period $k$ is even.
    
    Moreover, if we denote 
    $$
    \exp^{-1}_{\theta(s_n)}\left(f(z_{k-1})\right)
    =\exp^{-1}_{f^k\circ\theta(0_0)}\left(f(z_{k-1})\right)
    =f(z_{k-1})^s+f(z_{k-1})^c+f(z_{k-1})^u
    \in T_{\theta(s_n)}M(\delta),
    $$
    then Item 2 of Lemma \ref{Lem:Doimination} implies
    \begin{align*}
     f(z_{k-1})^c-\pi^c_{\theta_0(s_n)}\circ\varphi^s_{\theta_0(s_n)}(f(z_{k-1})^s)
     & \geq -3\eta\cdot\|Df(\theta(0_{k-1}))\|\cdot \|w_{k-1}^u-\pi^u_{\theta(0_{k-1})}\circ\varphi^s_{\theta(0_{k-1})}(w_{k-1}^s)\|\\
     & \geq -3\eta\cdot\frac{L}{n}\cdot\|Df(\theta(0_{k-1}))\|.
    \end{align*}
    
    Now we consider the action of the vector field $X_{\theta(s_n)}=D\exp_{\theta(s_n)}^{-1}(X)$ with time $1/n$.
    Denote 
    $$
    z_k=X_{1/n}\circ f(z_{k-1}), 
    \qquad {\rm and} \qquad
    \exp_{\theta(s_n)}^{-1}(z_k)=z_k^s+z_k^c+z_k^u\in T_{\theta(s_n)}M(\delta).
    $$ 
    We have the following claim:
    
    \begin{claim}\label{claim:last-step}
    	The point $z_k$ satisfies 
    	$$
    	z_k^c> \psi^{su}_{\theta(s_n)}(z_k^s+z_k^u)+\frac{b_1}{2n}.
    	$$
    	This implies there exists $t_k>0_0$ in $\RR_0$, such that $z_k=\Phi\circ h_{1/n}(t_k)$.
    \end{claim} 
    
    \begin{proof}[Poof of the Claim]
    	We have the following estimations:
    	\begin{align*}
    	z_k^c&=\pi^c_{\theta(s_n)}\circ X_{\theta(s_n),1/n} \circ\exp^{-1}_{\theta(s_n)}(f(z_{k-1})) \geq 
    	f(z_{k-1})^c+b_1\cdot\frac{1}{n} \\
    	& \geq \pi^c_{\theta(s_n)}\circ\varphi^s_{\theta(s_n)}(f(z_{k-1})^s)
    	-3\eta\cdot\frac{L}{n}\cdot\|Df(\theta(0_{k-1}))\|
    	+b_1\cdot\frac{1}{n} 
    	\end{align*}
    	Since $\|f(z_{k-1})^s\|\leq\|z_{k-1}^s\|\leq L/n$, we have
    	$z_k^c\geq 
    	 \left(b_1-4\eta L\cdot\|Df(\theta(0_{k-1}))\|\right)/n$.
    	
    	On the other hand, we have
    	\begin{align*}
    	   \|z_k^s+z_k^u\| &\leq \|f(z_{k-1})^s+f(z_{k-1})^u\|+\frac{9\eta_1}{n}\cdot\|X\| \\
    	   & \leq 10\sup_{z\in M}\|Df(z)\|\cdot\|z_{k-1}^s+z_{k-1}^u\|
    	   +\frac{9\eta_1}{n}\cdot\|X\| \\
    	   & \leq \frac{1}{n}\cdot \left( 10 L\sup_{z\in M}\|Df(z)\| +\|X\| \right) 
    	\end{align*}
    	This implies 
    	$$
    	\psi^{su}_{\theta(s_n)}\left(z_k^s+z_k^u\right)
    	\leq \frac{2\eta}{n}\cdot 
    	\left( 10 L\sup_{z\in M}\|Df(z)\| +\|X\| \right) 
    	$$
    	
    	Thus we have
    	\begin{align*}
    	   z_k^c &\geq \frac{1}{n}\cdot\left(b_1-4\eta L\cdot\|Df(\theta(0_{k-1}))\|\right)+
    	   \psi^{su}_{\theta(s_n)}(z_k^s+z_k^u)-\frac{2\eta}{n}\cdot \left( 10 L\sup_{z\in M}\|Df(z)\| +\|X\| \right) \\
    	   &\geq \psi^{su}_{\theta(s_n)}(z_k^s+z_k^u)+\frac{b_1}{n}-
    	   \frac{\eta}{n}\cdot\left(24L\sup_{z\in M}\|Df(z)\|+2\|X\| \right) 
    	\end{align*}
    	Recall that $\eta\leq 
    	  b_1/100\left(L\sup_{z\in M}\|Df(z)\|+\|X\|\right)$, which implies 
    	$$
    	z_k^c\geq\psi^{su}_{\theta(s_n)}(z_k^s+z_k^u)-\frac{b_1}{2n}.
    	$$
    	
    	The segment 
    	$$
    	\exp_{\theta(s_n)}\left(\left\lbrace z_k^s+z_k^u+v^c:  \psi^{su}_{\theta(s_n)}(z_k^s+z_k^u)\leq v^c\leq z_k^c \right\rbrace\right)
    	$$
    	is a center segment tangent to $10^{-3}$-cone field of $E^c$ everywhere with length at least $b_1/2n$. Thus if we denote
    	$z_k=\Phi\circ h_{1/n}(t_k)$, we must have
    	$t_k\geq s_n+\Delta_n>0_0$, which proves the claim.
    \end{proof}
     
    Finally, Claim \ref{claim:disk} and the equation \ref{local dominated} imply that, in $\sigma_{1/n}(\RR_0)$, we have
    $$
    F_{1/n}^k\circ h_{1/n}(0_0)\geq 
    F_{1/n}^k\circ h_{1/n}(t_0)\geq h_{1/n}(t_k)> h_{1/n}(0_0).
    $$
	Since $F_0^k(0_0)=s_n<0_0$ and the invariant section $\sigma_{\tau}$ vary continuously with respect to $\tau$, there exists $0<\tau_n<1/n$ such that
	$$
	F_{\tau_n}^k\circ h_{\tau_n}(0_0)=h_{\tau_n}(0_0)\in\RR_i.
	$$
	Thus the point $p_n=\Phi\circ h_{\tau_n}(0_0)$ is a periodic point of $X_{\tau_n}\circ f$ satisfying $d(x,p_n)<(L+1)/n$. This proves the theorem when $E^c$ is orientable.
	
	\vskip 3mm
	
	\noindent{\bf Case 2.} $E^c$ is non-orientable.
	
	\vskip 1mm
	
	In this case, we consider the double covering $\tilde{M}$ of $M$ such that the lifting bundle $\tilde{E}^c$ of $E^c$ is orientable. Then we make the local perturbation of the lifting diffeomorphism as Case 1. The only difference is the perturbation on $\tilde{M}$ should be symmetric that can be projected on $M$.
	
	Let $\pi:\tilde{M}\rightarrow M$ be the covering map of the double covering $\tilde{M}$. For the point $x\in\Omega(f)\setminus{\rm Per}(f)$ and its neighborhood $U$, we denote $\pi^{-1}(x)=\{x^+,x^-\}$ and $\pi^{-1}(U)=U^+\cup U^-$, where $x^+\in U^+$ and $x^-\in U^-$.
	
	We fix an orientation of $E^c|_{U}$ and an orientation of $\tilde{E}^c$, such that $D\pi:\tilde{E}^c|_{U^+}\rightarrow E^c|_U$ preserves the corresponding orientations. This implies $D\pi:\tilde{E}^c|_{U^-}\rightarrow E^c|_U$ reverses the corresponding orientations.
	Moreover, let $\tilde{f}$ be a lifting diffeomorphism of $f$ which preserves the orientation of $\tilde{E}^c$. The lifting diffeomorphism $\tilde{f}:\tilde{M}\rightarrow\tilde{M}$ is also partially hyperbolic:
	$$
	T\tilde{M}=\tilde{E}^s\oplus \tilde{E}^c\oplus \tilde{E}^c, 
	\qquad {\rm and} \qquad
	D\pi(\tilde{E}^i)=E^i, \qquad\forall i=s,c,u.
	$$
	We have the following claim.
	
	\begin{claim}\label{claim:orientation}
		The point $x^+$ is a non-wandering point of $\tilde{f}$. Moreover, for every $y^+\in U^+$, we have
		$$
		B_{\tilde{f}^i(y^+)}(\varepsilon_0)\cap B_{x^-}(3\varepsilon_0)=\emptyset,
		\qquad \forall i>0.
		$$
	\end{claim}
	
	\begin{proof}
		Since $B_x(5\varepsilon_0)\subset U$, we only need to show that $\tilde{f}^i(y^+)\notin U^-$ for every $y^+\in U^+$ and $i>0$. If there exists $y^+\in U^+$ and $i>0$, such that $\tilde{f}^i(y^+)\in U^-$, then $y=\pi(y^+)\in U$ and $f^i(y)\in U$ which satisfy $Df^i$ reverses the orientation of $E^c|_U$. This contradicts the assumption.
	\end{proof}
	
	 Let $\eta_1$ and $\delta_4$ be the constants decided by Lemma \ref{Lem:Doimination} with respect to $\tilde{M}$ and $\tilde{f}$. Shrinking $\delta_4$ if necessary, we assume that $\delta_4\leq\varepsilon_0$. 
	Let $X$ be a $C^{\infty}$ vector field defined on $M$, which satisfies:
	\begin{enumerate}
		\item The support of $X$ is contained in $B_x(3\varepsilon_0)$: ${\it supp}(X)\subset B_x(3\varepsilon_0)\subset U$;
		\item For every $z\in{\it supp}(X)$, $X(z)$ is positively transverse to $E^s_z\oplus E^u_z$ with respect to the local orientation of $E^c|_{U}$, and the decoposition $X(z)=X^s(z)+X^c(z)+X^u(z)\in E^s_z\oplus E^c_z\oplus E^u_z$ satisfies
		$$
		\|X^s(z)\|+\|X^u(z)\|\leq\eta_1\|X^c(z)\|.
		$$
		\item There exists $b_1>0$, such that $2b_1\leq\|X^c(z)\|$ for every $z\in B_x(2\varepsilon_0)$.
	\end{enumerate}
	If we consider the lifting vector field $\tilde{X}\in{\cal X}^{\infty}(\tilde{M})$, then it satisfies
	\begin{enumerate}
		\item The support of $\tilde{X}$ is contained in $B_{x^+}(3\varepsilon_0)\cup B_{x^-}(3\varepsilon_0)$. 
		\item For every $z^+\in U^+$ where $\tilde{X}(z^+)\neq0$, $\tilde{X}(z^+)$ is positively transverse to $\tilde{E}^s_{z^+}\oplus\tilde{E}^u_{z^+}$, and satisfies
		$$
		\|\tilde{X}^s(z^+)\|+\|\tilde{X}^u(z^+)\|\leq\eta_1\cdot\|\tilde{X}^c(z^+)\|. 
		$$
		Moreover, $\|X^c(z^+)\|\geq2b_1$ if $z^+\in B_{x^+}(2\varepsilon_0)$.
		\item For every $y^+\in U^+$ and every $i\geq 0$, the vector field 
		$$
		\tilde{X}_{\tilde{f}^i(y^+)}=
		D\exp_{\tilde{f}^i(y^+)}^{-1}(\tilde{X})=
		\tilde{X}^s_{\tilde{f}^i(y^+)}+\tilde{X}^c_{\tilde{f}^i(y^+)}
		   +\tilde{X}^u_{\tilde{f}^i(y^+)}
		\in\tilde{E}^s_{\tilde{f}^i(y^+)}\oplus
		   \tilde{E}^c_{\tilde{f}^i(y^+)}\oplus
		   \tilde{E}^u_{\tilde{f}^i(y^+)}
		$$ 
		satisfies
		$\tilde{X}^c_{\tilde{f}^i(y^+)}(v)\geq 0$ for every 
		$v\in T_{\tilde{f}^i(y^+)}\tilde{M}(\delta_4)$.
	\end{enumerate}
	 
	Now we repeat the proof of the case that $E^c$ is orientable for $\tilde{f}$ and $x^+$. The key fact is the third item of the property $\tilde{X}$. Every $\tilde{f}^i(y^+)=\theta(0_i)$ never enter the region $U^-$. This allows us to repeat the argument of Case 1.
	
	For every $n$ large enough, there exists $|\tau_n|<1/n$, such that $p_n^+$ is a periodic point of $\tilde{X}_{\tau_n}\circ\tilde{f}$ with $d(x^+,p_n^+)<(1+L)/n$. Here $L$ is the constant of Theorem \ref{thm:Lipschitz} with respect to $\tilde{f}$ and $\tilde{X}$. 
	
	This implies there exists $p_n=\pi(p_n^+)$ tends to $x$ and $X_{\tau_n}\circ f$ converges to $f$ in $C^r$-topology, such that $p_n\in{\rm Per}(X_{\tau_n}\circ f)$. This finishes the proof of Theorem \ref{Thm:Local-Perturb}.	
\end{proof}

\section{Topological degrees for returning maps}\label{sec:non-orientable}

In this section, we prove Theorem \ref{Thm:Topo-degree}. It shows if there exist orbit returns that reverse the local orientation of center bundle, then there exist periodic points automatically. The proof relies on the calculation of topological degrees for returning diffeomorphisms. 

Recall that $s={\rm dim}E^s$. For every $\delta>0$, we denote $\DD^s(\delta)= \left\lbrace v=(v_1,\cdots,v_s)\in\RR^s:\|v\|=\sum_{i=1}^{s}v_i^2\leq\delta \right\rbrace$, and $\DD^s=\DD^s(1)$ the unit ball in $\RR^s$.

\begin{lemma}\label{lem:fixed-point}
	Let $a<b$ be two real numbers, and $H:[a,b]\times\DD^s\longrightarrow\RR\times\RR^s$ be a continuous map with
    $$
	H(t,v)=\left(h_1(t,v),h_2(t,v)\right)\in\RR\times\RR^s, \qquad \forall(t,v)\in[a,b]\times\DD^s,   
	$$
	which satisfies
	\begin{itemize}
		\item $\|h_2(t,v)\|<1$, for every $(t,v)\in[a,b]\times\DD^s$;
		\item $h_1(a,v)>a$, and $h_1(b,v)<b$, for every $v\in\DD^s$.
	\end{itemize}
    Then $H$ has a fixed point in $[a,b]\times\DD^s$.
\end{lemma}

\begin{proof}
	By taking an affine conjugation, we only need to prove this lemma for $a=-1$ and $b=1$. The boundary $\partial([-1,1]\times\DD^s)$ is homeomorphic to the unit sphere $\SS^s\subset\RR^{s+1}$ by the  homeomorphism $Q:\partial([-1,1]\times\DD^s)\rightarrow\SS^s$ defined as
	$$
	Q(t,v)=\frac{(t,v)}{\sqrt{|t|^2+\|v\|^2}}, \qquad\forall (t,v)\in\partial([-1,1]\times\DD^s).
	$$
	
	Assume $H$ has no fixed points in $[-1,1]\times\DD^s$. We consider the continuous map
	$P:[-1,1]\times\DD^s\rightarrow\SS^s$
	defined as
	$$
	P(t,v)=\frac{(t,v)-H(t,v)}{\|(t,v)-H(t,v)\|}, \qquad\forall (t,v)\in[-1,1]\times\DD^s.
	$$
	Since $(t,v)\neq H(t,v)$, $P$ is well defined on $[-1,1]\times\DD^s$. This implies the topological degree of the map
	$$
	P|_{\partial([-1,1]\times\DD^s)}:\partial([-1,1]\times\DD^s)\rightarrow\SS^s
	$$
	is zero.
	
	For every $\mu\in[0,1]$ and $(t,v)\in\partial([-1,1]\times\DD^s)$, consider the map 
	\begin{align*}
	     &\qquad L :~[0,1]\times\partial([-1,1]\times\DD^s) \longrightarrow \partial([-1,1]\times\DD^s), \\
	     &L_{\mu}(t,v)=(t,v)-\mu\cdot H(t,v)=(t-\mu\cdot h_1(t,v), v-\mu\cdot h_2(t,v)).
	\end{align*}
	For every $\mu\in(0,1]$, it satisfies
	\begin{itemize}
		\item either $(t,v)\in[-1,1]\times\partial\DD^s$ with $\|v\|=1$, then
		    $$
		    \|L_{\mu}(t,v)\|\geq\|v-\mu\cdot h_2(t,v)\|\geq\|v\|-\mu\cdot\|h_2(t,v)\|>0;
		    $$
		\item or $(t,v)\in\{-1\}\times\DD^s$, then $\|L_{\mu}(t,v)\|\geq|-1-\mu\cdot h_1(-1,v)|>|\mu-1|\geq 0$;
		\item or $(t,v)\in\{1\}\times\DD^s$, then $\|L_{\mu}(t,v)\|\geq|1-\mu\cdot h_1(-1,v)|>|1-\mu|\geq 0$.	    
	\end{itemize}
    This implies $L_{\mu}(t,v)\neq0$ for every $\mu\in(0,1]$ and $(t,v)\in\partial([-1,1]\times\DD^s)$. When $\mu=0$, it satisfies $L_{\mu}(t,v)=(t,v)\neq0$ for every $(t,v)\in\partial([-1,1]\times\DD^s)$.
    
    So we can define a homotopy $P:[0,1]\times\partial([-1,1]\times\DD^s)\rightarrow\SS^s$ as
    $$
    P_{\mu}(t,v)=\frac{L_{\mu}(t,v)}{\|L_{\mu}(t,v)\|}, \qquad\forall(\mu,t,v)\in[0,1]\times\partial([-1,1]\times\DD^s).
    $$
    For $\mu=0$ and $\mu=1$, we have
    $$
    P_0\equiv Q, \qquad {\rm and} \qquad P_1\equiv P|_{\partial([-1,1]\times\DD^s)}.
    $$
    
    Thus $P|_{\partial([-1,1]\times\DD^s)}$ is homotopic to $Q$. This is a contradiction since the topological degree of $Q$ is one. Thus $H$ must has a fixed point in $[-1,1]\times\DD^s$.
\end{proof}

Now we can prove Theorem \ref{Thm:Topo-degree}.

\begin{proof}[\bf{Proof of Theorem \ref{Thm:Topo-degree}}]
	As we assumed, there exist $y_n\rightarrow x$ and $f^{k_n}(y_n)\rightarrow x$ with $k_n>0$, such that $Df^{k_n}(y_n)$ reverses the local orientation of $E^c$ around $x$.
	If $k_n$ is bounded, then $x$ is a periodic point and we are done. Otherwise, we must have $k_n\rightarrow+\infty$ as $n\rightarrow\infty$.

	The idea for proving this theorem is the following. We first take a center curve center at $y_n$. The union of local stable manifolds of this curve is a center-stable submanifold. The iteration of this cs-submanifold contracts the stable direction, reverses the center direction, and is close to the original cs-submanifold. Projecting by the holonomy map of ${\cal F}^u$, Lemma \ref{lem:fixed-point} shows we can find a periodic unstable leaf. There exists a periodic point in this unstable leaf.

	\begin{figure}[htbp]
		\centering
		\includegraphics[width=15cm]{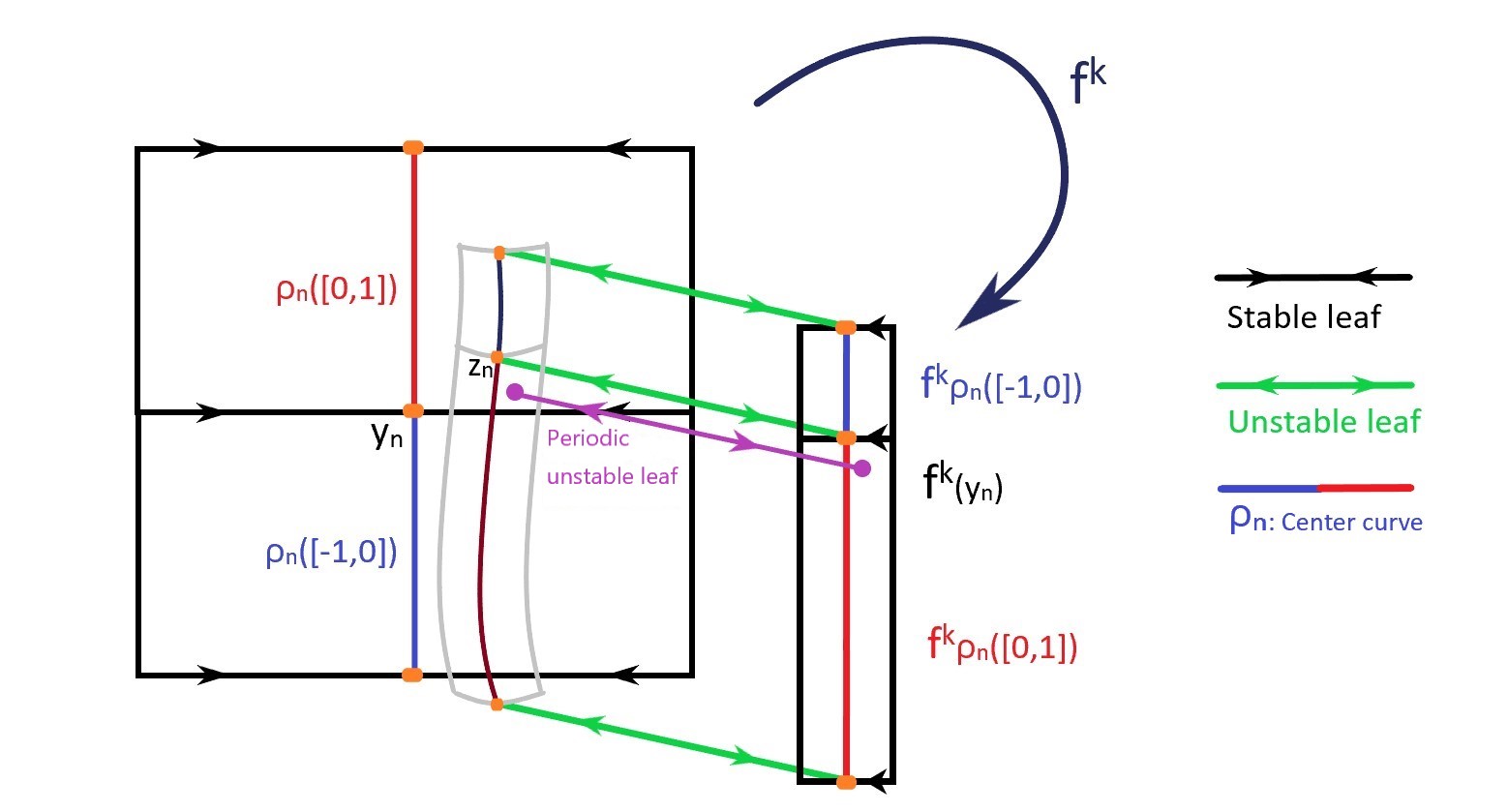}
		\caption{Return map}
	\end{figure}

	We fix a local orientation of $E^c$ around $x$.
	Let $\delta_0$ and $C_0$ be constants in Notation \ref{notation}. Then Lemma \ref{lem:local-stable}, \ref{lem:cs-manifold} and \ref{lem:local-product} all hold for constants $\delta_0$ and $C_0$.

	For every $n$, we take a $C^1$ curve $\rho_n:[-\delta_0,\delta_0]\rightarrow M$ satisfying
	\begin{itemize}
		\item $\rho_n(0)=y_n$;
		\item $\rho_n'(t)$ is the unit positive vector in $E^c(\rho_n(t))$ for every $t\in[-\delta_0,\delta_0]$.
	\end{itemize}
	Lemma \ref{lem:cs-manifold} shows that
	$$
	{\cal F}^{cs}_{\rho_n}(\delta_0)=\bigcup_{y\in\rho_n}{\cal F}^s_{y}(\delta_0)
	$$
	is an imbedded $C^1$-submanifold tangent to $E^s\oplus E^c$ everywhere.
	
	Let
	$$
	\Phi_n:{\cal F}^{cs}_{\rho_n}(\delta_0)\longrightarrow[-1,1]\times\DD^s,
	$$
	be a homeomorphism satisfying the following properties: 
	\begin{enumerate}
		\item For every $y\in{\cal F}^{cs}_{\rho_n}(\delta_0)$, there exists a unique $t\in[-1,1]$ such that $y\in{\cal F}^s_{\rho_n(\delta_0\cdot t)}(\delta_0)$, thus $\Phi_n(y)=(t,v)$ for some $v\in\DD^s$. This implies if $y_1,y_2\in\rho_n$, then $\phi_n(y_i)=(t_i,0)$ for $i=1,2$, and 
		\begin{align}\label{def:Phi_nc}
		     |t_1-t_2|\cdot\delta_0=d_{\rho_n}(y_1,y_2).
		\end{align}

		\item Item 1 of Lemma \ref{lem:local-stable} shows ${\cal F}^s_{\rho_n(\delta_0\cdot t)}(\delta_0)$ is diffeomorphic to $E^s_{\rho_n(\delta_0\cdot t)}(\delta_0)$. So we take a family of diffeomorphisms mapping ${\cal F}^s_{\rho_n(\delta_0\cdot t)}(\delta_0)$ to $\DD^s$ as
		$$
		\Phi_n|_{{\cal F}^s_{\rho_n(\delta_0\cdot t)}(\delta_0)}:
		~{\cal F}^s_{\rho_n(\delta_0\cdot t)}(\delta_0)~\longrightarrow~\{t\}\times\DD^s,
		$$
		such that
		$\Phi_n|_{{\cal F}^s_{\rho_n(\delta_0\cdot t)}(\delta_0)}$ varies $C^1$-continuously with respect to $\rho_n(\delta_0\cdot t)$,
		 and satisfies if $z_1,z_2\in{\cal F}^s_{\rho_n(\delta_0\cdot t)}(\delta_0)$ and $\Phi_n(z_i)=(t,v_i)$ for $i=1,2$, then
		 \begin{align}\label{def:Phi_ns}
		      \frac{\delta_0}{2}\cdot\|v_1-v_2\|~<~ 
		      d_{{\cal F}^s}(z_1,z_2)~<~
		      2\delta_0\cdot\|v_1-v_2\|.
		 \end{align}
	\end{enumerate}
	This homeomorphism $\Phi_n$
	gives a chart on ${\cal F}^{cs}_{\rho_n}(\delta_0)$. 
	Denote
	\begin{itemize}
		\item $\pi^c:[-1,1]\times\DD^s\rightarrow[-1,1]$ the projection to the first coordinate: $\pi^c(t,v)=t$;
		\item $\pi^s:[-1,1]\times\DD^s\rightarrow\DD^s$ the projection to the second coordinate: $\pi^s(t,v)=v$.
	\end{itemize}
	
	For $n$ large enough, ${\cal F}^u_{f^{k_n}(y_n)}(\delta_0)$ intersects ${\cal F}^{cs}_{\rho_n}(\delta_0)$ with a unique point $z_n$. If we denote
	$$
	\Phi_n(z_n)=(t_n,v_n)\in[-1,1]\times\DD^s,
	\qquad {\rm and} \qquad 
	\e_n=d\left(y_n,f^{k_n}(y_n)\right),
	$$
	then $\lim_{n\rightarrow\infty}\e_n=0$. From Lemma \ref{lem:local-product} and properties of $\Phi_n$ in (\ref{def:Phi_nc},\ref{def:Phi_ns}), we have
	$$
	d_{{\cal F}^u}\left(z_n,f^{k_n}(y_n)\right)<C_0\cdot\e_n, 
	\qquad 
	|t_n|<\frac{C_0\cdot\e_n}{\delta_0}, 
	\qquad {\rm and} \qquad
	\|v_n\|<\frac{2C_0\cdot\e_n}{\delta_0}.
	$$
	Here $C_0$ is the constant in Lemma \ref{lem:local-product} and Notation \ref{notation}.
	This implies 
	$\lim_{n\rightarrow\infty}|t_n|=\lim_{n\rightarrow\infty}\|v_n\|=0$.
	By taking the subsequence if necessary, we can assume 
	$$
	|t_n|<1/n,
	\qquad {\rm and} \qquad
	\|v_n\|<1/n.
	$$
	
	We want to define a continuous map $H_n:[a_n,b_n]\times\DD^s\rightarrow[-1,1]\times\DD^s$ for some real numbers $-1\leq a_n<b_n\leq 1$ and for every $n$ large enough which satisfies the assumption of Lemma \ref{lem:fixed-point}. 
	
	From the continuity, for every $t\in[-1,1]$ which is sufficiently close to $0$, we can define 
	 $$
	 H_n(t,0)\in[-1,1]\times\DD^s
	 =\Phi_n\left({\cal F}^{cs}_{\rho_n}(\delta_0)\right)
	 $$
	be the unique point satisfying
	 $$
	 H_n(t,0)=\Phi_n\left({\cal F}^u_{f^{k_n}(\rho_n(t))}(\delta_0)
	 \pitchfork{\cal F}^{cs}_{\rho_n}(\delta_0)\right).
	 $$
	It is clear that $H_n(0,0)=\Phi_n(z_n)=(t_n,v_n)$. 
	 
	Since $Df^{k_n}(y_n)$ reverse the local orientation of $E^c$ around $x$, we can define
	\begin{align}\label{def:a_n}
	  a_n&=\min\left\{ t\in[-1,0):~\pi^c\circ H_n(s,0)\leq\frac{2}{n}
	       ~{\rm for}~\forall t\leq s\leq 0 \right\}; \notag \\
	  b_n&=\max\left\{ t\in(0,1]:~\pi^c\circ H_n(s,0)\geq-\frac{2}{n}
	       ~{\rm for}~\forall 0\leq s\leq t \right\}. 
	\end{align}
    The curve $\rho_n$ is tangent to $E^c$, so does
	$f^{k_n}\circ\rho_n:[-\delta_0,\delta_0]\rightarrow M$. This implies the curve
	$$
	\Phi_n^{-1}\circ H_n(t,0):~[a_n,b_n]~\longrightarrow~
	       {\cal F}^{cs}_{\rho_n}(\delta_0)\subset M
	$$
	is the image of the holonomy map by local unstable foliation from the center curve $f^{k_n}\circ\rho_n([a_n,b_n])$ to ${\cal F}^{cs}_{\rho_n}(\delta_0)$:
	$$
	\Phi_n^{-1}\circ H_n(t,0)=\H^u\circ f^{k_n}\circ\rho_n(t), 
	\qquad \forall t\in[a_n,b_n].
	$$
	Here 
	$\H^u:{\cal F}^{cs}_{f^{k_n}\circ\rho_n}(\delta_0)\rightarrow
	{\cal F}^{cs}_{\rho_n}(\delta_0)$ is the local holonomy map from the center-stable submanifold ${\cal F}^{cs}_{f^{k_n}\circ\rho_n}(\delta_0)$ to 
	${\cal F}^{cs}_{\rho_n}(\delta_0)$ around $y_n$ to $z_n$. 
	Thus $\Phi_n^{-1}\circ H_n(t,0)$ is tangent to $E^c$ as $t$ varies. 
	
	From the local product structure, $\pi^c\circ H_n(t,0)$ is monotonous decreasing as $t$ increasing from $a_n$ to $b_n$. Moreover, the product structure stated in Lemma \ref{lem:local-product} shows that
	\begin{align}\label{equ:t0}
	    \|\pi^s\circ H_n(t,0)\|~\leq~ C_0\cdot\frac{2}{n}+\|v_n\|
	    ~\leq~\frac{2C_0+1}{n},
	    \qquad \forall t\in[a_n,b_n].
	\end{align}

	By taking the subsequence of $y_n$ and corresponding $k_n$ if necessary, we want $k_n$ large enough such that the diameter of 
	$f^{k_n}\big( {\cal F}^s_{\rho_n(\delta_0t)}(\delta_0) \big)$ is sufficiently small, such that for every $t\in[a_n,b_n]$ with $y=\rho_n(\delta_0\cdot t)$, and 
	$w\in{\cal F}^s_{y}(\delta_0)$ with $w=\Phi_n^{-1}(t,v)$ for some $v\in\DD^s$, the holonomy map $\H^u$ of unstable foliation is well defined form $f^{k_n}\circ{\cal F}^s_{y}(\delta_0)$ to  
	${\cal F}^{cs}_{\rho_n}(\delta_0)$ and satisfies
	\begin{align}\label{equ:tv}
	    |\pi^c\circ H_n(t,0)-\pi^c\circ H_n(t,v)| &=
	    |\pi^c\circ\Phi_n\circ\H^u(y)-\pi^c\circ\Phi_n\circ\H^u(w)|
	    <\frac{1}{n}, \notag \\
	    \|\pi^s\circ H_n(t,0)-\pi^s\circ H_n(t,v)\| &=
	    \|\pi^s\circ\Phi_n\circ\H^u(y)-\pi^s\circ\Phi_n\circ\H^u(w)\|
	    <\frac{1}{n}.
	\end{align}
    In particular, this implies for $n$ large enough, the unstable holonomy map $\H^u$ is well defined on
    $$
    \bigcup_{t\in[a_n,b_n]}f^{k_n}
    \left({\cal F}^s_{\rho_n(\delta_0\cdot t)}(\delta_0)\right).
    $$
    Thus we can define $H_n:[a_n,b_n]\times\DD^s\rightarrow[-1,1]\times\DD^s$ as
    $$
    H_n(t,v)~=~\Phi_n\circ\H^u\circ\Phi_n^{-1}(t,v)
    ~=~\Phi_n\left({\cal F}^u_{f^{k_n}\circ\Phi_n^{-1}(t,v)}(\delta_0)
    \pitchfork{\cal F}^{cs}_{\rho_n}(\delta_0)\right),
    \qquad \forall (t,v)\in[a_n,b_n]\times\DD^s.
    $$
    Moreover, for $n$ large enough, estimations in equations (\ref{equ:t0}) and (\ref{equ:tv}) guarantee that
    \begin{align}\label{equ:bound}
         |\pi^c\circ H_n(t,v)|~&\leq~ |\pi^c\circ H_n(t,0)|+\frac{1}{n}
         ~<~\frac{3}{n}~<~1, \\
         \|\pi^s\circ H_n(t,v)\|~&\leq~ \|\pi^s\circ H_n(t,0)\|+\frac{2C_0+1}{n}
         ~<~\frac{2C_0+2}{n}~<~1. \notag
    \end{align}
    This verifies the first condition of Lemma \ref{lem:fixed-point}.
	
	From the definition of $a_n$ in (\ref{def:a_n}), there are two possibilities:
	\begin{itemize}
		\item either $\pi^c\circ H_n(a_n,0)=\frac{2}{n}$;
		\item or $a_n=-1$ and $\pi^c\circ H_n(-1,0)<\frac{3}{n}$.
	\end{itemize}
	If the  first case applies, then for every $v\in\DD^s$, the estimation in (\ref{equ:tv}) shows that
	$$
	\pi^c\circ H_n(a_n,v)~>~ \pi^c\circ H_n(a_n,0)-\frac1n~=~\frac1n~>~0~>~a_n.
	$$
	Otherwise, $a_n=-1$ estimation in (\ref{equ:bound}) shows 
	$\pi^c\circ H_n(a_n,v)>a_n$ for every $v\in\DD^s$. In both cases, we have
	$$
	\pi^c\circ H_n(a_n,v)~>~a_n, \qquad \forall v\in\DD^s.
	$$
	The same argument shows that 
	$\pi^c\circ H_n(n_n,v)~<~b_n$ for every $v\in\DD^s$. 
	
	\vskip2mm
	
	Applying Lemma \ref{lem:fixed-point} to $H_n:[a_n,b_n]\times\DD^s\rightarrow[-1,1]\times\DD^s$, it has a fixed point $(c_n,v_n)\in[a_n,b_n]\times\DD^s$. Moreover, (\ref{equ:bound}) shows that
	$$
	|c_n|<\frac{3}{n}\rightarrow0, \qquad {\rm and} \qquad
	 \|v_n\|<\frac{2C_0+2}{n}\rightarrow0,
	\qquad {\rm as} \quad n\rightarrow+\infty.
	$$
	Let $r_n=\Phi_n^{-1}(c_n,u_n)$, thus we have $d(y_n,r_n)\rightarrow0$ as $n\rightarrow+\infty$. 
	Since $f^{k_n}\circ\Phi_n^{-1}\big([a_n,b_n\times\DD^n]\big)$ is a small center-stable disk satisfies 
	$$
	\Phi_n\circ\H^u\left( f^{k_n}\circ\Phi_n^{-1}\big([a_n,b_n\times\DD^n]\big) \right)
	~\subseteq~[-3/n,3/n]\times\DD^s((2C_0+2)/n),
	$$
	This implies the diameter of $f^{k_n}\circ\Phi_n^{-1}\big([a_n,b_n\times\DD^n]\big)$ tends to zero as $n\rightarrow0$. Thus we have
	$$
	d(y_n,f^{k_n}(r_n))~\leq~d(y_n,f^{k_n}(y_n))+d(f^{k_n}(y_n),f^{k_n}(r_n))
    ~\rightarrow~0, \qquad {\rm as} \quad n\rightarrow+\infty.
	$$
	
	The fact that $f^{k_n}(r_n)\in{\cal F}^u_{r_n}(\delta_0)$ implies there exists a periodic point $p_n\in{\cal F}^u_{r_n}(\delta)$ with period $k_n$. Moreover, the fact that $\lim_{n\rightarrow+\infty} d(r_n,f^{k_n}(r_n))=0$ implies $\lim_{n\rightarrow+\infty} d(r_n,p_n)=0$. So we have
	$$
	\lim_{n\rightarrow+\infty} d(x,p_n)\leq \lim_{n\rightarrow+\infty} d(x,y_n)+\lim_{n\rightarrow+\infty} d(y_n,r_n)+
	\lim_{n\rightarrow+\infty} d(r_n,p_n) =0.
	$$	
	This finishes the proof of Theorem \ref{Thm:Topo-degree}.
\end{proof}

\section{Divergence-free vector fields}\label{sec:vector-field}

In this section, we prove Theorem \ref{Thm:Div-free}, i.e. if  $f\in{\rm PH}^r(M)$ with one-dimensional oriented center bundle $E^c$ and is topologically mixing, then there exists a divergence-free vector field $X\in{\cal X}^{\infty}(M)$ which is positively transverse to $E^s\oplus E^u$.

Since $f^2$ is also topologically mixing, and the partially hyperbolic splitting of $f^2$ coincides with the splitting of $f$, we only need to prove Theorem \ref{Thm:Div-free} for $f^2$. So we assume $Df$ preserves the orientation of $E^c$.

\begin{definition}\label{def:positive-transv}
	Let $\gamma:[a,b]({\rm or}~\SS^1)\rightarrow M$ be a $C^1$-curve. We say $\gamma$ is positively transverse to $E^s\oplus E^u$ if $\gamma'(t)\in T_{\gamma(t)}M$ is positively transversely to $E^s\oplus E^u$ at $\gamma(t)$ for every $t\in[a,b]({\rm or}~\SS^1)$.
\end{definition}

\begin{remark}
	If $\gamma:[a,b]({\rm or}~\SS^1)\rightarrow M$ is a $C^1$-curve positively transverse to $E^s\oplus E^u$, and $\theta:[a,b]({\rm or}~\SS^1)\rightarrow M$ is $C^1$-close to $\gamma$. Then $\theta$ is also positively transverse to $E^s\oplus E^u$.
\end{remark}

\begin{lemma}\label{lem:transverse-segment}
	Let $f\in{\rm PH}^r(M)$ with one-dimensional oriented center bundle $E^c$, and $Df$ preserves the orientation of $E^c$. If $\gamma:[a,b]({\rm or}~\SS^1)\rightarrow M$ is a $C^1$-curve positively transverse to $E^s\oplus E^u$, then for any fixed $k\in\ZZ$,
	$$
	f^k\circ\gamma:[a,b]({\rm or}~\SS^1)\longrightarrow M
	$$
	is a $C^1$-curve positively transverse to $E^s\oplus E^u$.
\end{lemma}

\begin{proof}
	This lemma is the direct consequence of $E^s\oplus E^c\oplus E^u$ is a $Df$-invariant splitting, and $Df$ preserves the orientation of $E^c$.
\end{proof}

\begin{proposition}\label{prop:transverse-cycle}
	Let $f\in{\rm PH}^r(M)$ with one-dimensional oriented center bundle $E^c$, and $Df$ preserves the orientation of $E^c$. If $f$ is topologically mixing, then for every $x\in M$, there exists a $C^{\infty}$ imbedded circle $\gamma_x:\SS^1\rightarrow M$, such that $\gamma_x$ is positively transverse to $E^s\oplus E^u$ and $x\in\gamma_x(\SS^1)$.
\end{proposition}

\begin{proof}
	Since $f$ is partially hyperbolic and topologically mixing, the set $\Omega(f)\setminus{\rm Per}(f)$ is dense in $M$. 
	Proposition \ref{prop:center-curve} shows that, for every $x\in M$, there exists a complete curve $\theta:\RR\rightarrow M$, such that  $f^k(\theta(\RR))=\theta(\RR)$ with some $k>0$, and $\theta(0)$ can be arbitrarily close to $x$. Moreover, $\theta'(t)$ is the positive unit vector in $E^c_{\theta(t)}$ for every $t\in\RR$. 
	
	Fix a small constant $\delta>0$, we require that $\theta(0)$ is close to $x$ enough, such that for any point $y\in B(\theta(-1),\delta)$, there exists a $C^1$ curve
	$\sigma:[-1,1]\rightarrow M$ satisfying
	\begin{itemize}
		\item $\sigma(-1)=y$, $\sigma(0)=x$, and $\sigma(1)=\theta(1)$;
		\item $\sigma'(-1)$ is the unit positive vector in $E^c_{\sigma(-1)}=E^c_y$;
		\item $\sigma'(1)=\theta'(1)$ is the unit positive vector in $E^c_{\sigma(1)}=E^c_{\theta(1)}$;
		\item $\sigma$ is positively transverse to $E^s\oplus E^u$.
	\end{itemize}
	
	Take $z_0\in M$ be a limit point of $\theta$, i.e. there exists $t_n\rightarrow+\infty$, such that $\theta(t_n)\rightarrow z_0$. Now we fix another small constant $\eta>0$, such that for every $n$ large enough, and every point $z\in B(z_0,\eta)$, there exists a $C^1$ curve $\rho_n:[0,1]\rightarrow M$ satisfying
	\begin{itemize}
		\item $\rho_n(0)=\theta(t_n-1)$, and $\rho_n(1)=z$;
		\item $\rho_n'(0)$ is the unit positive vector in $E^c_{\theta(t_n-1)}$, and $\rho_n'(1)$ is the unit positive vector in $E^c(z)$;
		\item $\rho_n$ is positively transverse to $E^s\oplus E^u$.
	\end{itemize}
	
	Since $f$ is topologically mixing, there exists $N_0\in\NN$, such that for every $m>N_0$, we have 
	$$
	f^m(B(z_0,\eta))\cap B(\theta(-1),\delta)\neq\emptyset.
	$$
	Take some $l\in\NN$, such that $k\cdot l>N_0$. Denote $y\in B(\theta(-1),\delta)$ and $z\in B(z_0,\eta)$ such that
	$$
	f^{k\cdot l}(z)=y.
	$$
	Then we can take the corresponding curves $\sigma$ and $\rho_n$ for $n$ large enough as above.
	
	The fact that $f^k(\theta(\RR))=\theta(\RR)$ implies there exists a sequence of real numbers $\{s_n\}_n$, such that
	$$
	\theta(s_n)=f^{k\cdot l}\circ\theta(t_n-1), \qquad\forall n\in\NN.
	$$
	Since $t_n\rightarrow+\infty$ as $n\rightarrow\infty$, we also have $s_n\rightarrow+\infty$ as $n\rightarrow\infty$.
	
	So for some $n_0$ large enough, we have $s_{n_0}>1$. The $C^1$-curve $f^{k\cdot l}\circ \rho_{n_0}:[0,1]\rightarrow M$ satisfies
	\begin{itemize}
		\item $f^{k\cdot l}\circ \rho_{n_0}(0)=\theta(s_{n_0})$, and $f^{k\cdot l}\circ \rho_{n_0}(1)=y$;
		\item $[f^{k\cdot l}\circ \rho_{n_0}]'(0)$ is a positive vector in $E^c_{\theta(s_{n_0})}$, and $[f^{k\cdot l}\circ \rho_{n_0}]'(1)$ is a positive vector in $E^c_y$;
		\item $f^{k\cdot l}\circ \rho_{n_0}$ is positively transverse to $E^s\oplus E^u$ by Lemma \ref{lem:transverse-segment}.
	\end{itemize}
	Reparameterizing $f^{k\cdot l}\circ \rho_{n_0}$, we can have a $C^1$-curve $\varrho:[0,1]\rightarrow M$ satisfies
	\begin{itemize}
		\item $\varrho(0)=\theta(s_{n_0})$, and $\varrho(1)=y$;
		\item $\varrho'(0)$ is the unit positive vector in $E^c_{\theta(s_{n_0})}$, and $\varrho'(1)$ is the unit positive vector in $E^c_y$;
		\item $\varrho$ is positively transverse to $E^s\oplus E^u$.
	\end{itemize}
	
	\begin{figure}[htbp]
		\centering
		\includegraphics[width=15cm]{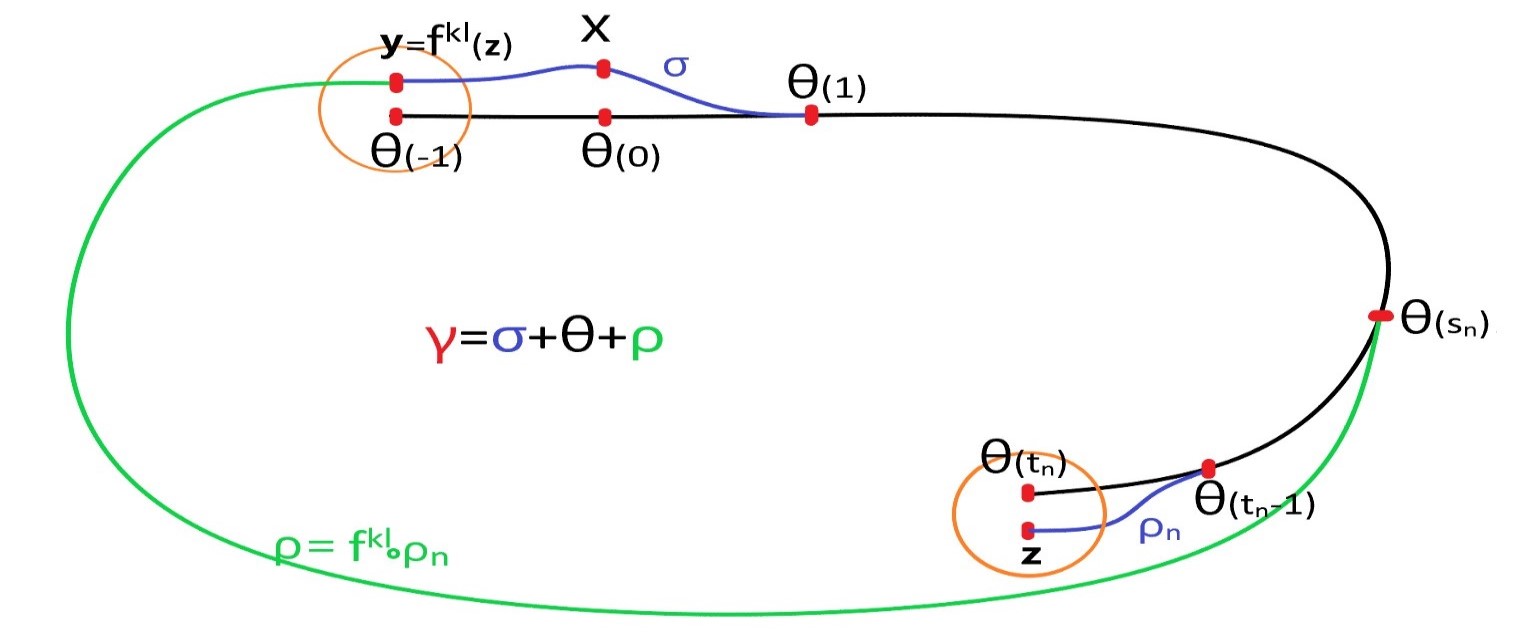}
		\caption{Positively Transverse Cycle}
	\end{figure}
	
	Now we define the curve $\gamma:[-1,s_{n_0}+1]\rightarrow M$ as following:
	$$
	\gamma(t)=\left\{
	\begin{array}{lll}
	\sigma(t), \qquad & t\in[-1,1]; \\
	\theta(t), \qquad & t\in[1,s_{n_0}]; \\
	\varrho(t-s_{n_0}), \qquad & t\in[s_{n_0},s_{n_0}+1].
	\end{array}
	\right.
	$$
	We can see that $\gamma(0)=\sigma(0)=x$. Since 
	\begin{itemize}
		\item $\sigma(1)=\theta(1)$, and $\sigma'(1)=\theta'(1)$ is the unit positive vector in $E^c_{\theta(1)}$;
		\item $\theta(s_{n_0})=\varrho(0)$, and $\theta'(s_{n_0})=\varrho'(0)$ is the unit positive vector in $E^c_{\theta(s_{n_0})}$,
	\end{itemize}
	$\gamma$ is a $C^1$ curve on $M$. Moreover, since
	\begin{itemize}
		\item $\sigma(-1)=\varrho(1)=y$, and $\sigma'(-1)=\varrho'(1)$ is the unit positive vector in $E^c_y$,
	\end{itemize}
	$\gamma$ is the image of a $C^1$-map from $\SS^1$ to $M$ by identifying $-1$ and $s_{n_0}+1$. In particular, $\gamma$ is positively transverse to $E^s\oplus E^u$.
	
	Notice that $\gamma$ may has self-intersections. However, since $d={\rm dim}M\geq3$, by small $C^1$-perturbations while keeping $x$ in the curve, we can get a $C^\infty$ imbedded circle $\gamma_x:\SS^1\rightarrow M$, such that $\gamma_x$ is positively transverse to $E^s\oplus E^u$ and $x\in\gamma_x(\SS^1)$.
\end{proof}

\begin{lemma}\label{lem:vector-field}
	Let $f\in{\rm PH}^r(M)$ with one-dimensional oriented center bundle $E^c$, and $Df$ preserves the orientation of $E^c$. If $f$ is topologically mixing, then for every $x\in M$, there exsits a divergence-free vector field $Y_x\in{\cal X}^{\infty}(M)$, such that $Y_x(x)\neq 0$, and $Y_x(y)$ is positively transverse to $E^s\oplus E^u$ for every $y\in M$ where $Y_x(y)\neq 0$.
\end{lemma}

\begin{proof}
	From Proposition \ref{prop:transverse-cycle}, there exists a $C^{\infty}$ imbedded circle $\gamma_x:\SS^1\rightarrow M$, such that $\gamma_x$ is positively transverse to $E^s\oplus E^u$ and $x\in\gamma_x(\SS^1)$. Since the isotopy class of diffeomorphisms on $\DD^{d-1}$ has two elements, one is the identity one, the other is the orientation reversing one.
	Considering a small tubular neighborhood of $\gamma_x$, there are two possibilities:
	
	\vskip1mm
	
	\noindent{\bf Case \uppercase\expandafter{\romannumeral1}}: 
	The tubular neighborhood of $\gamma_x$ is diffeomorphic to $\DD^{d-1}\times\SS^1$, where $\DD^{d-1}$ is the unit ball in $\RR^{d-1}$.
	
	In this case, let
	$$
	\Gamma_x:\DD^{d-1}\times\SS^1= \left\lbrace (a_1,\cdots,a_{d-1},b)\in\RR^{d-1}\times\SS^1:\sum_{i=1}^{d-1}a_i^2\leq 1 \right\rbrace \longrightarrow M,
	$$
	be a $C^\infty$ imbedding of $\DD^{d-1}\times\SS^1$ into $M$, such that $\Gamma_x|_{(0,\cdots,0)\times\SS^1}=\gamma_x$, and $\Gamma_x(0,\cdots,0,0)=x$.
	
	Since $\gamma_x$ is positively transverse to $E^s\oplus E^u$, we can assume that for every $(a_1,\cdots,a_{d-1})\in\RR^{d-1}$, $\Gamma_x\left((a_1,\cdots,a_{d-1},b)\times\SS^1\right)$ is positively transverse to $E^s\oplus E^u$ by shrinking the tubular neighborhood if necessary. This is equivalent to that 
	${\rm D}\Gamma_x\left(\frac{\partial}{\partial b}|_{(a_1,\cdots,a_{d-1},b)}\right)$
	is positively transverse to 
	$E^s_{\Gamma_x(a_1,\cdots,a_{d-1},b)}\oplus E^u_{\Gamma_x(a_1,\cdots,a_{d-1},b)}
	$
	for every $(a_1,\cdots,a_{d-1},b)\in\DD^{d-1}\times\SS^1$. 
	
	Since $\Gamma_x$ is a $C^\infty$ imbedding, there exists a $C^\infty$ positive function
	$$
	V(a_1,\cdots,a_{d-1},b):
	~\DD^{d-1}\times\SS^1\longrightarrow \RR_+,
	$$
	such that the Lebesgue measure $m$ of $M$ on $\Gamma_x(\DD^{d-1}\times\SS^1)$ satisfies
	$$
	\Gamma_x^*(m)(a_1,\cdots,a_{d-1},b)=V(a_1,\cdots,a_{d-1},b)\cdot {\rm d}a_1{\rm d}a_2\cdots{\rm d}a_{d-1}{\rm d}b.
	$$ 
	
	We define a $C^\infty$-bump function $\varphi:[0,1]\rightarrow \RR$, such that
	$$
	\varphi(t)\left\{
	\begin{array}{lll}
	=1, \qquad & t\in[0,1/3], \\
	\in(0,1), \qquad & t\in(1/3,2/3), \\
	=0, \qquad & t\in[2/3,1].
	\end{array}
	\right.
	$$

	Let $Z$ be a $C^\infty$ vector field on $\DD^{d-1}\times\SS^1$ defined as
	$$
	Z(a_1,\cdots,a_{d-1},b)=V(a_1,\cdots,a_{d-1},b)
	\cdot\varphi\left(\sum_{i=1}^{d-1}a_i^2\right)
	\cdot\frac{\partial}{\partial b}.
	$$
	From the definition, we have $Z(a_1,\cdots,a_{d-1},b)\equiv0$ if $2/3\leq \sum_{i=1}^{d-1}a_i^2\leq1$.
	
	Let $\Phi_t$ be the flow generated by $Z$ on $\DD^{d-1}\times\SS^1$. For any $(a_1,\cdots,a_{d-1},b)\in\DD^{d-1}\times\SS^1$, denote
	$$
	(a_1(t),\cdots,a_{d-1}(t),b(t))=\Phi_t(a_1,\cdots,a_{d-1},b).
	$$
	Since the vector field $Z$ is a smooth function times $\frac{\partial}{\partial b}$, we have $a_i(t)\equiv a_i$ for $i=1,\cdots,d-1$.
	
	In the coordinate $ \left\lbrace \frac{\partial}{\partial a_1},\cdots,\frac{\partial}{\partial a_{d-1}},\frac{\partial}{\partial b} \right\rbrace $, we have
	$$
	{\rm D}\Phi_t=\left(\begin{array}{cc}
	{\rm Id}_{d-1} & 0 \\
	* & \frac{V(a_1(t),\cdots,a_{d-1}(t),b(t))}{V(a_1,\cdots,a_{d-1},b)}
	\end{array}\right).
	$$
	This implies 
	$$
	{\rm det}({\rm D}\Phi_t(a_1,\cdots,a_{d-1},b))\cdot V(a_1,\cdots,a_{d-1},b)=V(a_1(t),\cdots,a_{d-1}(t),b(t)).
	$$
	That is $Z$ preserves the volume 
	$$
	\Gamma_x^*(m)(a_1,\cdots,a_{d-1},b)=V(a_1,\cdots,a_{d-1},b)\cdot {\rm d}a_1{\rm d}a_2\cdots{\rm d}a_{d-1}{\rm d}b.
	$$
	
	We define the vector field on $M$ as
	$$
	Y_x(y)=\left\{
	\begin{array}{ll}
	{\rm D}\Gamma_x^{-1}(Z), \qquad &\mbox{if $y\in\Gamma_x\left(\DD^{d-1}\times\SS^1\right)$;} \\
	0, \qquad & \mbox{if $y\notin\Gamma_x\left(\DD^{d-1}\times\SS^1\right)$.}
	\end{array}
	\right.
	$$ 
	So $Y_x$ is a $C^\infty$ divergence-free vector field on $M$. Moreover, for every point $y$ where $Y_x(y)\neq0$, $Y_x(y)$ is equal to ${\rm D}\Gamma_x^{-1}(Z)$, which is a positive number times 
	${\rm D}\Gamma_x^{-1}\left(\frac{\partial}{\partial b}\right)$. This implies $Y_x(y)$ is positively transverse to $E^s(y)\oplus E^u(y)$.
	This proves the first case.
	
	\vskip1mm
	
	\noindent{\bf Case \uppercase\expandafter{\romannumeral2}}: The tubular neighborhood of $\gamma_x$ is diffeomorphic to the twisted $\DD^{d-1}$-bundle over $\SS^1$, i.e.
	\begin{align*}
	\DD^{d-1}&\ltimes\SS^1 ~=~ \DD^{d-1}\times[0,1]/\sim  \nonumber \\ 
	&= \left\lbrace (a_1,\cdots,a_{d-1},b)\in\RR^{d-1}\times[0,1]:~ \sum_{i=1}^{d-1}a_i^2\leq1,~(a_1,a_2,\cdots,a_{d-1},b)
	\sim(-a_1,a_2,\cdots,a_{d-1},b) \right\rbrace . 
	\end{align*}

	In this case, we can define the vector field exactly the same as  $\DD^{d-1}\times\SS^1$. The only difference is the Lebesgue measure $m$ on this tubular neighborhood is a $d$-form without signature.
	
	Let $\Gamma_x:\DD^{d-1}\ltimes\SS^1\rightarrow M$ be a $C^\infty$ imbedding, such that 
	$$
	\Gamma_x|_{\{(0,\cdots,0,b):b\in\SS^1\}}=\gamma_x, 
	\qquad {\rm and} \qquad
	\Gamma_x(0,\cdots,0,0)=x.
	$$ 
	Moreover, we can assume ${\rm D}\Gamma_x\left(\frac{\partial}{\partial b}|_{(a_1,\cdots,a_{d-1},b)}\right)$ is positively transverse to 
	$
	E^s_{\Gamma_x(a_1,\cdots,a_{d-1},b)}\oplus E^u_{\Gamma_x(a_1,\cdots,a_{d-1},b)}
	$
	for every $(a_1,\cdots,a_{d-1},b)\in\DD^{d-1}\ltimes\SS^1$.
	
	There exists a $C^\infty$ positive function
	$$
	V(a_1,\cdots,a_{d-1},b):
	~\DD^{d-1}\ltimes\SS^1\longrightarrow \RR_+,
	$$
	such that the Lebesgue measure $m$ of $M$ on $\Gamma_x(\DD^{d-1}\ltimes\SS^1)$ satisfies
	$$
	\Gamma_x^*(m)(a_1,\cdots,a_{d-1},b)=V(a_1,\cdots,a_{d-1},b)\cdot |{\rm d}a_1{\rm d}a_2\cdots{\rm d}a_{d-1}{\rm d}b|.
	$$ 
	
	Let $Z$ be a $C^\infty$ vector field defined as
	$$
	Z(a_1\cdots,a_{d-1},b)= V(a_1\cdots,a_{d-1},b)
	\cdot\varphi\left(\sum_{i=1}^{d-1}a_i^2\right)
	\cdot\frac{\partial}{\partial b},
	$$
	where $\varphi$ is the bump function defined above. In particular, we have 
	$$
	\varphi(a_1,a_2\cdots,a_{d-1},b)=\varphi(-a_1,a_2\cdots,a_{d-1},b).
	$$
	So $Z$ is a well-defined $C^{\infty}$ vector field on $\DD^{d-1}\ltimes\SS^1$.
	Since the $Z$-orbit of every point is either a circle or the point itself,
	the same argument shows that the flow generated by $Z$ preserves the volume $V(a_1,\cdots,a_{d-1},b)
	|{\rm d}a_1{\rm d}a_2\cdots{\rm d}a_{d-1}{\rm d}b|$ on $\DD^{d-1}\ltimes\SS^1$.
	
	Define the vector field on $M$ as
	$$
	Y_x(y)=\left\{
	\begin{array}{ll}
	{\rm D}\Gamma_x^{-1}(Z), \qquad &\mbox{if $y\in\Gamma_x\left(\DD^{d-1}\ltimes\SS^1\right)$;} \\
	0, \qquad & \mbox{if $y\notin\Gamma_x\left(\DD^{d-1}\ltimes\SS^1\right)$.}
	\end{array}
	\right.
	$$ 
	We can see that $Y_x$ is a $C^\infty$ vector field which preserves the Lebesgue measure of $M$, and is positively transverse to $E^s\oplus E^u$ at every point where it is non-vanishing. This finishes our proof.
\end{proof}

Now we can prove Theorem \ref{Thm:Div-free}.

\begin{proof}[\bf{Proof of Theorem \ref{Thm:Div-free}}]
	Since the partially hyperbolic splittings of $f$ and $f^2$ are the same, we can assume $Df$ preserves the orientation of $E^c$.
	From Lemma \ref{lem:vector-field}, for every $x\in M$, there exists a divergence-free vector field $Y_x\in{\cal X}^{\infty}(M)$, such that $Y_x$ is positively transverse to $E^s\oplus E^u$ where it is non-vanishing. Let 
	$$
	U_x=\{y\in M: Y_x(y)\neq0\}.
	$$
	Then $U_x$ is an open set, and $x\in U_x$. 
	
	Take a finite cover $M\subseteq U_{x_1}\cup\cdots\cup U_{x_k}$, and consider the vector field
	$$
	X=Y_{x_1}+\cdots+Y_{x_k}\in{\cal X}^{\infty}(M).
	$$
	Recall that for the Lebesgue measure $m$ which is locally a d-form on $M$. The vector field $X$ preserves it if and only if $\emph{L}_Xm=0$ everywhere. And we have
	\begin{align*}
	    \emph{L}_Xm&=i_X\circ{\rm d}m+{\rm d}\circ i_Xm={\rm d}\circ i_Xm={\rm d}\circ i_{Y_{x_1}+\cdots+Y_{x_k}}m \\
	    & =\sum_{j=1}^{k}{\rm d}\circ i_{Y_{x_j}}m=\sum_{j=1}^{k}\emph{L}_{Y_{x_j}}m=0.
	\end{align*}
	So $X$ is a nowhere vanishing divergence-free vector field on $M$. 
	
	Moreover, for every $x\in M$, there exists some $j\in\{1,\cdots,k\}$, such that $Y_{x_j}(x)$ is positively transverse to $E^s\oplus E^u$ at $x$. And for any $j\neq l\in\{1,\cdots,k\}$, $Y_{x_j}(x)$ is either zero, or positively transverse to $E^s_x\oplus E^u_x$. Since $E^c$ is one dimensional, we have $\sum_{j=1}^{k}Y_{x_j}(x)$ is positively transverse to $E^s\oplus E^u$ at $x$. Thus $X$ is positively transverse to $E^s\oplus E^u$ on $M$.
\end{proof}

\noindent Shaobo Gan, School of Mathematical Sciences, Peking University, Beijing 100871, China\\
E-mail address: gansb@pku.edu.cn
\vspace{0.5cm}

\noindent Yi Shi, School of Mathematical Sciences, Peking University, Beijing 100871, China\\
E-mail address: shiyi@math.pku.edu.cn

\end{document}